\documentclass[reqno, 11pt, a4paper]{amsart}

\usepackage{
    a4wide,
    amssymb,
    bbm,
    centernot,
    mathtools,
} 
\usepackage{graphicx}
\usepackage[cal=euler, scr=boondoxo]{mathalfa}
\usepackage[colorlinks=true, linkcolor=blue,citecolor=blue]{hyperref}
\usepackage{float}

\newtheorem{theorem}{Theorem}[section]
\newtheorem{cor}[theorem]{Corollary}
\newtheorem{prop}[theorem]{Proposition}

\newtheorem{lem}[theorem]{Lemma}

\theoremstyle{definition}

\theoremstyle{remark}
\newtheorem{rem}[theorem]{Remark}

\numberwithin{equation}{section}

\newcommand{\capa}{\mathrm{Cap}}

\newcommand{\De}{\mathrm{d}}
\newcommand{\homo}{\mathrm{hom}}

\newcommand{\cA}{\ensuremath{\mathcal A}}
\newcommand{\cB}{\ensuremath{\mathcal B}}
\newcommand{\cC}{\ensuremath{\mathcal C}}
\newcommand{\cD}{\ensuremath{\mathcal D}}
\newcommand{\cE}{\ensuremath{\mathcal E}}
\newcommand{\cF}{\ensuremath{\mathcal F}}
\newcommand{\cG}{\ensuremath{\mathcal G}}
\newcommand{\cH}{\ensuremath{\mathcal H}}

\newcommand{\cK}{\ensuremath{\mathcal K}}
\newcommand{\cL}{\ensuremath{\mathcal L}}
\newcommand{\cM}{\ensuremath{\mathcal M}}

\newcommand{\cR}{\ensuremath{\mathcal R}}
\newcommand{\cS}{\ensuremath{\mathcal S}}

\newcommand{\cU}{\ensuremath{\mathcal U}}

\newcommand{\bbE}{\ensuremath{\mathbb E}}

\newcommand{\bbL}{\ensuremath{\mathbb L}}

\newcommand{\bbN}{\ensuremath{\mathbb N}}

\newcommand{\bbP}{\ensuremath{\mathbb P}}

\newcommand{\bbR}{\ensuremath{\mathbb R}}

\newcommand{\bbX}{\ensuremath{\mathbb X}}

\newcommand{\bbZ}{\ensuremath{\mathbb Z}}
\newcommand{\bfD}{\ensuremath{\mathbf D}}

\begin{document}

\definecolor{airforceblue}{RGB}{204, 0, 102}
\newenvironment{draft}
  {\par\medskip
  \color{airforceblue}%
  \medskip}

\title[Disconnection and entropic repulsion for the harmonic crystal]{Disconnection and entropic repulsion for the harmonic crystal with random conductances}



\author{Alberto Chiarini}
\address{Eindhoven University of Technology}
\curraddr{Department of Mathematics and Computer Science, 5600 MB Eindhoven}
\email{chiarini@tue.nl}
\thanks{}

\author{Maximilian Nitzschner}
\address{Courant Institute of Mathematical Sciences, New York University}
\curraddr{251 Mercer Street, 10012 New York, NY, USA}
\email{maximilian.nitzschner@cims.nyu.edu}
\thanks{}

\begin{abstract}
We study level-set percolation for the harmonic crystal on $\bbZ^d$, $d \geq 3$, with uniformly elliptic random conductances. We prove that this model undergoes a non-trivial phase transition at a critical level that is almost surely constant under the environment measure. Moreover, we study the disconnection event that the level-set of this field below a level $\alpha$ disconnects the discrete blow-up of a compact set $A \subseteq \bbR^d$ from the boundary of an enclosing box. We obtain quenched asymptotic upper and lower bounds on its probability in terms of the homogenized capacity of $A$, utilizing 
results from Neukamm, Sch\"affner and Schl\"omerkemper, see~\cite{neukamm2017stochastic}.  Furthermore, we give upper bounds on the probability that a local average of the field deviates from some profile function depending on $A$, when disconnection occurs. The upper and lower bounds concerning disconnection that we derive are plausibly matching at leading order. 
In this case, this work shows that conditioning on disconnection leads to an entropic push-down of the field. The results in this article generalize the findings of~\cite{nitzschner2018entropic} and~\cite{chiarini2019entropic} by the authors which treat the case of constant conductances. Our proofs involve novel ``solidification estimates'' for random walks, which are similar in nature to the corresponding estimates for Brownian motion derived by Sznitman and the second author in~\cite{nitzschner2017solidification}. 
\end{abstract}

\subjclass[2010]{}
\dedicatory{}
\maketitle

\section{Introduction}
\label{sec:intro}

In the present article we investigate the effect of \emph{impurities} on the percolative behavior of the level-sets of a harmonic crystal. Percolation models stemming from random interfaces have been investigated since the eighties, see~\cite{bricmont1987percolation,molchanov1983percolation}, and their study has gained considerable attention recently, in particular due to the renewed interest in models with long-range dependence. The discrete Gaussian free field or harmonic crystal constitutes a prominent example in this regard. In a physical context, the Gaussian free field with constant conductances on $\bbZ^d$, $d\geq 3$, may be interpreted as a microscopic description of the fluctuations in a \textit{homogeneous} crystal at non-zero temperature. The percolation phase transition and its behavior away from criticality have been thoroughly investigated in the previous decade, see~\cite{drewitz2018geometry,drewitz2017sign,duminil2020equality,popov2015decoupling,rodriguez2013phase,sznitman2015disconnection,sznitman2019macroscopic}, and connections to other percolation models, in particular random interlacements, have emerged, see~\cite{lupu2016loop,sznitman2012isomorphism}.

The microscopic description of \textit{inhomogeneous} crystals in a similar fashion motivates the introduction of a variant of the Gaussian free field with random conductances. To our knowledge this model first appeared in~\cite{caputo2003finite}, and was studied in~\cite{biskup2011scaling}, see also~\cite[Section 6]{biskup2011recent}. Here we prove the existence of a non-trivial, almost surely constant critical level for the percolation of the level-set of this field for all $d \geq 3$ in the case of uniformly elliptic, stationary and ergodic random conductances. Motivated by the results of~\cite{sznitman2015disconnection},~\cite{nitzschner2018entropic} and~\cite{chiarini2019entropic} we further study a certain \textit{disconnection event} in which the level-set of the Gaussian free field with random conductances below level $\alpha \in \bbR$ disconnects the discrete blow-up of a compact set $A \subseteq \bbR^d$ (with certain regularity properties) from the boundary of an enclosing box. The level $\alpha$ that we consider is such that the level-set above level $\alpha$ of the Gaussian free field with random conductances is in a strongly percolative regime, and therefore the disconnection event is atypical. We obtain quenched large deviation upper and lower bounds on the disconnection that substantially generalize the findings of~\cite{nitzschner2018entropic,sznitman2015disconnection}. Moreover, we study the effect of disconnection on the local behavior of the field, in a similar fashion as in~\cite{chiarini2019entropic}. We show that if certain critical levels coincide (an equality which is plausible but open at the moment), the disconnection event forces macroscopic averages of the Gaussian free field with random conductances to be pinned to a deterministic level given by $-(\alpha_\ast-\alpha) \mathscr{h}_A$ with high probability, where $\alpha_\ast$ is the percolation threshold of the model and $\mathscr{h}_A$ is the harmonic potential of the set $A$ for the limiting Brownian that is associated with the random conductance model. This effect has a similar flavor as the (classical) \textit{entropic repulsion}, which has been thoroughly investigated in the case of constant conductances for instance in~\cite{bolthausen2001entropic,bolthausen1995entropic,deuschel1999pathwise}.

One prominent feature of the Gaussian free field with random conductances is the characterization of its covariances by the Green function of a random walk in a random environment. It is well known (see for instance~\cite{barlow2010invariance,sidoravicius2004quenched}) that for uniformly elliptic, stationary and ergodic random conductances, the random walk fulfills a quenched functional central limit theorem: For almost every realization of the environment, the (diffusively scaled) walk converges in law to some limiting Brownian motion with a non-degenerate covariance matrix $a^\homo$. The asymptotic large deviation bounds that we obtain display a behavior involving such \textit{stochastic homogenization}: The capacity of $A$ for the limiting Brownian motion controls the large deviation bounds for the disconnection probability, which is a large-scale effect. On the other hand, this rate also depends on critical levels for the Gaussian free field with random conductances, which captures its local percolative behavior. Our proofs rely fundamentally on an analogue of \textit{solidification estimates} which were developed in~\cite{nitzschner2017solidification} in a continuum set-up for Brownian motion. Here, we instead prove solidification estimates for random walks among uniformly elliptic conductances and related capacity controls, which are of independent interest.
\vspace{\baselineskip}

We will now describe the model and the results in this article in more detail. Consider $\bbZ^d$, $d \geq 3$, as a graph with nearest-neighbor edges $\bbE_d$. This graph is then equipped with uniformly elliptic weights
which are bounded both from above and from below: To each edge $e \in \bbE_d$, we assign a conductance $\omega_e \in [\lambda,1]$, where $\lambda \in (0,1)$. We denote the set of all configurations of conductances by $\Omega_\lambda = [\lambda,1]^{\bbE_d}$, and we define for $\omega \in \Omega_\lambda$ by $\bbP^\omega$ the law on $\bbR^{\bbZ^d}$ such that
\begin{equation}\label{eq:IntroDefGFF}
\begin{minipage}{0.8\linewidth}
  under $\bbP^\omega$, the canonical field $(\varphi_x)_{x \in \bbZ^d}$ is a centered Gaussian field with covariances $\bbE^\omega[\varphi_x\varphi_y] =  g^\omega(x,y)$ for $x,y \in \bbZ^d$, 
\end{minipage}
\end{equation}
where $g^\omega(\cdot,\cdot)$ is the Green function of the (continuous-time, constant-speed) simple random walk on the weighted graph $(\bbZ^d,\bbE_d,\omega)$, see~\eqref{eq:GreenFunction}. Formally, one can view $\bbP^\omega$ as a Gibbs measure with
\begin{equation}
\begin{split}
\bbP^\omega(\mathrm{d}\varphi) \ & \text{``}\propto \text{''} \ \exp\left\{-\frac{1}{2}\cE^\omega(\varphi,\varphi) \right\} \prod_{x \in \bbZ^d} \lambda^{1}(\mathrm{d}\varphi_x), \\
 \cE^\omega(\varphi,\varphi) & = \frac{1}{2} \sum_{|x-y| = 1} \omega_{\{x,y\}} (\varphi_y-\varphi_x)^2,
\end{split}
\end{equation}
where $\lambda^1$ is the Lebesgue measure on $\bbR$ and $|\cdot|$ denotes the Euclidean distance.
 We endow $\Omega_\lambda$ with the canonical $\sigma$-algebra of cylinders $\cG$ and subsequently consider a probability measure (the environment measure) $P$ on $(\Omega_\lambda,\cG)$ that is stationary and ergodic with respect to shifts (see~\eqref{eq:ShiftsEnvironment} and below for details). The above ``energy'' $\cE^\omega(\cdot,\cdot)$  is the Dirichlet form associated with the generator $\cL^\omega_V$ acting on functions  $f : \bbZ^d \rightarrow \bbR$,
 \begin{equation}
 \label{eq:GeneratorIntro}
 \cL^\omega_V f(x) =  \sum_{y \, : \, |y-x| = 1} \omega_{\{x,y\}}(f(y) - f(x)),\qquad x\in \bbZ^d,
 \end{equation}
  of the (variable-speed) random walk among random conductances, known as the Random Conductance Model. The study of its homogenization properties has been the  object of very active research in the last three decades, see the survey~\cite{biskup2011recent} and references therein, and also~\cite{andres2013invariance,andres2018quenched,andres2015invariance,armstrong2018elliptic,bella2020quenched,biskup2020quenched,biskup2018limit,gloria2015quantification} for recent developments in the field. In particular, the homogenization of the capacity associated with $\cE^\omega$ will play an important role for the asymptotic bounds for the probability of disconnection events.
 
 \vspace{\baselineskip}
 
  The Gaussian free field with random conductances is obtained by first sampling $\omega \in \Omega_\lambda$ according to the law $P$ and then considering the law $\bbP^\omega$ as in~\eqref{eq:IntroDefGFF}. In this article we are mostly interested in \emph{quenched} results for this field, namely results holding for $P$-a.e.\ $\omega\in \Omega_\lambda$. 
We study the percolation of the upper level-sets (excursion sets) of the field $(\varphi_x)_{x\in \bbZ^d}$, namely, for $\alpha \in \bbR$, we introduce the random subset 
 \begin{equation}
 \label{eq:LevelSetIntro}
 E^{\geq \alpha} = \{ x \in \bbZ^d  \, : \,  \varphi_x \geq \alpha \},
\end{equation}  
which we refer to as the \emph{level-set} above $\alpha$.  For $\omega \in \Omega_\lambda$, we introduce the critical level
\begin{equation}
            \alpha_*^\omega = \inf\big\{\alpha\in \bbR:\, \bbP^\omega [E^{\geq \alpha}\mbox{ contains an infinite cluster}\,] = 0\big\} \in [-\infty,\infty].
        \end{equation}
As a combination of Proposition~\ref{prop:alpha_ast}, Theorem~\ref{thm:StretchedExpDecay} and Theorem~\ref{thm:alpha_bar} we show that $\alpha_\ast^\omega$ is $P$-a.s.\ constant, finite and strictly positive, and we denote the value of $\alpha^\omega_\ast$ on a set of full $P$-measure as $\alpha_\ast$. This establishes the non-trivial percolation phase transition of $E^{\geq \alpha}$. To capture the nature of the phase transition more precisely, we introduce two other critical parameters, $\alpha^\omega_{\ast\ast}$ and $\overline{\alpha}^\omega$ (see~\eqref{eq:alpha_astast_Def} and~\eqref{eq:OverlineAlphaDef} respectively), such that $\alpha > \alpha^\omega_{\ast\ast}$ describes a strongly non-percolative regime, whereas $\alpha < \overline{\alpha}^\omega$ characterizes a strongly percolative regime for $\varphi$, given the conductances $\omega$. We show that these parameters are $P$-a.s.\ constant and their values, given on a set of full $P$-measure by $\alpha_{\ast\ast}$ and $\overline{\alpha}$, fulfill $0 < \overline{\alpha} \leq \alpha_\ast  \leq \alpha_{\ast\ast} < \infty$. This extends the results of~\cite{rodriguez2013phase} and~\cite{drewitz2018geometry} to our context, and our proof utilizes the findings of the latter reference for a part of the argument. We refer to Remark~\ref{rem:OtherCritParam} for a discussion of the parameters $\alpha_{\ast\ast}$ and $\overline{\alpha}$. In the special case of i.i.d.\ conductances $\omega$, it is known that the diffusively scaled Random Conductance Model converges to an isotropic Brownian motion with a non-degenerate diffusivity constant. The same Brownian motion also appears as a diffusive limit of the variable-speed random walk with appropriately chosen \textit{constant} conductances. It is an open problem to determine whether the parameter $\alpha_\ast$ coincides with the corresponding critical parameter for the Gaussian free field with these constant conductances. For general ergodic, stationary and isotropic conductances, we argue that this is not the case. We again refer to Remark~\ref{rem:OtherCritParam} for more on this matter.

 For $\alpha > \alpha_{\ast\ast}$, we show in Theorem~\ref{thm:StretchedExpDecay} that the ``connectivity function'' $\bbP^\omega[x \stackrel{\geq \alpha}{\leftrightarrow} x+z ]$, which describes the probability that a nearest-neighbor path in $E^{\geq \alpha}$ connects $x$ to $x+z$, admits a stretched exponential bound in $|z|_\infty \wedge \frac{L^\rho}{100}$ uniformly over $x$ in a box of size $L^\rho$ centered at the origin, where $\rho > 1$, and $|z|_\infty$ denotes the sup-norm of $z$. The proof of this bound for the connectivity function in the strongly non-percolative regime relies on a quenched decoupling inequality, see Proposition~\ref{prop:QuenchedDecouplingIneq}, and a ``sprinkling procedure''. The latter allows us to dominate the long-range dependence of the field on multiple scales by slightly increasing the level, when going from one scale to the next (see~\eqref{eq:SprinklingIneq} and below for a precise statement). The decoupling inequality comes as a straightforward adaptation of~\cite{popov2015decoupling} and the sprinkling procedure is taken from~\cite{popov2015soft}, however some care is  required due to the inhomogeneity, the lack of translation-invariance of the field, and the specification of the parameter $\alpha_{\ast\ast}$. 
\vspace{\baselineskip}

We now turn to the results concerning disconnection.
Let $A \subseteq \bbR^d$ be a compact set with non-empty interior contained in the interior of
a box of side-length $2M$, $M > 0$, centered at the origin.  
One defines the discrete blow-up of $A$ and the outer boundary of the discrete blow-up of the box enclosing $A$ by
\begin{equation}
\label{eq:BlowUpBoxDef}
A_N = (NA) \cap \bbZ^d, \qquad S_N = \{ x \in \bbZ^d \, : \, |x|_\infty = \lfloor MN \rfloor\},
\end{equation}
respectively, where $\lfloor \, \cdot\, \rfloor$ denotes the integer part of a real number. Motivated by the findings of~\cite{nitzschner2018entropic} and~\cite{chiarini2019entropic} (see also~\cite{sznitman2015disconnection}) we introduce the \textit{disconnection event}
\begin{equation}
\label{eq:DefDisconnection}
\cD^\alpha_N = \Big\lbrace A_N \stackrel{\geq \alpha}{\centernot \longleftrightarrow}   S_N \Big\rbrace,
\end{equation}
in which no nearest-neighbor path in $E^{\geq \alpha}$ connects $A_N$ and $S_N$. 

In Theorem~\ref{thm:MainLowerBound}, we derive for $\alpha < \alpha_{\ast\ast}$ the quenched asymptotic lower bound
\begin{equation}
\label{eq:LowerBoundIntro}
    \liminf_N \frac{1}{N^{d-2}} \log \bbP^\omega[\cD_N^\alpha] \geq -\frac{1}{2} (\alpha_{**}- \alpha)^2 \capa^{\homo}(A),
\end{equation}
for $P$-a.e.\ $\omega \in \Omega_\lambda$, where $\capa^{\homo}(A)$ is the capacity of $A$ for a Brownian motion with covariance matrix $a^\homo$ that appears as the $P$-a.s.\ scaling limit of the random walk among random conductances, see~\eqref{eq:CapacityContinuumDef} for a precise definition.
This generalizes Theorem 2.1 of~\cite{nitzschner2018entropic} (see also Theorem 2.1 of~\cite{sznitman2015disconnection}) to our set-up, and we employ a capacity convergence via stochastic homogenization for the proof, see Proposition~\ref{prop:homo_killed}.

As far as upper bounds are concerned, we show in Theorem~\ref{thm:UpperBound} that for $\alpha < \overline{\alpha}$, it holds that
\begin{equation}
\label{eq:UpperBoundIntro}
    \limsup_{N\to\infty} \frac{1}{N^{d-2}} \log \bbP^\omega[\cD_N^\alpha] \leq -\frac{1}{2} (\overline{\alpha}- \alpha)^2 \capa^{\homo}(\mathring{A}),
\end{equation}
for $P$-a.e.\ $\omega \in \Omega_\lambda$, where $\mathring{A}$ is the interior of $A$. It is plausible that in fact $\overline{\alpha} = \alpha_{\ast\ast}$, so that~\eqref{eq:LowerBoundIntro} and~\eqref{eq:UpperBoundIntro} would provide matching asymptotic lower and upper bounds if $A$ is regular in the sense that $\capa^{\homo}(A) = \capa^{\homo}(\mathring{A})$, see also Remark~\ref{rem:on_regularity} concerning this condition. The upper bound~\eqref{eq:UpperBoundIntro} has been derived in the case of constant conductances as Theorem 3.1 of~\cite{nitzschner2018entropic} (building on Theorem 5.5 of~\cite{sznitman2015disconnection}, in which $A = [-1,1]^d$ and where the convexity of $A$ played a major role in the proof). Importantly, in the case of constant conductances, the equality of the critical parameters $\overline{\alpha} = \alpha_{\ast} = \alpha_{\ast\ast}$ has recently been established in~\cite{duminil2020equality}. 
Progress towards a proof of the corresponding equalities in the case of random conductances may come from adapting the techniques of~\cite{duminil2020equality}: On the one hand, it is still possible to decompose the Gaussian free field with random conductances into an infinite sum of independent Gaussian fields with finite range of dependence. On the other hand, the lack of translation-invariance of the field, and the nature of our parameters (which involve uniform estimates anchored at the origin, see Section~\ref{sec:GFF_RandomCond}) require some additional care.


Finally, we also investigate the macroscopic height profile of the field conditioned on the disconnection event $\mathcal{D}^\alpha_N$ in the strongly percolative regime $\alpha < \overline{\alpha}$. To this end, we introduce the random (signed) measure on $\bbR^d$,
\begin{equation}
\label{eq:EmpiricalGFFMeasureDef}
\bbX_N = \frac{1}{N^d} \sum_{x \in \bbZ^d} \varphi_x \delta_{\frac{x}{N}},
\end{equation}
and for any continuous compactly supported function $\eta : \bbR^d \rightarrow \bbR$ and a signed Radon measure $\mu$ on $\bbR^d$, we define
\begin{equation}
\langle \mu, \eta \rangle = \int \eta(x) \mu(\mathrm{d}x).
\end{equation}
Furthermore we introduce the ``profile'' function
\begin{equation}
\label{eq:LocalProfileFunction}
\mathscr{H}^\alpha_{\mathring{A}} = -(\overline{\alpha}- \alpha) \mathscr{h}_{\mathring{A}}.
\end{equation}
The function $\mathscr{h}_{\mathring{A}}$ is the harmonic potential of $\mathring{A}$, associated with the limiting Brownian motion obtained from the quenched functional central limit theorem, see~\eqref{eq:HarmonicPotContinuumDef}. 

We show in Theorem~\ref{thm:EntropicRepulsion} that for $\alpha < \overline{\alpha}$, $\Delta > 0$ and a compactly supported continuous function $\eta : \bbR^d \rightarrow \bbR$, one has the quenched asymptotic upper bound
\begin{equation}
\label{eq:EntropicRepulsionIntro}
\begin{split}
\limsup_{N \rightarrow \infty} \frac{1}{N^{d-2}} \log \, & \bbP^\omega\Big[ \big\vert \langle \bbX_N, \eta \rangle - \langle \mathscr{H}^\alpha_{\mathring{A}}, \eta \rangle \big\vert \geq \Delta ; \cD^\alpha_N \Big] \\
& \leq -\frac{1}{2}(\overline{\alpha}-\alpha)^2 \capa^{\homo}(\mathring{A}) - c_1(\Delta,\alpha,\eta),
\end{split}
\end{equation}
for $P$-a.e.\ $\omega \in \Omega_\lambda$, where $c_1(\Delta,\alpha,\eta)$ is a positive constant which depends on $\Delta, \alpha$ and $\eta$ as well as on $A$, $M$ and $d$. This result can be understood as follows: If the critical parameters $\overline{\alpha}, \alpha_\ast$ and $\alpha_{\ast\ast}$ coincide and $A$ is regular in the sense that $\capa^{\homo}(A) = \capa^{\homo}(\mathring{A})$, then a combination of~\eqref{eq:EntropicRepulsionIntro} with the quenched asymptotic lower bound~\eqref{eq:LowerBoundIntro} yields that for every $\alpha < \alpha_\ast$, 
\begin{equation}
\label{eq:EntropicRepulsionConditioned}
\lim_{N \rightarrow \infty} \bbE^\omega\big[ \big\vert \langle \bbX_N, \eta \rangle - \langle \mathscr{H}^\alpha_{A}, \eta \rangle \big\vert \wedge 1 | \cD^\alpha_N \big] = 0,
\end{equation}
for $P$-a.e.\ $\omega \in \Omega_\lambda$. In other words, conditionally on the disconnection event $\cD^\alpha_N$, the local macroscopic average of the Gaussian free field with random conductances is ``pinned'' to $\mathscr{H}^\alpha_A$ with high probability, for $P$-a.e.\ $\omega \in \Omega_\lambda$. Results of this type were established for constant conductances in~\cite{chiarini2019entropic}, where the aforementioned ``pinning'' was also shown to be uniform in $\eta$ over a certain class of bounded Lipschitz functions and a profile description (akin to~\cite{bolthausen1993critical}) was also developed. In this work, we do not aim to establish these additional properties so as not to dilute the focus.

Heuristically,~\eqref{eq:EntropicRepulsionConditioned} can be related to the optimal way for the Gaussian free field with random conductances to enforce disconnection. This shares some flavor with a capacity order large deviation principle for $\bbX_N$, investigated (for constant conductances) in~\cite{bolthausen1993critical}, and the rates in~\eqref{eq:LowerBoundIntro} and~\eqref{eq:UpperBoundIntro} correspond to the cost of observing a local shift of the field produced by the profile function $\mathscr{H}^\alpha_{A}$. We also refer to~\cite[Section 6]{li2015large}, in which a large deviation principle for the occupation time of random interlacements is used to study disconnection by ``high-density regions''. 

The analogues of the upper bounds~\eqref{eq:UpperBoundIntro} and~\eqref{eq:EntropicRepulsionIntro} for the case of constant conductances were derived in~\cite{nitzschner2018entropic} and~\cite{chiarini2019entropic} by employing the solidification estimates and related capacity bounds for Brownian motion from~\cite{nitzschner2017solidification}. In the present case, we develop similar solidification estimates for random walks among uniformly elliptic conductances, which are relevant in their own right. 

Specifically, we introduce a notion of \textit{porous interfaces} $\Sigma \subseteq \mathbb{Z}^d$ surrounding the discrete blow-up $A_N$ of a compact set $A \subseteq \bbR^d$ with non-empty interior. Roughly speaking, these porous interfaces vary over a class of ``deformations'' (felt at distance $\varepsilon \in \bbN$) and ``strength'' $\chi \in (0,1)$ of a \textit{hard interface} $S$ at distance $2^{\ell_\ast}$ (with $\ell_\ast$ a non-negative integer) from $A_N$.  

 We prove in Theorem~\ref{thm:SolidificationTheorem} a \textit{solidification estimate}, which informally can be stated as
\begin{equation}
\label{eq:SolidMainResultIntro}
 \lim_{N \rightarrow \infty} \;  \sup_{\omega\in \Omega_\lambda} \widetilde{\sup} \;  \sup_{x \in A_N} P^\omega_x[ \textnormal{Random walk never enters }\Sigma] = 0,
\end{equation}
where $P^\omega_x$ is the law of a random walk on $(\bbZ^d,\bbE_d,\omega)$ starting in $x \in \bbZ^d$, and $\widetilde{\sup}$  stands for the supremum over all porous deformations $\Sigma$ of $S$ such that $\varepsilon /2^{\ell_\ast} \leq a_N$, for a given sequence $(a_N)_{N \geq 0}$ of positive real numbers with $a_N \rightarrow 0$ as $N \rightarrow \infty$.  The proof of~\eqref{eq:SolidMainResultIntro} involves the construction of a certain \emph{resonance set} associated with $A_N$, which is hard to avoid for a random walk starting in $A_N$. The solidification estimates and related capacity controls obtained in Corollary~\ref{cor:CapacityDirichletSolidification} are instrumental for proving the asymptotic upper bounds~\eqref{eq:UpperBoundIntro} and~\eqref{eq:EntropicRepulsionIntro}. 

We briefly comment on the proofs of our results and give some intuition as well as some further directions and open problems. The most challenging part of this work concerns the upper bounds~\eqref{eq:UpperBoundIntro} and~\eqref{eq:EntropicRepulsionIntro}. We devise a variant of the coarse-graining procedure that was introduced in~\cite{nitzschner2017solidification} and also used in~\cite{chiarini2019entropic,chiarini2020entropic,nitzschner2018entropic,sznitman2019macroscopic} for the disconnection event $\cD^\alpha_N$. We then make use of certain Gaussian bounds leading to controls involving the capacity of the porous interfaces that are amenable to the solidification-type bounds described above. In a last step, we apply a $\Gamma$-convergence result from~\cite{neukamm2017stochastic}, which essentially yields the convergence of the (discrete) capacity $N^{2-d} \capa^\omega(A_N)$ to $\capa^\homo(A)$ for large $N$, for $P$-a.e.\ $\omega\in \Omega_\lambda$. Remarkably, both the solidification estimates and the Gaussian bounds only rely on  well-known quenched (killed) heat kernel or Green function estimates (see for instance~\cite{barlow2017random} for a general treatment). 
A special emphasis is therefore put on separating the parts of the proof that use solidification-type results from the homogenization features entering the picture, making the approach reasonably robust. One may hope that more refined heat kernel bounds that are available could potentially facilitate the study of the Gaussian free field in degenerate random environments (see for instance~\cite{andres2019heat,hambly2009parabolic} for relevant bounds in the case of degenerate conductances). 

As remarked earlier, there are strong links between the Gaussian free field and random interlacements, for which similar disconnection-type and other events of a large deviation nature have been studied extensively in recent years in the case of constant conductances, see~\cite{chiarini2020entropic,li2014lower,li2015large,nitzschner2017solidification,sznitman2017disconnection,sznitman2019bulk,sznitman2019macroscopic,sznitman2020excess}. The results of this work may potentially help in shedding some light on the situation of disconnection phenomena for random interlacements and random walks among random conductances and entropic repulsion for the occupation time measure, in the spirit of~\cite{chiarini2020entropic,li2017lower,li2014lower,sznitman2017disconnection}. In a different direction, we remark that the Gaussian free field with random conductances also has implications for more general gradient fields with nonconvex potentials, as studied in~\cite{biskup2011scaling}. For such fields the percolative properties are even less accessible (see~\cite{rodriguez2016decoupling} for a treatment of percolation models of this type). Finally, the study of disconnection by level-sets of the planar Gaussian free field (with suitably chosen boundary conditions on a large box) remains open also in the case of constant conductances. We refer to~\cite{ding2018chemical,ding2020percolation} for some recent results concerning level-set percolation in the planar case.
\vspace{\baselineskip}

The organization of this article is as follows. In Section~\ref{sec:NotationUsefulResults}, notation is introduced together with some known results on random walks and the Gaussian free field on the weighted graph $(\bbZ^d, \bbE_d,\omega)$, heat kernel bounds, some potential theory and the change of probability method. In Section~\ref{sec:GFF_RandomCond} we introduce the parameters $\alpha^\omega_{\ast}$, $\alpha^\omega_{\ast\ast}$ and $\overline{\alpha}^\omega$ and show that they are constant, finite and strictly positive for $P$-a.e.\ $\omega \in \Omega_\lambda$, see Proposition~\ref{prop:alpha_ast} and Theorems~\ref{thm:StretchedExpDecay} and~\ref{thm:alpha_bar}. We also show a stretched exponential decay of the connectivity function for $\alpha > \alpha_{\ast\ast}$.
Section~\ref{sec:Solidification} contains the proof of the solidification estimate~\eqref{eq:SolidMainResultIntro} and related capacity controls, see Theorem~\ref{thm:SolidificationTheorem} and Corollary~\ref{cor:CapacityDirichletSolidification}. In Section~\ref{sec:LowerBound}, we prove in Theorem~\ref{thm:MainLowerBound} the main quenched asymptotic lower bound on disconnection~\eqref{eq:LowerBoundIntro} and in Proposition~\ref{prop:homo_killed} the convergence of capacities that we employ in this bound and later. Section~\ref{sec:Quenched} contains quenched Gaussian bounds that are the foundation of the upper bounds in the next section. In Section~\ref{sec:UpperBounds} we show the main asymptotic upper bounds~\eqref{eq:UpperBoundIntro} and~\eqref{eq:EntropicRepulsionIntro} in Theorems~\ref{thm:UpperBound} and~\ref{thm:EntropicRepulsion}, using the results of Sections~\ref{sec:Solidification} and~\ref{sec:Quenched}.
\vspace{\baselineskip}

We will use the following convention concerning constants. We denote by $c$, $c'$, $\ldots$ positive constants with values that change from place to place. Numbered constants $c_1$, $c_2$, $\ldots$ are
defined at the place of their first occurrence within the text and remain fixed. All constants may implicitly depend on the dimension and on the parameter $\lambda \in (0,1)$. Dependence of constants on additional parameters will appear explicitly in the notation.
\section{Notation and some useful facts}
\label{sec:NotationUsefulResults}

In this section we introduce further notation and present some known results concerning random walks among inhomogeneous conductances, potential theory, the Gaussian free field, and an entropy inequality that will be useful in the derivation of lower bounds on the probability of disconnection in Section~\ref{sec:LowerBound}. We tacitly assume throughout the article that $d \geq 3$. 

Let us start with some elementary notation. We let $\bbN  = \{0,1,2,...\}$ stand for the set of natural numbers. For real numbers $s, t$, we let $s \wedge t$ and $s \vee t$ stand for the minimum and maximum of $s$ and $t$, respectively, and we denote by $\lfloor s \rfloor$ the integer part of $s$, when $s$ is non-negative. We denote by $| \cdot |$, $| \cdot |_1$ and $| \cdot |_\infty$ the Euclidean, $\ell^1$- and $\ell^\infty$-norms on $\bbR^d$, respectively. For $x \in \bbZ^d$ and $r \geq 0$, we write $B(x,r) = \{y \in \bbZ^d \, : \, |x-y|_\infty \leq r\} \subseteq \bbZ^d$ and $B^{(1)}(x,r) = \{y \in \bbZ^d \, : \, |x-y|_1 \leq r\} \subseteq \bbZ^d\}$ for the (closed) $\ell^\infty$- and $\ell^1$-balls of radius $r \geq 0$ and center $x \in \bbZ^d$, respectively. If $x,y \in \bbZ^d$ fulfill $|x - y| = 1$, we call them neighbors and write $x \sim y$. A function $\pi : \{ 0, ..., N \} \rightarrow \bbZ^d$ is called a nearest-neighbor path (of length $N \geq 1$) if $\pi(i) \sim \pi(i+1)$ for all $0 \leq i \leq N -1$. For $K \subseteq \bbZ^d$, we let $|K|$ stand for the cardinality of $K$, and we write $K \subset \subset \bbZ^d$ if $|K| < \infty$. Moreover, we write $\partial K = \{ y \in \bbZ^d \setminus K \, : \, y \sim x \text{ for some } x \in K \}$ for the external boundary of $K$, and $\partial_{\text{in}}K = \{y \in K \, : \, y \sim x \text{ for some } x \in \bbZ^d \setminus K \}$ for the internal boundary of $K$. For $K,L \subseteq \bbZ^d$, we  let $d_\infty(K,L)$ stand for the $\ell^\infty$-distance between $K$ and $L$. For a set $D \subseteq \bbR^d$ we denote by $D_N = (ND) \cap \bbZ^d$ the discrete blow-up of $D$, and for $\delta > 0$, we let $D^\delta$ be the closed $\delta$-neighborhood of $D$. For $\eta, \eta' \in \bbR^{\bbZ^d}$ we write $\eta \leq \eta'$ if $\eta_x \leq \eta'_x$ for all $x \in \bbZ^d$. A function $f : \bbR^{\bbZ^d} \rightarrow \bbR$ is called increasing if $\eta \leq \eta'$ implies $f(\eta) \leq f(\eta')$, and decreasing if the function $-f$ is increasing. Moreover, we say that $f$ is supported on $K \subseteq \bbZ^d$ if $f(\eta) = f(\eta')$ whenever $(\eta_x)_{x \in K} = (\eta'_x)_{x \in K}$. For functions $u, v : \bbZ^d \rightarrow \bbR$, we  routinely write $\langle u,v \rangle_{\bbZ^d} = \sum_{z \in \bbZ^d} u(z)v(z)$, if $|uv|$ is summable.

\vspace{\baselineskip}

We consider the integer lattice $\bbZ^d$ as a graph with edge set $\bbE_d = \{\{x,y\}: \,x,y\in \bbZ^d,\, x\sim y \}$. As in the introduction, we fix a parameter $\lambda\in (0,1)$ and consider $\Omega_\lambda = [\lambda,1]^{\bbE_d}$, the set of weight configurations on the graph $(\bbZ^d, \bbE_d)$. For any $\omega\in \Omega_\lambda$ we let 
\begin{equation}
\label{eq:UniformEll}
     \omega_{x,y} = \omega_{\{x,y\}} = \omega_{y,x} (\in [\lambda, 1])
\end{equation}
stand for the conductance along the undirected edge $\{x,y\} \in \bbE_d$. For $\omega \in \Omega_\lambda$, we also define
\begin{equation}
\label{eq:reversibMeasure}
\omega_x = \sum_{z \, : \, z \sim x} \omega_{x,z}, \qquad x \in \bbZ^d.
\end{equation}
Given $\omega \in \Omega_\lambda$, one can define the operators $\cL_V^\omega$ and $\cL^\omega_C$ acting on functions  $f : \bbZ^d \rightarrow \bbR$ by
\begin{align}
\label{eq:GeneratorVSRW}
    \cL^\omega_V f(x) &=  \sum_{y \, : \, y \sim x} \omega_{x,y}(f(y) - f(x)),\qquad x\in \bbZ^d; \\
\label{eq:GeneratorCSRW}
    \cL^\omega_C f(x)& =  \sum_{y \, : \, y \sim x} \frac{\omega_{x,y}}{\omega_x}(f(y) - f(x)),\qquad x\in \bbZ^d,
\end{align}
which are the generators of the variable-speed random walk or the constant-speed random walk on $(\bbZ^d, \bbE_d, \omega)$, respectively. We say that $f : \bbZ^d \rightarrow \bbR$ is $\omega$-harmonic in $U \subseteq \bbZ^d$ if for all $x \in U$, one has $\cL_C^\omega f(x) = 0$. 
\vspace{\baselineskip}

We now define the continuous-time, constant-speed simple random walk on the weighted graph $(\bbZ^d, \bbE_d, \omega)$ and discuss some potential theory associated with it. Given $\omega \in \Omega_\lambda$, we denote by $P^\omega_x$ a probability measure on the space $(\bbZ^d \times  (0,\infty))^{\bbN}$ under which $(Y_n)_{n \geq 0}$ has the law of a simple random walk among conductances $\omega$, starting from $x\in \bbZ^d$, governed by the transition probabilities
\begin{equation}
\label{eq:ContTimeRW}
r^\omega(y,z) = \begin{cases}
\frac{\omega_{y,z}}{\omega_z}, & \text{ if } y \sim z, \\
0, & \text{ otherwise,}
\end{cases}
\end{equation} 
and $(\zeta_n)_{n \geq 0}$ is a sequence of i.i.d.~exponential variables with parameter $1$, where $(Y_n, \zeta_n)_{n \geq 0}$ denote the canonical $\bbZ^d \times (0,\infty)$-valued coordinates on $(\bbZ^d \times  (0,\infty))^{\bbN}$. We attach to $w\in (\bbZ^d \times  (0,\infty))^{\bbN}$ a continuous-time trajectory $X_t(w)$, $t \geq 0$, by setting
\begin{equation}
X_t(w) = Y_k(w), \text{ for $t \geq 0$, when } \sum_{i = 0}^{k-1} \zeta_i(w) \leq t < \sum_{i = 0}^k \zeta_i(w)
\end{equation}
(if $k = 0$, the sum on the left is understood as $0$). Thus, under $P^\omega_x$, the trajectory $X_\cdot$ is a continuous-time Markov chain with generator~\eqref{eq:GeneratorCSRW}. The expectation corresponding to $P^\omega_x$ is denoted $E^\omega_x$. The variable-speed random walk $\overline{X}_\cdot$ is a Markov chain with generator~\eqref{eq:GeneratorVSRW}, and can be obtained from the constant-speed random walk $X$ by a time-change (see, e.g.,~\cite[Section 1.6]{sznitman2012topics}). Unless otherwise specified, we use random walk as a shorthand for the continuous-time, constant-speed random walk.

Given $U \subseteq \bbZ^d$, we introduce stopping times (with respect to the canonical filtration $(\cF_t)_{t \geq 0}$ generated by $(X_t)_{t \geq 0}$) $H_U = \inf\{ t \geq 0 \, : \, X_t \in U\}$, $\widetilde{H}_U = \inf\{ t \geq \zeta_1 \, : \, X_t \in U\}$, $T_U = \inf \{ t \geq 0 \, : \, X_t \notin U \}$, which are the entrance, hitting and exit times of $U$. We then introduce the heat kernel $q^\omega_t$ and the killed heat kernel (upon leaving $U \subseteq \bbZ^d$) $q^{\omega}_{t,U}$ via
\begin{align}
q^\omega_t(x,y) & = \frac{P^\omega_x[X_t = y]}{\omega_y},  &t \geq 0,\, x,y \in \bbZ^d, \text{ and} \\
\label{eq:KilledHeatKernel}
q^{\omega}_{t,U}(x,y) & = \frac{P^\omega_x[X_t = y, T_U > t]}{\omega_y},  &t \geq 0,\, x,y \in \bbZ^d.
\end{align}
We denote the Green function of the walk by $g^\omega(\cdot, \cdot)$ and the Green function of the walk killed upon leaving $U$ by $g_U^\omega(\cdot, \cdot)$:
\begin{align}
\label{eq:GreenFunction}
g^\omega(x,y) & = E^\omega_x\Big[ \int_0^\infty \mathbbm{1}_{\{X_t = y \}}\, \mathrm{d}t \Big] / \omega_y = \int_0^\infty q^\omega_t(x,y)\, \mathrm{d}t, \text{ and} \\
\label{eq:KilledGreenFunction}
g_U^\omega(x,y) & = E^\omega_x\Big[ \int_0^{T_U} \mathbbm{1}_{\{X_t = y \}}\, \mathrm{d}t \Big] / \omega_y = \int_0^\infty q^{\omega}_{t,U}(x,y)\, \mathrm{d}t.
\end{align}
By the ellipticity assumption on $\omega$, it is known that $q^\omega_t$, $q^\omega_{t,U}$, $g^\omega$ and $g^\omega_U$ are finite and symmetric (for every $t \geq 0$), and that $q^\omega_{t,U}(\cdot, \cdot)$ and $g^\omega_U(\cdot, \cdot)$ both vanish if one of their arguments lies in the complement of $U$.

Moreover, as an application of the strong Markov property at the exit time of $U$, one has 
\begin{equation}
\label{eq:DecompositionGreenKilledGreen}
g^\omega(x,y) = g^\omega_U(x,y) + E^\omega_x[T_U < \infty, g^\omega(X_{T_U},y)], \qquad x,y \in \bbZ^d,
\end{equation}
see, e.g., Proposition 1.6 in~\cite{sznitman2012topics}. Applying this identity to $U = \{x\}$ (and symmetry) readily shows that 
\begin{equation}
\label{eq:harmonicityOfGreen}
g^\omega(y,\cdot) \text{ is $\omega$-harmonic in } \bbZ^d \setminus \{y\}, \qquad y \in \bbZ^d.
\end{equation}

 It is known, see for instance Theorem 6.28 in~\cite{barlow2017random}, that the heat kernel fulfills lower and upper Gaussian bounds
\begin{equation}
\frac{c}{t^{d/2} }e^{-c'\frac{|x-y|^2}{t}} \leq q^\omega_t(x,y) \leq \frac{C}{t^{d/2}} e^{-C'\frac{|x-y|^2}{t}}, \qquad t \geq 1 \vee \tilde{c}|x-y|,\, x,y \in \bbZ^d,\, \omega \in \Omega_\lambda. 
\end{equation}
Importantly, the constants in these bounds depend on the conductances $\omega \in \Omega_\lambda$ only through $\lambda$ (which we suppress in the notation, see also the convention on constants at the end of Section~\ref{sec:intro}). These bounds readily imply bounds on the Green function, which are uniform in $\omega \in \Omega_\lambda$:
\begin{equation}
\label{eq:QuenchedGFEstimate}
\frac{c_2}{|x-y|^{d-2} \vee 1} \leq g^\omega(x,y) \leq \frac{c_3}{|x - y|^{d-2} \vee 1}, \qquad x, y \in \bbZ^d,\, \omega \in \Omega_\lambda. 
\end{equation}

We also have uniform killed heat kernel bounds, for instance, we can utilize a (slightly modified) version of Theorem 5.26 in~\cite{barlow2017random} to state that for $\vartheta \in (0,1)$ there exists an $R_0(\vartheta) \in \bbN$ such that for all $R \geq R_0(\vartheta)$, one has for every $\omega \in \Omega_\lambda$ and $x_0 \in \bbZ^d$:
\begin{equation}
\label{eq:HK_bound}
q_{t,B(x_0,R)}^{\omega}(x_1,x_2) \geq c_4(\vartheta) t^{-d/2}, \ x_1,x_2 \in B(x_0, (1-\vartheta)R) \text{ and } c_5(\vartheta)R^2 \leq t \leq R^2,
\end{equation} 
where $c_5(\vartheta) < 1$. 

 By Proposition 6.2 of~\cite{delmotte1997inegalite}, which follows from the elliptic Harnack inequality for the weighted graph $(\bbZ^d, \bbE^d,\omega)$, the following H\"older-regularity property of $\omega$-harmonic functions holds: There exists a constant $\tau > 0$ (only depending on $\lambda$), such that if $u : \bbZ^d \rightarrow \bbR$ is $\omega$-harmonic in $B^{(1)}(x_0,2r)$ for $x_0 \in \bbZ^d$, $r > 0$, then
 \begin{equation}
\label{eq:HoelderRegProperty}
 |u(x) - u(y)| \leq C \left( \frac{|x-y|_\infty}{r} \right)^\tau \sup_{z \in B^{(1)}(x_0,2r)}|u(z)| , \qquad x,y \in B^{(1)}(x_0,r).
\end{equation}  

\vspace{\baselineskip}

We now introduce some potential theory associated with the random walk. Given $A \subset \subset \bbZ^d$ and $\omega \in \Omega_\lambda$, we define the equilibrium measure of $A$
\begin{equation}
\label{eq:DefEqMeasure}
e^\omega_A(x) = P_x^\omega[\widetilde{H}_A = \infty]\omega_x \mathbbm{1}_A(x), \qquad x \in \bbZ^d,
\end{equation}
and its (finite) total mass, the capacity of $A$
\begin{equation}
\capa^\omega(A) = \sum_{x \in A}e_A^\omega(x). 
\end{equation}
In the case of a closed $\ell^\infty$-ball $B(x,L)$, where $x \in \bbZ^d$ and $L \geq 1$, one has the classical estimate
\begin{equation}
\label{eq:QuenchedBoxCapacityEstimate}
cL^{d-2} \leq \capa^\omega(B(x,L)) \leq C L^{d-2},
\end{equation}
in which the constants depend on $\omega \in \Omega_\lambda$ only through $\lambda$, see Lemma 7.21 (a) in~\cite{barlow2017random} and its proof, using the quenched bounds~\eqref{eq:QuenchedGFEstimate}.
The equilibrium potential of $A$ is defined as
\begin{equation}
\label{eq:EquilibriumPotential}
h^\omega_A(x) = P^\omega_x[H_A < \infty], \qquad x \in \bbZ^d,
\end{equation}
 and it is related to the equilibrium measure $e^\omega_A$ by the identity (see for instance Proposition 7.2 of~\cite{barlow2017random})
\begin{equation}
\label{eq:LastExDecomp}
h^\omega_A(x) = \sum_{y \in A}g^\omega(x,y) e^\omega_{A}(y), \qquad x \in \bbZ^d.
\end{equation}
 
We will also need the notion of the Dirichlet form, defined for $f: \bbZ^d \rightarrow \bbR$:
\begin{equation}
\cE^\omega(f) = \frac{1}{2} \sum_{x \sim y} \omega_{x,y} \big( f(y) - f(x) \big)^2.
\end{equation}
For $f,g : \bbZ^d \rightarrow \bbR$, we define by polarization (and symmetry of $\omega_{x,y}$) 
\begin{equation}
\cE^\omega(f,g) = \frac{1}{2} \sum_{x \sim y} \omega_{x,y} \big( f(y) - f(x) \big)\big( g(y) - g(x) \big),
\end{equation} 
when the resulting series converges absolutely (and so $\cE^{\omega}(f) = \cE^{\omega}(f,f)$). The Dirichlet form is related to the capacity of $A$ via the identity
\begin{equation}
\label{eq:Capacity_and_Dirichlet}
\capa^\omega(A) = \cE^\omega(h^\omega_A).
\end{equation}
For $h : \bbZ^d \rightarrow \bbR$ with finite support one has the inequality
\begin{equation}
\label{eq:EnergyDirichletFormBound}
\begin{split}
\Big\vert \sum_{x \in \bbZ^d} f(x) h(x) \Big\vert & \leq W^\omega(h)^{1/2}\cE^\omega(f)^{1/2}, \text{ with} \\
W^\omega(h) & = \sum_{x,y \in \bbZ^d} g^\omega(x,y)h(x)h(y), 
\end{split}
\end{equation}
for every $f : \bbZ^d \rightarrow \bbR$, see Proposition 1.3 in~\cite{sznitman2012topics} and its proof. Finally, we also introduce the capacity of $A\subseteq  B\subset \subset \bbZ^d$ for the random walk killed upon exiting $B$. That is, we set
\begin{equation}
    \capa^\omega_B(A) =  \cE^\omega(h^\omega_{A,B}, h^\omega_{A,B}), \qquad  h^\omega_{A,B}(x) = P_x^\omega[H_A<T_B],\qquad x\in \bbZ^d,
\end{equation}
and recall that one has the variational characterization
\begin{equation}\label{eq:variationalchar}
    \capa^\omega_B(A) =\inf \cE^\omega(f),
\end{equation}
where the infimum is taken over all $f:\bbZ^d\to \bbR$ such that $f|_A=1$ and $f|_{\bbZ^d\setminus B} = 0$.
\vspace{\baselineskip}

We now turn to the Gaussian free field on the weighted graph $(\bbZ^d, \bbE^d, \omega)$ for a fixed $\omega \in \Omega_\lambda$, which we introduced in~\eqref{eq:IntroDefGFF}, and recall a classical domain Markov property. Specifically, we define for $U \subseteq \bbZ^d$ and $\omega \in \Omega_\lambda$ the ($\omega$-)harmonic average $\xi^{\omega,U}$ of $\varphi$ in $U$ and the ($\omega$-)local field $\psi^{\omega,U}$ by
\begin{align}
\label{eq:HarmonicAverageDef}
\xi_x^{\omega,U}& = E^\omega_x[\varphi_{X_{T_U}}, T_U < \infty] = \sum_{y \in \bbZ^d} P^\omega_x[X_{T_U} = y, T_U < \infty]\varphi_y, \qquad x \in \bbZ^d; \\
\label{eq:LocalFieldDef}
\psi_x^{\omega,U} &= \varphi_x - \xi_x^{\omega,U}, \qquad x \in \bbZ^d.
\end{align}
Note that by definition $\varphi_x = \xi_x^{\omega,U} + \psi_x^{\omega,U}$, $\psi^{\omega,U}_x = 0$ if $x \in \bbZ^d \setminus U$, whereas $\xi^{\omega,U}_x = \varphi_x$ in this case. The domain Markov property asserts that
\begin{equation}
\label{eq:DomainMP}
\begin{minipage}{0.8\linewidth}
  $(\psi^{\omega,U}_x)_{x \in \mathbb{Z}^d}$ is independent of $\sigma(\varphi_y \, : \, y \in U^c)$ (in particular of  $(\xi^{\omega,U}_x)_{x \in \mathbb{Z}^d}$),
  and is distributed as a centered Gaussian field with covariance $g^\omega_U(\cdot,\cdot)$,
\end{minipage}
\end{equation}
where, $g^\omega_U(\cdot,\cdot)$ is the Green function of the random walk among conductances $\omega$ killed upon exiting $U$, see~\eqref{eq:KilledGreenFunction}.
\vspace{\baselineskip}

Lastly, we recall a classical entropy inequality that will be instrumental in the derivation of large deviation lower bounds in Section~\ref{sec:LowerBound}. Let $\widetilde{\bbP}$ and $\bbP$ be two probability measures with $\widetilde{\bbP}$ absolutely continuous with respect to $\bbP$. We define the relative entropy of $\widetilde{\bbP}$ with respect to $\bbP$ as 
\begin{equation}
\label{eq:RelEntropy}
H(\widetilde{\bbP} | \bbP) = \widetilde{\bbE} \Big[ \log \frac{\De \widetilde{\bbP}}{\De \bbP} \Big] = \bbE\Big[\frac{\De \widetilde{\bbP}}{\De \bbP}\log \frac{\De \widetilde{\bbP}}{\De \bbP} \Big] \in [0,\infty],
\end{equation}
where $\widetilde{\bbE}$ and $\bbE$ denote the expectation with respect to the probability measures $\widetilde{\bbP}$ and $\bbP$, respectively. For an event $F$ with positive $\widetilde{\bbP}$-probability, one has
\begin{equation}
\label{eq:RelEntropyInequality}
\bbP[F] \geq \widetilde{\bbP}[F]\exp\Big(-\frac{1}{\widetilde{\bbP}[F]} \Big(H(\widetilde{\bbP} | \bbP) + \frac{1}{\mathrm{e}}\Big) \Big),
\end{equation}
see, e.g., p.76 of~\cite{deuschel2001large}.

\section{Level-set percolation of the Gaussian free field with random conductances}
\label{sec:GFF_RandomCond}

    
    In this section we consider the Gaussian free field with random conductances. More precisely, we introduce a stationary and ergodic environment measure $P$ governing the random conductances $\omega \in \Omega_\lambda$. For a realization of $\omega$ sampled according to $P$, we then consider the field $\varphi$ as defined in~\eqref{eq:IntroDefGFF}, and characterize the percolation phase transition of its upper level-sets $E^{\geq \alpha}$, see~\eqref{eq:LevelSetIntro}. In particular, we give the definitions of the thresholds $\alpha^\omega_\ast$, $\alpha^\omega_{\ast\ast}$ and $\overline{\alpha}^\omega$, that are shown to be constant, with values $\alpha_\ast,\alpha_{\ast\ast}$ and $\overline{\alpha}$, for $P$-a.e.\ $\omega \in \Omega_\lambda$ and to fulfill $0 <\overline{\alpha} \leq \alpha_\ast \leq \alpha_{\ast\ast} < \infty$, see Proposition~\ref{prop:alpha_ast}, Theorem~\ref{thm:StretchedExpDecay} and Theorem~\ref{thm:alpha_bar} below.  
    
The parameter $\alpha_\ast$ describes the percolation threshold, and separates a phase $\alpha > \alpha_\ast$, where a (unique) infinite component exists in $E^{\geq \alpha}$ from a phase where all connected components of $E^{\geq \alpha}$ are finite. The range $\alpha > \alpha_{\ast\ast}$ describes a \textit{strongly non-percolative regime} for $E^{\geq \alpha}$, where one essentially has a stretched exponential decay of the connection probability, see Theorem~\ref{thm:StretchedExpDecay}. On the other hand, for $\alpha < \overline{\alpha}$ one observes a \textit{strongly percolative regime} for $E^{\geq \alpha}$, which is loosely speaking characterized by the  existence and local uniqueness of large clusters, similar as in~\cite{drewitz2018geometry}. Note however that the version of $\overline{\alpha}^\omega$ we study is potentially strictly bigger than the parameter studied in~\cite{drewitz2018geometry}, see Remark~\ref{rem:OtherCritParam}.

In the case of constant, non-random conductances, one has the equality of the thresholds $\overline{\alpha} = \alpha_\ast = \alpha_{\ast\ast}$, as shown recently in Theorem 1.1 of~\cite{duminil2020equality}. However, in our case, the equality of the critical parameters cannot immediately be obtained as a straightforward adaptation of the methods used in~\cite{duminil2020equality}, partially due to the lack of translation-invariance of the field.  
\vspace{\baselineskip}

Recall the definition of $\Omega_\lambda$ above~\eqref{eq:UniformEll}. We endow $\Omega_\lambda$ with the canonical $\sigma$-algebra of cylinders $\cG$. We consider a probability measure (the environment measure) $P$ on $(\Omega_\lambda,\cG)$ and the group of shifts
    \begin{equation}
    \label{eq:ShiftsEnvironment}
        \tau_x : \Omega_\lambda \to \Omega_\lambda,\quad (\tau_x \omega)_{y,z} =\omega_{x+y,x+z},\quad  x,y,z\in \bbZ^d,\, \omega\in \Omega_\lambda. 
    \end{equation}   
    We assume that $P$ is stationary and ergodic with respect to $(\tau_x)_{x\in \bbZ^d}$, that is,
    \begin{itemize}
        \item[i)] $P[\tau_x(A)] = P[A]$ for all $A\in \cG$ and all $x\in\bbZ^d$;
        \item[ii)] If $f:\Omega_\lambda \to \bbR$ is measurable and such that $f(\tau_x \omega) = f(\omega)$ for all $x\in \bbZ^d$ and $P$-a.e.\ $\omega \in \Omega_\lambda$, then $P$-a.s., $f$ is constant. 
    \end{itemize}
    
    Additionally, we consider the group of space-shifts on $\bbR^{\bbZ^d}$
    \begin{equation}
        t_x: \bbR^{\bbZ^d}\to \bbR^{\bbZ^d}, \quad t_x \varphi_\cdot = \varphi_{\cdot+x},\quad  \varphi \in \bbR^{\bbZ^d},\,  x\in \bbZ^d.    
    \end{equation}
    It follows from the definitions that for all bounded measurable functions $f:\bbR^{\bbZ^d} \to \bbR$, $\bbE^\omega[f(t_x \varphi)] = \bbE^{\tau_x \omega} [f(\varphi)]$ for all $x\in \bbZ^d$, $P$-a.s.

    As an immediate consequence of the stationarity and ergodicity of the environment measure we deduce that the critical parameter for level-set percolation is deterministic.
\begin{prop}
   \label{prop:alpha_ast} 
     The critical parameter
        \begin{equation}
            \alpha_*^\omega = \inf\big\{\alpha\in \bbR:\, \bbP^\omega [E^{\geq \alpha}\mbox{ contains an infinite cluster}\,] = 0\big\}
        \end{equation}
        is $P$-a.s.\ constant.
    \end{prop}
    \begin{proof} We observe that the event $\{E^{\geq \alpha}\mbox{ contains an infinite cluster}\}$ is translation-invariant with respect to $t_x$, $x\in \bbZ^d$. Thus, for $P$-a.e.\ $\omega\in \Omega$ and all $x\in \bbZ^d$
        \begin{equation*}
            \bbP^\omega [E^{\geq \alpha}\mbox{ contains an infinite cluster}] = \bbP^{\tau_x\omega} [E^{\geq \alpha}\mbox{ contains an infinite cluster}].
        \end{equation*}
    By ergodicity and stationarity of $P$, $ \bbP^\omega [E^{\geq \alpha}\mbox{ contains an infinite cluster}]$ is $P$-a.s.\ constant, and consequently $\alpha_*^\omega$ is also $P$-a.s.\ constant. 
\end{proof}
Next, we introduce a parameter $\alpha^\omega_{\ast\ast}\geq \alpha^\omega_{*}$, characterized by a uniform decay of a box-crossing probability. More precisely, it is defined by
     \begin{equation}
     \label{eq:alpha_astast_Def}
            \alpha^\omega_{\ast\ast} = \inf \Big\{\alpha\in \bbR:\,\text{$\exists \rho > 1$ with } \text{$\lim_{L\to\infty} \sup_{x\in B(0,L^\rho)}\!\! \bbP^\omega[B(x,L) \stackrel{\geq \alpha }{\longleftrightarrow} \partial B(x,2L)] = 0$}\Big\}.
     \end{equation} 
     Here and in the following, for $H,K \subseteq \bbZ^d$, the notation $ \lbrace H \stackrel{\geq \alpha}{ \longleftrightarrow}   K\rbrace$ denotes the existence of a nearest-neighbor path in $E^{\geq \alpha}$ starting in $H$ and ending in $K$.
   The parameter $\alpha^\omega_{\ast\ast}$ determines a strongly non-percolative regime $\alpha > \alpha_{\ast\ast}^\omega$. We show in the next theorem that this parameter is in fact $P$-a.s.\ constant and finite for any dimension $d \geq 3$, and that there is a $\rho = \rho(\omega) > 1$  such that in the strongly percolative regime, the connectivity function $\sup_{x \in B(0,\frac{1}{2}L^\rho)} \bbP^\omega[x \stackrel{\geq \alpha}{\longleftrightarrow} x+z ]$ has a stretched exponential decay in $|z|_\infty\wedge \frac{L^\rho}{100}$.

\begin{theorem}
\label{thm:StretchedExpDecay}
\begin{align}
\label{eq:alpha_astastConstFinite}
&\mbox{$\alpha_{\ast\ast}^\omega$ is $P$-a.s.\ constant and finite,} \\
\label{eq:alpha_astastpercolation}
&\alpha_\ast \leq \alpha_{\ast\ast}.
\end{align}
Moreover, for $d \geq 4$, $\alpha > \alpha_{\ast\ast}$, and $P$-a.e.\ $\omega \in \Omega_\lambda$, there exists a $\rho =\rho(\omega) >1$ such that
\begin{equation}
\label{eq:StretchedExpBiggerEqualFour}
\sup_{x \in B(0,\frac{1}{2} L^\rho)} \bbP^\omega[x \stackrel{\geq \alpha}{\longleftrightarrow} x+z ] \leq C(\omega,\alpha) e^{-\frac{c(\omega,\alpha)}{L} \left(|z|_\infty \wedge  \frac{L^\rho}{100}\right)}, \qquad z \in \bbZ^d, L \geq 1.
\end{equation}
In $d = 3$, for every $b > 1$, $\alpha > \alpha_{\ast\ast}$, and  $P$-a.e.\ $\omega \in \Omega_\lambda$, there exists a $\rho = \rho(\omega) > 1$ such that
\begin{equation}
\label{eq:StretchedExpThree}
\sup_{x \in B(0,\frac{1}{2}L^\rho)} \bbP^\omega[x \stackrel{\geq \alpha}{\longleftrightarrow} x+z ] \leq C(\omega,\alpha,b) e^{-\frac{c(\omega,\alpha,b)}{L\log^{3b}|z|_\infty}\left(|z|_\infty \wedge \frac{L^\rho}{100} \right)}, \qquad z \in \bbZ^d, L \geq 1.
\end{equation}
\end{theorem}

\begin{rem} Let $d\geq 4$,  $\alpha>\alpha_{**}$, $\omega$ in a subset of $\Omega_\lambda$ of full $P$-measure, and $\rho(\omega)$ as above. As a consequence of~\eqref{eq:StretchedExpBiggerEqualFour} and choosing $L=|z|_\infty^{1/\rho}$, we have
     \begin{equation}
         \bbP^\omega[0 \stackrel{\geq \alpha}{\longleftrightarrow} z ] \leq C e^{- c |z|^{1-1/\rho}_\infty}, \qquad z \in \bbZ^d.
     \end{equation}    
for some positive constants $c,C$ that depend on $\omega$ and $\alpha$. Note that this formally corresponds to (2.1) in Theorem 2.1 of~\cite{popov2015decoupling} when setting $\rho = \infty$ in the exponent. 
In the case of constant conductances, the asymptotics of the connection probability have been investigated in much more detail, see~\cite{goswami2021radius}. In particular, the logarithmic correction $\log^{3b} |z|_\infty$ for $b > 1$ in $d = 3$ (see~\eqref{eq:StretchedExpThree}) can be replaced by $\log |z|_\infty$ in the case of constant conductances.
\end{rem}

The claims~\eqref{eq:StretchedExpBiggerEqualFour} and~\eqref{eq:StretchedExpThree} of Theorem~\ref{thm:StretchedExpDecay} will follow by a routine application of a multiscale argument following the arguments of~\cite[Section 7]{popov2015soft}, together with a Gaussian decoupling inequality which is similar to Theorem 1.2, Corollary 1.3 and Proposition 1.4 of~\cite{popov2015decoupling} (which dealt with the case of the Gaussian free field with constant conductances). 

We begin by developing the relevant (quenched) decoupling inequality.  

\begin{prop}
\label{prop:QuenchedDecouplingIneq}
Let $K_1 = B(x,L_1), K_2 = B(y,L_2) \subseteq \bbZ^d$, $L_1,L_2 \geq 0$, $x,y \in \bbZ^d$ be disjoint boxes. Assume that $f_1, f_2 : \bbR^{\bbZ^d} \rightarrow [0,1]$  are such that $f_i$ is supported on $K_i$, $i \in \{1,2\}$, and assume that $f_2$ is increasing.  For $\omega \in \Omega_\lambda$ and $\delta > 0$, we define 
\begin{equation}
G_\delta^\omega = \left\{ \sup_{x \in K_2} |\xi^{\omega,K_1^c}_x| \leq \frac{\delta}{2} \right\}
\end{equation}
(recall~\eqref{eq:HarmonicAverageDef} for the definition of the field $\xi^{\omega,K_1^c}$).

One has, $\bbP^\omega$-a.s.,
\begin{equation}
\label{eq:Decoupling1}
\begin{split}
 (\bbE^\omega[f_2(\varphi - \delta)] & - \bbP^\omega[(G^\omega_\delta)^c]) \mathbbm{1}_{G_\delta^\omega}  \leq \bbE^\omega[f_2(\varphi)|\sigma( \varphi_{x} \, : \,  x \in K_1) ] \mathbbm{1}_{G^\omega_\delta} \\
 & \leq (\bbE^\omega[f_2(\varphi + \delta)] +\bbP^\omega[(G^\omega_\delta)^c]) \mathbbm{1}_{G_\delta^\omega},
 \end{split}
\end{equation} 
and moreover
 \begin{equation}
 \begin{split}
 \label{eq:Decoupling2}
\bbE^\omega f_1(\varphi) \bbE^\omega f_2(\varphi - \delta) - 2\bbP^\omega[(G^\omega_\delta)^c]& \leq \bbE^\omega[f_1(\varphi)f_2(\varphi)] \\
& \leq \bbE^\omega f_1(\varphi)\bbE^\omega f_2(\varphi+\delta) + 2\bbP^\omega[(G^\omega_\delta)^c].
\end{split}
\end{equation}
Furthermore, if we set $s = d_\infty(K_1,K_2) > 0$ and $r = L_1 \wedge L_2$, one has
\begin{equation}
\label{eq:DecouplingErrorBound}
\sup_{\omega \in \Omega_\lambda} \bbP^\omega[(G^\omega_\delta)^c] \leq c( r+s)^{d-1} \exp\left(-c'\delta^2 s^{d-2}\right).
\end{equation}
\end{prop}
\begin{proof}
The decoupling inequalities~\eqref{eq:Decoupling1} and~\eqref{eq:Decoupling2} follow exactly as Theorem 1.2 and Corollary 1.3 in~\cite{popov2015decoupling}, with the necessary changes of $\bbP$ to $\bbP^\omega$ and use of the respective domain Markov property for the $\omega$-dependent Gaussian free field (see~\eqref{eq:DomainMP}).

We turn to the proof of~\eqref{eq:DecouplingErrorBound}. Let us first consider the case that $L_1 \leq L_2$, where we introduce the box $\widetilde{K}_1 = B(x,L_1 + s)$, and notice that the vertices of its interior boundary $\partial_{\text{in}}\widetilde{K}_1$ have a $\ell^\infty$-distance at least $s$ to $K_1$. For $x \in \partial_{\text{in}}\widetilde{K}_1$, we set
\begin{equation}
\Lambda^\omega_{\delta,x} = \left\{|\xi^{\omega,K_1^c}_x| \leq \frac{\delta}{2} \right\}.
\end{equation}
By an application of~\eqref{eq:DecompositionGreenKilledGreen}, we find that for all $x \in \partial_{\text{in}}\widetilde{K}_1$
\begin{equation}
\begin{split}
\label{eq:VarianceBoundXiDecouplingProof}
\bbE^\omega[(\xi^{\omega,K_1^c}_x)^2] & = \sum_{y \in K_1} P_x^\omega[H_{K_1} <\infty, X_{H_{K_1}} = y] g^\omega(x,y)  \\
&\leq \sup_{y \in K_1} g^\omega(x,y) \stackrel{\eqref{eq:QuenchedGFEstimate}}{\leq} \frac{c_3}{s^{d-2}}.
\end{split}
\end{equation}
Since every nearest-neighbor path from $K_2$ to $K_1$ passes through $\partial_{\text{in}}\widetilde{K}_1$, we find (using the strong Markov property of the random walk)
\begin{equation}
(G^\omega_\delta)^c \subseteq \bigcup_{x \in \partial_{\text{in}}\widetilde{K}_1} (\Lambda^\omega_{\delta,x})^c.
\end{equation}
The claim then follows by applying a union bound with the estimate $|\partial_{\text{in}} \widetilde{K}_1| \leq c'( r+s)^{d-1}$ together with the exponential Chebyshev inequality, the bound~\eqref{eq:VarianceBoundXiDecouplingProof}, and taking finally the supremum over $\omega \in \Omega_\lambda$. The case where $L_1 > L_2$ follows analogously, by instead considering the box $\widetilde{K}_2 = B(y,L_2+s)$ and its interior boundary $\partial_{\text{in}} \widetilde{K}_2$, which again has to be crossed by every path from $K_2$ to $K_1$. 
\end{proof}

\begin{rem}
\label{rem:GeneralizedDecoupling}
The decoupling inequality can in fact be made more general by considering $K_1$ finite and $K_2$ arbitrary, disjoint from $K_1$, with a slightly more involved expression for the polynomial factor in the error bound~\eqref{eq:DecouplingErrorBound}, see (1.10) of Proposition 1.4 in~\cite{popov2015decoupling}. However, the present result is sufficient for our purposes. 
\end{rem}

\begin{proof}[Theorem~\ref{thm:StretchedExpDecay}]
We start by proving~\eqref{eq:alpha_astastConstFinite}. Let $z \in \bbZ^d$. If $\alpha > \alpha^\omega_{\ast\ast}$, then there exists $\rho > 1$ such that 
\begin{equation}
\lim_{L \rightarrow \infty}\sup_{x \in B(0,L^\rho)}\bbP^\omega[B(x,L) \stackrel{\geq \alpha}{\longleftrightarrow} \partial B(x,2L)] = 0.
\end{equation}
Consider $1 < \widetilde{\rho} < \rho$, then, for every $L \geq c(\rho,\widetilde{\rho},z)$, one has $B(z, L^{\widetilde{\rho}})\subseteq B(0,L^\rho)$ and consequently it holds that 
\begin{equation}
\lim_{L \rightarrow \infty}\sup_{x \in B(0,L^{\widetilde{\rho}})}\bbP^{\tau_z\omega}[B(x,L) \stackrel{\geq \alpha}{\longleftrightarrow} \partial B(x,2L)] = 0.
\end{equation} 
This shows that $\alpha^\omega_{\ast\ast} \geq \alpha^{\tau_z\omega}_{\ast\ast}$ for every $z \in \bbZ^d$. By symmetry, one has that in fact $\alpha^\omega_{\ast\ast} = \alpha^{\tau_z\omega}_{\ast\ast}$ and therefore by stationarity and ergodicity of $P$, $\alpha^\omega_{\ast\ast}$ is constant $P$-a.s.\ 
The finiteness of $\alpha_{\ast\ast}$ follows for instance from Remark 7.2 1) and Corollary 7.3 in~\cite{drewitz2018geometry}, see also Remark~\ref{rem:OtherCritParam}, 2). The statement~\eqref{eq:alpha_astastpercolation} follows by definition, since $\{E^{\geq \alpha}\mbox{ contains an infinite connected component}\} = \bigcup_{x \in \bbZ^d} \{x \stackrel{\geq \alpha}{\longleftrightarrow} \infty\}$, and
\begin{equation}
\bbP^\omega[x \stackrel{\geq \alpha}{\longleftrightarrow} \infty] \leq \bbP^\omega[B(x,L)\stackrel{\geq \alpha}{\longleftrightarrow} \partial B(x,2L) ].
\end{equation}

We turn to the statements~\eqref{eq:StretchedExpBiggerEqualFour} and~\eqref{eq:StretchedExpThree} about the decay of the connectivity function above $\alpha_{\ast\ast}$. Let $\bbP^\omega_\alpha$ be the law of the level-set $E^{\geq \alpha}$ on $\{0,1\}^{\bbZ^d}$. By combination of~\eqref{eq:Decoupling2} and~\eqref{eq:DecouplingErrorBound} we see that for any increasing events $A_1$ and $A_2$ depending on disjoint boxes of size $r$ with distance $s > 0$ from each other, one has the ``sprinkling inequality''
\begin{equation}
\label{eq:SprinklingIneq}
\bbP^\omega_{\alpha}[A_1 \cap A_2] \leq \bbP^\omega_{\alpha-\delta}[A_1]\bbP^\omega_{\alpha-\delta}[A_2] + c(r+s)^{d-1} \exp\left(-c' \delta^2 s^{d-2}\right).
\end{equation}
The proof will proceed along the lines of Section 7 of~\cite{popov2015soft}, but some care is required due to the local nature of $\alpha_{\ast\ast}$. 

We introduce some notation first. Throughout we are considering $\alpha> \alpha_{\ast\ast}$, and set $\rho > 1$ (the choice of $\rho$ will be specified later).  Let $b \in (1,2]$ and set for $L_1 \geq 100$:
\begin{equation}
L_{k+1} = 2\left(1 + (k+5)^{-b}\right) L_k, \ k \geq 1.
\end{equation}
This sequence fulfills, for some $c(b) \in (1,\infty)$, 
\begin{equation}
L_1 2^{k-1} \leq L_k \leq c(b) L_1 2^{k-1}.
\end{equation}
We then set for $k \geq 1$ and $x \in \bbZ^d$
\begin{equation}
C_x^k = [0,L_k)^d \cap \bbZ^d+x, \qquad D_x^k = [-L_k, 2L_k) \cap \bbZ^d +x,
\end{equation}
and consider events of the form 
\begin{equation}
A_x^k(\alpha) = \{C^k_x \stackrel{\geq \alpha }{\longleftrightarrow} \bbZ^d \setminus D_x^k \},
\end{equation}
which fulfill 
\begin{equation}
A_x^{k+1}(\alpha) \subseteq \bigcup_{i \leq 3^d, j \leq 2d7^{d-1}} A^k_{x_j^k(x)}(\alpha) \cap A^k_{y_j^k(x)}(\alpha),
\end{equation}
where the collection of points $\{x_j^k(x)\}_{i = 1}^{3^d}$ is such that $C_x^{k+1}$ consists of the union of $(C^k_{x_i^k(x)}; i = 1,..., 3^d)$, and the union of all $(C^k_{y_j^k(x)}; j = 1,...,2d7^{d-1})$ is disjoint from $D_x^{k+1}$ and contains $\partial (\bbZ^d \setminus D_x^{k+1})$, see (7.7) and (7.8) of~\cite{popov2015soft}. In particular, one has that $|x -y^k_j(x)|_\infty \leq 3 L_{k+1}$ and $|x-x^k_i(x)| \leq 3L_{k+1}$ for all $ 1 \leq i\leq 3^d, 1 \leq j\leq 2d7^{d-1}$. We now define for $k \geq 1$ such that $L_1^\rho \geq 10L_k$
\begin{equation}
p^\omega_{k}(\alpha) = \sup_{x \in B(0, L_1^\rho - 10L_k)} \bbP^\omega[A^k_x(\alpha)]
\end{equation}
(importantly, this quantity deviates from the respective one in the proof in~\cite{popov2015soft}, since our model is manifestly not translation-invariant).

We then introduce scales for the ``sprinkling'' that we will use, namely we set $\widehat{\alpha} = \frac{\alpha_{\ast\ast}+\alpha}{2} (> \alpha_{\ast\ast})$ and the intermediate values (for sufficiently small $\varepsilon > 0$)
\begin{equation}
\alpha_k = \frac{\widehat{\alpha}}{\prod_{j = 1}^{k-1} (1- \varepsilon j^{-b})} (< \alpha), \ k \in \bbN.
\end{equation}
Using the sprinkling inequality~\eqref{eq:SprinklingIneq} with $r = 3 \sqrt{d} L_k$, $s \geq 2^{k-1} L_1(k+5)^{-b}$ and $\delta = \varepsilon k^{-b}\alpha_{k+1}$ (and $\alpha_k < \alpha$), one obtains 
\begin{equation}
\begin{split}
\bbP^{\omega}[A^{k+1}_x(\alpha_{k+1})] &\leq \sum_{i \leq 3^d, j \leq 2d7^{d-1}} \bbP^\omega[A^k_{x^k_i(x)}(\alpha_k)]\bbP^\omega[A^k_{y^k_j(x)}(\alpha_k)] \\&
+ c2^{kd}L_1^d \exp\left( -c' \frac{1}{k^{2b}} \left( \frac{L_1 2^k}{(k+5)^b} \right)^{d-2} \right).
\end{split}
\end{equation}
The terms on the right-hand side are bounded by taking respective suprema over $x \in B(0,L^\rho_1 - 10L_{k+1})$, since for $x \in B(0,L_1^\rho - 10L_{k+1})$, all $x^k_i(x)$
and $y^k_j(x)$ belong to $B(0,L_1^\rho- 7L_{k+1}) \subseteq B(0,L_1^\rho - 10 L_k)$, where the latter inclusion follows from $\frac{L_{k+1}}{L_k} \geq 2 > \frac{10}{7}$. After this step, one takes a supremum over $x \in B(0,L_1^\rho - 10 L_{k+1})$, which then yields
\begin{equation}
p^\omega_{k+1}(\alpha_{k+1}) \leq \frac{2d \cdot 21^d}{7} p^\omega_k(\alpha_k)^2 + c2^{kd}L_1^d \exp\left( -c'  \frac{1}{k^{2b}} \left( \frac{L_1 2^k}{(k+5)^b} \right)^{d-2} \right).
\end{equation}
This last part will allow an induction argument to proceed. Since $\widehat{\alpha} > \alpha_{\ast\ast}$, there is a choice of $\rho = \rho(\omega) >1$ such that for large enough $\widetilde{L}_1$, one has for all $L_1 \geq \widetilde{L}_1$ that 
\begin{equation}
p^\omega_1(\widehat{\alpha}) \leq \sup_{x \in B(0,L_1^\rho)} \bbP^\omega[ C_x^1 \stackrel{\geq \widehat{\alpha} }{\longleftrightarrow} \bbZ^d \setminus D_x^1] < \frac{7}{2d \cdot 21^d}.
\end{equation}
In the case of $d \geq 4$, the same argument leading up to (7.14) of~\cite{popov2015soft} can be used to establish
\begin{equation}
\label{eq:UniformDecayBoundalpha_k}
p^\omega_k(\alpha_k) \leq e^{-\cA - \cB 2^k},
\end{equation}
with some $\cA > 0$ and $\cB \in (0,1)$. Now let $x \in B(0,\frac{1}{2} L_1^\rho)$, then 
\begin{equation}
\bbP^\omega[x \stackrel{\geq \alpha}{\longleftrightarrow} x + z] \leq \bbP^\omega[[0,L_k)^d+ x \stackrel{\geq \alpha}{\longleftrightarrow} \partial([-L_k,2L_k)^d + x)]
\end{equation}
where $k = \max\{m : \frac{3}{2} L_m < |z|_\infty \wedge \frac{L_1^\rho}{100} \}$. 
We have that $L_k = O(2^k L_1)$ and $B(0,L_1^\rho - 10L_k)$ contains $B(0,\frac{1}{2}L_1^\rho)$, which yields (setting $L = L_1$)
\begin{equation}
\sup_{x \in B(0,\frac{1}{2}L^\rho)} \bbP^\omega[x \stackrel{\geq \alpha}{\longleftrightarrow} x + z] \leq \exp(-\cA - c\cB \tfrac{|z|_\infty \wedge (L^\rho/100) }{L}),
\end{equation}
by using~\eqref{eq:UniformDecayBoundalpha_k} and the fact that $\alpha_k < \alpha$. The case $d = 3$ can be obtained similarly by a modification of the argument presented in (7.16)--(7.18) of~\cite{popov2015soft}.
\end{proof}
Finally we define the critical parameter     
    \begin{equation}
    \label{eq:OverlineAlphaDef}
    \begin{aligned}
        \overline{\alpha}^\omega = \sup\Big\{ \alpha \in \bbR:\,  &\mbox{for all  $\alpha>\beta_+> \beta_-$},\\
        &\mbox{$\varphi$ under $\bbP^\omega$ strongly percolates at levels $\beta_+, \beta_-$}\Big\}
    \end{aligned}
    \end{equation}
    where the Gaussian free field $\varphi$ under $\bbP^\omega$  is said to strongly percolate at levels $\beta_+,\beta_-$ if there exists $\rho > d-1$ such that with $B_x = (x+[0,L)^d) \cap \bbZ^d$, 
    \begin{equation}
    \label{eq:NoLargeComponent}
        \lim_{L\to \infty}\sup_{x\in B(0,L^\rho)} \frac{1}{\log L} \log \bbP^{\omega}\left[\begin{minipage}{0.5\textwidth}$B_x\cap E^{\geq \beta_+}$ has no component of diameter at least $\tfrac{L}{10}$ \end{minipage}
        \right] = -\infty,
    \end{equation}
    and for any $z = L e$ with $|e|=1$, with the notation $D_x = (x + [-3L, 4L)^d) \cap \bbZ^d$, 
    \begin{equation}
    \label{eq:NoLargeComponentsDisconn}
        \lim_{L\to \infty}\sup_{x\in B(0,L^\rho)} \frac{1}{\log L} \log \bbP^{\omega}\left[
        \begin{minipage}{0.5\textwidth}
            there exist components of $B_x\cap E^{\geq \beta_+}$ and $B_{x+z}\cap E^{\geq \beta_+}$ with diameter at least $\tfrac{L}{10}$ which are not connected in $D_x\cap E^{\geq \beta_-}$
        \end{minipage}
        \right] = -\infty.
    \end{equation}

The requirement that $\rho > d -1$ in the above definition allows us to guarantee that in the strongly percolative regime $\alpha < \overline{\alpha}^\omega$, certain local fields within boxes of size $L$ centered at $z$ are well-behaved, uniformly for all $z$ inside $B(0,L^\rho)$ (see Proposition~\ref{prop:PsiBadUniformMesoscopic} for a precise statement). The coarse-graining procedure we employ in Section~\ref{sec:UpperBounds}, where we derive the asymptotic upper bound on the probability of the disconnection event $\cD^\alpha_N$, makes it necessary to control the behavior of the field within boxes of size roughly $L \simeq (N \log N) ^{\frac{1}{d-1}} \ll N$, for $N$ large. Put differently, the strongly percolative regime should allow us to control the probabilities in~\eqref{eq:NoLargeComponent} and~\eqref{eq:NoLargeComponentsDisconn} for boxes of size $L$ uniformly over $B(0,L^\rho)$, with $\rho > d-1$.

\begin{theorem}
\label{thm:alpha_bar}
\begin{align}
\label{eq:alpha_barpositive}
&\mbox{$\overline{\alpha}^\omega$ is $P$-a.s.\ constant and strictly positive,} \\
\label{eq:alpha_barpercolation}
&\overline{\alpha} \leq \alpha_{\ast}.
\end{align}
\end{theorem}
\begin{proof} We start by showing that $\overline{\alpha}^\omega$ is $P$-a.s.\ constant and strictly positive.  Fix $\alpha<\overline{\alpha}^{\omega}$ and $z\in \bbZ^d$. Arguing exactly as in the proof of Theorem~\ref{thm:StretchedExpDecay}, we obtain that if $\varphi$ under $\bbP^\omega$ strongly percolates at levels $\alpha>\beta_+>\beta_-$, then $\varphi$ under $\bbP^{\tau_z \omega}$ strongly percolates at levels $\alpha>\beta_+>\beta_-$. Thus, $\overline{\alpha}^{\omega} \leq \overline{\alpha}^{\tau_z\omega}$ for $P$-a.e.\ $\omega\in \Omega_\lambda$ and for all $z\in \bbZ^d$. By symmetry, we in fact have that $\overline{\alpha}^{\omega} = \overline{\alpha}^{\tau_z\omega}$ for $P$-a.e.\ $\omega\in \Omega_\lambda$ and for all $z\in \bbZ^d$. Since the environment measure is stationary and ergodic, this implies that $\overline{\alpha}^{\omega}$ is $P$-a.s.\ constant. 
Observe that for all $\omega\in \Omega_\lambda$, by Theorem 1.1 of~\cite{drewitz2018geometry}, we have $\overline{\alpha}^\omega>0$, see also Remark~\ref{rem:OtherCritParam}, 2). As  $\overline{\alpha}^\omega$ is $P$-a.s.\ constant, this concludes the proof of~\eqref{eq:alpha_barpositive}. Finally let $\alpha< \overline{\alpha}^{\omega}$, it is routine to see with a union bound that $\bbP^\omega[E^{\geq \alpha}\mbox{ contains an infinite cluster}] > 0$ and thus $\alpha\leq \alpha_\ast^{\omega}$, see for example Remark 5.1 of~\cite{sznitman2015disconnection}. Since $\alpha< \overline{\alpha}^{\omega}$ was arbitrary, we obtain $\overline{\alpha}^{\omega}\leq  \alpha_\ast^{\omega}$. 
\end{proof}

\begin{rem}
\label{rem:OtherCritParam}
1) In the case of constant conductances, the parameter $\overline{\alpha}$ has been introduced to study the well-behavedness of (parts of) the supercritical phase of $E^{\geq \alpha}$. For instance~\cite{drewitz2014chemical} establishes bounds on chemical distances and a shape theorem for clusters in $E^{\geq \alpha}$, for $\alpha < \overline{\alpha}$. It is also known that the random walk on the (unique) infinite cluster of $E^{\geq \alpha}$ fulfills a quenched functional central limit theorem, see~\cite{procaccia2016quenched}. The parameter $\alpha_{\ast\ast}$ has been introduced in~\cite{rodriguez2013phase} for the case of constant conductances and studied extensively in~\cite{popov2015decoupling} to establish the (stretched) exponential decay, similar to Theorem~\ref{thm:StretchedExpDecay}, valid for all $z \in \bbZ^d$. Recently it was shown in~\cite{duminil2020equality} that in this case, one has $\overline{\alpha} = \alpha_{\ast}  = \alpha_{\ast\ast}$.

2) The case of non-random, inhomogeneous conductances is studied on general graphs in~\cite{drewitz2018geometry}. In their set-up, the quantities $\overline{h}^\omega$ and $h_{\ast\ast}^\omega$ for the weighted graph $(\bbZ^d,\bbE_d,\omega)$ are introduced in (1.9) and (7.9) of the same reference, and these quantities roughly correspond to $\overline{\alpha}^\omega$ in~\eqref{eq:OverlineAlphaDef} and $\alpha_{\ast\ast}^\omega$ in~\eqref{eq:alpha_astast_Def}, where the uniformity of the respective defining properties is required over all of $\bbZ^d$. We therefore have for any $\omega \in \Omega_\lambda$
\begin{equation}
0 < \overline{h}^\omega \leq \overline{\alpha}^\omega \leq \alpha^\omega_{\ast} \leq \alpha^\omega_{\ast\ast} \leq h^\omega_{\ast\ast} < \infty,
\end{equation}
using Theorem 1.1 and Remark 7.2 1) of~\cite{drewitz2018geometry}. Note that we may conceivably have strict inequalities $\overline{h}^\omega < \overline{\alpha}$ and $\alpha_{\ast\ast} < h_{\ast\ast}^\omega$ for $P$-a.e.\ $\omega \in \Omega_\lambda$. Indeed, considering the case of i.i.d. conductances, requiring uniformity in~\eqref{eq:alpha_astast_Def},~\eqref{eq:NoLargeComponent} or~\eqref{eq:NoLargeComponentsDisconn} over all $\bbZ^d$ leads to the inspection of ``bad regions'' for the conductances via a Borel-Cantelli argument. The definitions we use instead are tailor-made to avoid this effect, and still strong enough to capture both the relevant connectivity decay in Theorem~\ref{thm:StretchedExpDecay} for $\alpha > \alpha_{\ast\ast}$ and the small probability of $\psi^\omega$-bad boxes over an appropriately sized bigger box in Proposition~\ref{prop:PsiBadUniformMesoscopic} for $\alpha < \overline{\alpha}$. The uniformity over boxes of size $L^\rho$ we require in our definitions is reminiscent of the  properties of ``very good'' boxes used for instance in~\cite{barlow2004random}.

3) In the special case in which the conductances $\omega \in \Omega_\lambda$ are i.i.d.\ under the environment measure $P$, it is well-known that the variable-speed random walk in $(\bbZ^d,\bbE_d,\omega)$ with diffusive scaling has a $P$-a.s.\ scaling limit given by a Brownian motion with covariance matrix $a^{\homo} = \sigma^2 I_d$, where $I_d$ is the $(d \times d)$-identity matrix and $\sigma^2 > 0$ (see Theorem 1.1 of~\cite{sidoravicius2004quenched}). By considering constant conductances $\omega^{(\sigma)} \equiv \frac{1}{2}\sigma^2$ on $(\bbZ^d, \bbE_d)$, it is straightforward to show that the corresponding random walk on the weighted graph $(\bbZ^d, \bbE_d,\omega^{(\sigma)})$ with diffusive scaling has the same scaling limit. It is presently an open problem to relate the parameter $\alpha_\ast^{(\sigma)}$ corresponding to the Gaussian free field with \textit{constant} conductances given by $\omega^{(\sigma)}$ to the parameter $\alpha_\ast$ from~\eqref{prop:alpha_ast}. However, by relaxing the i.i.d.~assumption, one can construct an ergodic, stationary and isotropic random environment for which $\alpha_\ast$ differs from the corresponding parameter $\alpha_\ast^{(\sigma)}$. We sketch the construction of such an environment below.
  
  Set $d \geq 4$ and consider independent Poisson point processes $(\cM_i)_{i=1,\ldots,d}$ on $\bbR$, with intensity measure $\mu \, \De \lambda^1$, governed by $P$, where $\lambda^1$ denotes the Lebesgue measure on $\bbR$ and $\mu > 0$ is chosen later. For $\cM_i = \sum_{j = 1}^\infty \delta_{x_j^{(i)}} $, $i=1,\ldots,d$, we set $A_i = \{ x_j^{(i)} \, : \, j \in \bbN\}$. Then, for $\lambda \in (0,1)$ fixed, we define the conductances for $x,y\in \bbZ^d$ with $x\sim y$ as follows:
  \begin{equation}
    \begin{cases}
      \omega_{x,y} = \lambda, &\text{if $d_\infty(\{x,y\},A_1\times\ldots \times A_d)<1$,}\\
      \omega_{x,y} = 1, &\text{otherwise,} \\ 
    \end{cases}
  \end{equation}
  Since the environment is elliptic, it is known that the corresponding variable-speed random walk converges to a Brownian motion with diffusivity $\sigma^2$ which can be bounded by
  \begin{equation}
    \lambda < \frac{1}{E_P[\omega_{0,x}^{-1}]}\leq\frac{\sigma^2}{2} \leq E_P[\omega_{0,x}] < 1,
  \end{equation}
  where $x \sim 0$ (see for instance Section 4 in~\cite{biskup2011recent}). In particular $ \alpha_\ast^{(\sigma)} > \alpha_\ast^{(\sqrt{2})}$ (the latter corresponds to the percolation threshold for the Gaussian free field with constant unit conductances). We will show below that $\alpha_\ast^{(\sqrt{2})} \geq \alpha_\ast$, thus leading to the desired conclusion. Heuristically, this is due to the fact that a path connecting any $z\in \bbZ^d$ to infinity above some level $\alpha > \alpha_\ast^{(\sqrt{2})}$ must cross arbitrarily large regions with constant unit conductances, in which the field is locally in a strongly non-percolative regime.

 Let $\kappa>0$ be a real number to be fixed later. We note that eventually every set $[n,2n]\cap A_i$ and $[-2n,-n]\cap A_i$ has a gap between consecutive points of size at least $\kappa \log n$. Indeed, for all $i=1,\ldots, d$,
  \begin{equation}
    \begin{aligned}
      &P\big([n,2n]\cap A_i \text{ does not have a gap of size at least }\kappa \log n\big) \\ &\leq P\big( \cM_i( [\kappa (k-1) \log n ,\kappa k \log n ))\geq 1 ,\, \text{ for all } \tfrac{n}{\kappa\log n} +1\leq k \leq \tfrac{2n}{\kappa\log n}\big)
      \\
      &\leq \Big(1 - \frac{1}{n^{\mu\kappa}}\Big)^{\left\lfloor n/(\kappa\log n) \right\rfloor },
    \end{aligned}
  \end{equation}
  which is summable over $n \in \bbN$ as long as 
  \begin{equation}
  \label{eq:muKappaCondition}
  \mu \kappa < 1,
\end{equation}
and similarly for $[-2n,-n] \cap A_i$.  By the Borel-Cantelli Lemma, this implies that for each $i = 1,...,d$ there exists $n_0 \in \bbN$ (depending on $\cM_1,...,\cM_n$) such that for all $n \geq n_0$ and all $1 \leq i \leq d$, $[n,2n] \cap A_i$ and $[-2n,-n] \cap A_i$ have a gap of size at least $\kappa \log n$.

In particular, one can construct an ``interface'' within the annulus $B(0,2n) \setminus B(0,n)$, such that its $\frac{\kappa}{3}\log n$-neighborhood only contains unit conductances. More precisely, for $n \geq n_0$ one can consider the set $W_n \subseteq B(0,2n) \setminus B(0,n)$ of points $x$, for which
 \begin{equation} 
 \label{eq:SlabCondition}
 \omega_{z,y} = 1, \text{ for all } z \sim y \in \bbZ^d \text{ with } |e_i \cdot (
  y - x)| \leq \tfrac{\kappa}{3}\log n , \text{ for some } 1 \leq i \leq d,
 \end{equation}
 where $(e_i \, : \, i = 1,...,d)$ stands for the canonical basis of $\bbR^d$. By construction, for $z \in \bbZ^d$ and $n$ large enough, 
 \begin{equation}
 \label{eq:MustCrossSlab}
   \begin{minipage}{0.9\textwidth}
   Every infinite nearest-neighbor path in $\bbZ^d$ starting in $z$ must cross  $W_n$. \end{minipage}
\end{equation} 
 We then set $V = B(x, \frac{\kappa}{10}\log n  )$. One can consider a probability measure $\widehat{\bbP}^\omega$, for which one has (on $\bbZ^d$)
 \begin{equation}
 \label{eq:CouplingOfFields}
 \phi^\omega = \psi^V + \xi^{\omega,V}, \qquad \text{and} \qquad \phi = \psi^V + \xi^{V},
 \end{equation}
  where $\psi^V$ is a Gaussian free field with zero boundary conditions on $\partial V$ and $\xi^V$, $\xi^{\omega,V}$ are independent Gaussian fields obtained by the Markov property for the Gaussian free field within $V$, with conductances $\omega$ and $1$, respectively (see~\eqref{eq:HarmonicAverageDef}--\eqref{eq:DomainMP}). Note that $\phi^\omega$ under $\widehat{\bbP}^\omega$ has the same law as $\varphi$ under $\bbP^\omega$, whereas $\phi$ under $\widehat{\bbP}^\omega$ is a Gaussian free field with unit conductances. Moreover, we consider the events 
 \begin{equation}
 \cG^\omega_\delta = \bigg\{ \sup_{y \in B(x,\frac{\kappa}{100}\log n ) } |\xi^{\omega,V}_y| \leq \frac{\delta}{2} \bigg\}, \qquad \cG_\delta =\bigg\{ \sup_{y \in B(x,\frac{\kappa}{100}\log n ) } |\xi^{V}_y| \leq \frac{\delta}{2} \bigg\}.
\end{equation}  
   For $\alpha > \alpha_\ast^{(\sqrt{2})} + \delta$, it follows that  
  \begin{equation}
  \label{eq:CounterexampleConnectionBound}
  \begin{split}
  \bbP^\omega\left[x \stackrel{ \geq \alpha }{ \longleftrightarrow} \partial B(x, \tfrac{\kappa}{100}\log n) \right] & \leq \widehat{\bbP}^\omega\left[x \stackrel{\{\phi^\omega \geq \alpha\}}{ \longleftrightarrow} \partial B(x, \tfrac{\kappa}{100}\log n), \cG^\omega_\delta \cap \cG_\delta \right] \\
  & + \widehat{\bbP}^\omega\left[ (\cG^\omega_\delta)^c\right] + \widehat{\bbP}^\omega\left[ (\cG_\delta)^c\right] \\
  & \stackrel{\eqref{eq:CouplingOfFields}}{\leq} \widehat{\bbP}^\omega\left[x \stackrel{\{\phi \geq \alpha- \delta\}}{ \longleftrightarrow} \partial B(x, \tfrac{\kappa}{100}\log n)\right] \\
  & + \widehat{\bbP}^\omega\left[ (\cG^\omega_\delta)^c\right] + \widehat{\bbP}^\omega\left[ (\cG_\delta)^c\right].
  \end{split}
  \end{equation}
 One can use a version of~\eqref{eq:DecouplingErrorBound} (see also Remark~\ref{rem:GeneralizedDecoupling}) to infer a bound for the last two summands on the right-hand side in~\eqref{eq:CounterexampleConnectionBound}. For the first summand, we now use the exponential decay of the connection probability (see Theorem 2.1 in~\cite{popov2015decoupling}) in combination with the sharpness of the phase transition for the Gaussian free field with constant conductances (see~\cite{duminil2020equality}), hence: 
  \begin{equation}
  \begin{split}
  \bbP^\omega\left[x \stackrel{ \geq \alpha }{ \longleftrightarrow} \partial B(x, \tfrac{\kappa}{100}\log n) \right] & \leq C (\kappa\log n)^{d-1} \exp\big(- c^\star \kappa \log n \big)  \\
  & + C (\kappa\log n)^d \exp\big( - c' \delta^2(\kappa \log n)^{d-2}\big).
  \end{split}
  \end{equation}
We finally fix $\kappa$ large enough such that $c^\star \kappa > d$ and $\mu$ small enough such that~\eqref{eq:muKappaCondition} holds. Using a union bound one has for $z \in \bbZ^d$ (for $n$ large enough)
\begin{equation}
\begin{split}
\bbP^\omega\left[z \stackrel{\geq \alpha}{ \longleftrightarrow} \infty \right] & \stackrel{\eqref{eq:MustCrossSlab}}{\leq} C  n^{d} \cdot \sup_{x \in W_n }  \bbP^\omega\left[x \stackrel{ \geq \alpha }{ \longleftrightarrow} \partial B(x, \tfrac{\kappa}{100}\log n) \right] \\
& ( \rightarrow 0 \text{ as } n \rightarrow \infty, \text{ by~\eqref{eq:CounterexampleConnectionBound}}). 
\end{split}
\end{equation}
 It follows that for $\alpha > \alpha_\ast^{(\sqrt{2})} + \delta$, $\bbP^\omega [E^{\geq \alpha}\mbox{ contains an infinite cluster}\,] = 0$, so $\alpha_\ast \leq \alpha_\ast^{(\sqrt{2})} + \delta$. Since $\delta > 0$ was arbitrary, we have $\alpha_\ast \leq \alpha_\ast^{(\sqrt{2})} ( < \alpha_\ast^{(\sigma)})$.

4) Beside the parameter $\alpha_\ast$ for the Gaussian free field with random conductances $\omega \in \Omega_\lambda$ one can also study the percolation threshold $\widehat{\alpha}_\ast$ for the corresponding \textnormal{annealed} model. The latter is obtained by averaging over the environment $\omega$, namely by considering the field $(\varphi_x)_{x \in \bbZ^d}$ under the annealed measure $P \times \bbP^\omega$. Since $\Omega_\lambda \ni \omega \mapsto \bbP^\omega [E^{\geq \alpha}\mbox{ contains an infinite cluster}]$ is $P$-a.s.~constant (see the proof of Proposition~\ref{prop:alpha_ast}), one has the equality $\widehat{\alpha}_\ast = \alpha_\ast$, and therefore also a non-trivial phase transition. Remarkably, the field $\varphi$ is stationary under the annealed measure, but not Gaussian except in the case of constant conductances.
\end{rem}

\section{Solidification estimates for random walks}
\label{sec:Solidification}
In this section we develop uniform estimates on the absorption of a random walk on $(\mathbb{Z}^d,\bbE_d)$, $d \geq 3$, equipped with uniformly elliptic weights $\omega \in \Omega_\lambda$ by porous interfaces $\Sigma$ surrounding the discrete blow-up $A_N$ of a compact set $A \subseteq \mathbb{R}^d$ with non-empty interior. 

The main result comes in Theorem~\ref{thm:SolidificationTheorem} below, and constitutes an adaptation of Theorem 3.1 of~\cite{nitzschner2017solidification} to a discrete set-up with non-uniform weights. These estimates and  related uniform controls on the capacity of the porous interfaces, see Corollary~\ref{cor:CapacityDirichletSolidification}, are pivotal for the proofs of the upper bounds on the probability of the disconnection events presented in Theorems~\ref{thm:UpperBound} and~\ref{thm:EntropicRepulsion} of Section~\ref{sec:UpperBounds}. Importantly, as in~\cite{nitzschner2017solidification}, no convexity assumption is made on the set $A$. In particular, the comparison of the capacities $\capa^\omega(\Sigma)$ and $\capa^\omega(A_N)$ as they appear in Corollary~\ref{cor:CapacityDirichletSolidification} is more delicate since projection arguments which facilitate such bounds in the case of convex sets $A$ are not available.
Similarly as in~\cite{nitzschner2017solidification}, our approach will involve considerations of relative volumes in boxes of size $2^\ell$, at multiple scales. Due to the discrete nature of our problem, this can only be performed up to a certain precision. Moreover, quenched lower bounds for the heat kernel on the weighted graph $(\bbZ^d, \bbE_d, \omega)$ are instrumental to overcome the lack of a scaling invariance that one has in the case of Brownian motion.  

Let us now introduce some notation. As in the introduction, we let $\lambda \in (0,1)$ which will be kept fixed and let $\omega \in \Omega_\lambda = [\lambda,1]^{\bbE_d}$ be real, uniformly elliptic conductances on the edges of $(\bbZ^d, \bbE_d)$.
Recall that all constants may depend implicitly on $\lambda$. We recall the notation for the continuous-time random walk among the conductances $(\omega_e)_{e \in \bbE_d}$ in~\eqref{eq:ContTimeRW} and below, and introduce for $r \in \mathbb{N}$ the $(\cF_t)$-stopping time
\begin{equation}
\label{eq:tau_Def}
\tau_r = \inf\{ t \geq 0 \, : \, |X_t - X_0|_\infty \geq r \}.
\end{equation}
We consider now a bounded, non-empty set $U_0 \subseteq \bbZ^d$ and its complement $U_1$. The boundary $S = \partial U_0 = \partial_{\text{in}} U_1$ is a bounded, non-empty subset of $\bbZ^d$. We are interested in measuring a certain local density of $U_0$ in boxes of dyadic scale. For this, we consider a non-negative integer $\ell$ and $x\in \bbZ^d$ and define the local density measure by
\begin{equation}\label{eq:densitymeausre}
    \mu_{x,\ell} (y) = \frac{1}{|B(x,2^\ell)|} \mathbbm{1}_{B(x,2^\ell)}(y), \qquad y\in \bbZ^d.
\end{equation} 

The local density functions associated with $U_1$ are defined as

\begin{equation}
\label{eq:DefSigmaEll}
\begin{split}
\sigma_{\ell}(x) & = \mu_{x,\ell}(U_1) \left( = \tfrac{|B(x,2^{\ell}) \cap U_1| }{|B(x,2^{\ell})|} \right), \\
\widetilde{\sigma}_{\ell}(x) & = \mu_{x,\ell+2}(U_1) \left( = \tfrac{|B(x,4 \cdot 2^{\ell}) \cap U_1| }{|B(x,4 \cdot 2^{\ell})|} \right).
\end{split}
\end{equation}

For a given function $f :\bbZ^d \to \bbR$ we write $(f)_{B(x,2^\ell)}$ for the average of $f$ on $B(x,2^\ell)$. Given a bounded, non-empty set $A \subseteq \bbZ^d$ and $\ell_{\ast} \geq 0$, we define
\begin{equation}
\label{eq:U0l_def}
\cU_{\ell_\ast, A} = \{ U_0 \subseteq \bbZ^d \text{ bounded} \, : \, \sigma_\ell(x) \leq \tfrac{1}{2}, \text{ for all } x \in A, \ell \leq \ell_\ast \}.
\end{equation}
Intuitively, any $U_0 \in \cU_{\ell_\ast, A}$ is such that $A \subseteq \bbZ^d$ is ``well inside'' $U_0$. For a given $U_0 \in \cU_{\ell_\ast, A}$ as in~\eqref{eq:U0l_def}, and $\omega \in \Omega_\lambda$ consider $\varepsilon \in \mathbb{N}$ and $\chi \in (0,1)$, and define the \emph{porous interfaces}
\begin{equation}
\label{eq:ClassOfPorousInterfaces}
\cS^\omega_{U_0, \varepsilon, \chi} = \{ \Sigma \subseteq \bbZ^d \text{ bounded} \, : \, P^\omega_x[H_\Sigma < \tau_\varepsilon] \geq \chi, \text{ for all } x \in S \}
\end{equation}
(recall~\eqref{eq:tau_Def} for the definition of $\tau_\varepsilon$). The value $\varepsilon$ governs the distance of the porous interfaces from $S (= \partial U_0)$, while $\chi$ is a measure for the strength at which it is felt for a random walk among given conductances $\omega$. 

We now present the uniform controls (``solidification estimates'') alluded to above in the following main theorem.
\begin{theorem}
\label{thm:SolidificationTheorem}
Let $A \subseteq \bbR^d$ be compact with non-empty interior, $\chi \in (0,1)$ and $(a_N)_{N \geq 0}$ a sequence of positive reals with $a_N \rightarrow 0$ as $N \rightarrow \infty$. It holds that
\begin{equation}
\label{eq:SolidificationStatement}
\lim_{N \rightarrow \infty} \sup_{\varepsilon/ 2^{\ell_\ast} \leq a_N} \sup_{U_0 \in \cU_{\ell_\ast, A_N}} \sup_{\omega \in \Omega_\lambda} \sup_{\Sigma \in \cS^\omega_{U_0, \varepsilon, \chi} } \sup_{x \in A_N} P^\omega_x[H_\Sigma = \infty] = 0.
\end{equation}
Moreover, one has that
\begin{equation}
\label{eq:Solidification_From_Outside}
\lim_{N \rightarrow \infty} \sup_{\varepsilon/ 2^{\ell_\ast} \leq a_N} \sup_{U_0 \in \cU_{\ell_\ast, A_N}} \sup_{\omega \in \Omega_\lambda} \sup_{\Sigma \in \cS^\omega_{U_0, \varepsilon, \chi} } \sup_{x \in \bbZ^d} \big(P^\omega_x[H_\Sigma = \infty] - P^\omega_x[H_{A_N} = \infty]  \big) = 0.
\end{equation}
\end{theorem}
\begin{figure}
\begin{center}
\includegraphics[scale=.5]{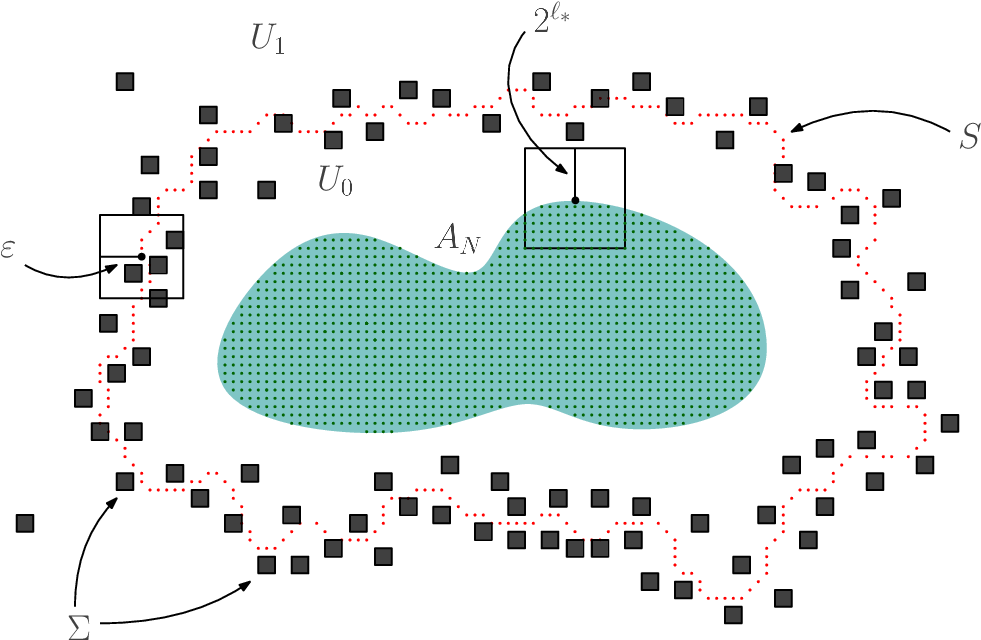}
\end{center}
\caption{An illustration of a $U_0$ in $\cU_{\ell_\ast,A_N}$ and $\Sigma$ in $\cS^\omega_{U_0,\varepsilon,\chi}$.}
\label{fig:interface}
\end{figure}

An immediate consequence of this theorem is the  following capacity lower bound that will enter the proofs in the following sections. 

\begin{cor}
\label{cor:CapacityDirichletSolidification}
Let $A \subseteq \bbR^d$ be compact with non-empty interior, $\chi \in (0,1)$ and $(a_N)_{N \geq 0}$ a sequence of positive real numbers with $a_N \rightarrow 0$ as $N \rightarrow \infty$. Then, one has
\begin{equation}
\label{eq:Capacity_Lower_Bound}
\liminf_{N \rightarrow \infty} \inf_{\varepsilon/2^{\ell_\ast} \leq a_N}  \inf_{U_0 \in \cU_{\ell_\ast, A_N}}  \inf_{\omega \in \Omega_\lambda}\inf_{\Sigma \in \cS^\omega_{U_0,\varepsilon, \chi}} \frac{\capa^\omega(\Sigma)}{\capa^\omega(A_N)} \geq 1,
\end{equation}
and 
\begin{equation}
\label{eq:Dirichlet_Form_Bound}
\begin{split}
\varlimsup_{N \rightarrow \infty} \sup_{\varepsilon/ 2^{\ell_\ast} \leq a_N} \sup_{U_0 \in \cU_{\ell_\ast, A_N}} \sup_{\omega \in \Omega_\lambda} \sup_{\Sigma \in \cS^\omega_{U_0, \varepsilon, \chi} } & \frac{1}{N^{d-2}} \big[ \cE^\omega( h^\omega_{A_N} - h^\omega_\Sigma) \\
& - (\capa^\omega(\Sigma) - \capa^\omega(A_N)) \big]  \leq 0.
\end{split}
\end{equation}
\end{cor}
\begin{proof}
We first show~\eqref{eq:Capacity_Lower_Bound}. One has that for any $N \geq N_0$, $\varepsilon / 2^{\ell_\ast} \leq a_N$, $U_0 \in \cU_{\ell_\ast, A_N}$, $\Sigma \in \cS^\omega_{U_0, \varepsilon, \chi}$:
\begin{equation}
\begin{split}
\capa^\omega(\Sigma) & \geq \sum_{y \in \bbZ^d} P^\omega_y [H_{A_N} < \infty] e^\omega_\Sigma(y) \geq \inf_{x \in A_N} P^\omega_x[H_\Sigma < \infty] \capa^\omega(A_N),
\end{split}
\end{equation}
using~\eqref{eq:LastExDecomp} and the symmetry of $g^\omega(\cdot,\cdot)$. Dividing by $\capa^\omega(A_N)$ (for large enough $N$) and taking the respective infima on the right-hand side readily yields the claim upon using Theorem~\ref{thm:SolidificationTheorem}. 

We now turn to the proof of~\eqref{eq:Dirichlet_Form_Bound}. By the bilinearity of the Dirichlet form and~\eqref{eq:Capacity_and_Dirichlet}, it suffices to prove that
\begin{equation}
\label{eq:LiminfDirichlet}
\liminf_{N \rightarrow \infty} \inf_{\varepsilon/2^{\ell_\ast} \leq a_N}  \inf_{U_0 \in \cU_{\ell_\ast, A_N}}  \inf_{\omega \in \Omega_\lambda}\inf_{\Sigma \in \cS^\omega_{U_0,\varepsilon, \chi}}  \frac{1}{N^{d-2}} \big( \cE^\omega(h^\omega_{A_N}, h^\omega_\Sigma) - \capa^\omega(A_N) \big) \geq 0,
\end{equation}
see also the proof of Lemma 2.2 of~\cite{chiarini2019entropic} for a similar argument. To see this, we utilize the discrete Gauss-Green identity, which yields that
\begin{equation}
\begin{split}
\cE^\omega(h^\omega_{\Sigma}, h^\omega_{A_N} ) & = \sum_{x \in \bbZ^d} h^\omega_\Sigma(x) e^\omega_{A_N}(x)  \geq \inf_{x \in A_N} P^\omega_x[H_\Sigma < \infty] \capa^\omega(A_N). 
\end{split}
\end{equation}
This establishes~\eqref{eq:LiminfDirichlet} upon using~\eqref{eq:SolidificationStatement} and $\capa^\omega(A_N) \leq c(A) N^{d-2}$, by~\eqref{eq:QuenchedBoxCapacityEstimate}. 
\end{proof}

We now turn to the proof of Theorem~\ref{thm:SolidificationTheorem}, which is given in three steps. In a first step, we study in more detail how a random walk enters sets of well-balanced local density. Thereafter, we introduce a discrete ``resonance set'' that is hard to avoid for random walks starting within $U_0$. In a last step, we essentially use the resonance set as a substitute for a Wiener-type criterion to show in an asymptotic sense, that the ``porous interfaces'' $\Sigma$ cannot be avoided by a random walk starting in $A_N$. 

\subsection{Local density functions}

In this subsection we fix $U_0 \subseteq \bbZ^d$ non-empty and bounded and denote its complement in $\bbZ^d$ as $U_1$. We collect some properties of the local density functions $\sigma_{\ell}$ and $\widetilde{\sigma}_{\ell}$, $\ell \geq 0$, associated with $U_1$ (see~\eqref{eq:DefSigmaEll}). 

The following two lemmas are discrete versions of Lemma 1.1 and Lemma 1.2 in~\cite{nitzschner2017solidification}, respectively. Their proofs proceed as in the continuum case and are omitted.  Lemma~\ref{lem:Lipschitz} establishes a Lipschitz property of the local density and relates $\sigma_\ell$ to the average of $\sigma_{\ell'}$ on $B(x,2^\ell)$ when $\ell' < \ell$. Lemma~\ref{lem:Dichotomy} shows that when $\ell' < \ell$ and $\sigma_{\ell'}$ has an average $\beta'$ in the box $B(x,2^\ell)$, then either $\sigma_{\ell'}$ has both values bigger and smaller than $\beta'$ by a certain amount, or remains close to $\beta'$ in a sizable part of $B(x,2^\ell)$. This property is then used in the proof of Proposition~\ref{prop_1_step}. 
\begin{lem} 
\label{lem:Lipschitz}
For $x, y \in\bbZ^d$ and $\ell \in \mathbb{N}$, one has
\begin{equation}\label{eq:lip}
    |\sigma_\ell(x) -\sigma_\ell(x+y)|\leq 2^{-\ell} |y|_1.
\end{equation}    
Moreover, for any $\ell' <\ell$ non-negative integers and $c_0 = d 2^{d-1}$ one has
\begin{equation}\label{eq:average}
    |\sigma_\ell(x) - (\sigma_{\ell'})_{B(x,2^{\ell})}| \leq c_0\, 2^{\ell' - \ell}
\end{equation}
(recall the notation $(\cdot)_{B(x,2^\ell)}$ for the average on $B(x,2^\ell)$, see above~\eqref{eq:U0l_def}).
\end{lem}

\begin{lem} 
\label{lem:Dichotomy}
Consider $x\in\bbZ^d$ and $0 \leq \ell'<\ell$ integers and set $\beta' = (\sigma_{\ell'})_{B(x,2^\ell)}$. Then, for all $0\leq \delta \leq \beta'\wedge (1-\beta')$ at least one of the following holds true:
\begin{equation}
\label{eq:dichotomy}
    \left\{ 
    \begin{aligned}
        &\text{i) } \mu_{x,\ell}(\{\sigma_{\ell'} > \beta' + \delta \})\geq \frac{\delta}{2} \mbox{ and } \mu_{x,\ell}(\{\sigma_{\ell'} < \beta' - \delta\})\geq \frac{\delta}{2}, \\
        &\text{ii) } \mu_{x,\ell}(\{\beta'-\delta \leq \sigma_{\ell'} \leq \beta' + \delta\})\geq \frac{1}{4}-\frac{\delta}{2}.
    \end{aligned}    
    \right.
\end{equation}
\end{lem}

As a next step, we give a version of Propositions 1.3 and 1.4 of~\cite{nitzschner2017solidification} for our purposes. We show that for $\ell' < \ell$ big enough, when $\sigma_{\ell'}$ has average $\beta'$ over $B(x,2^{\ell})$ not too close to $0$ or $1$, a random walk starting from $x$ has a non-degenerate probability to enter a region where $\sigma_{\ell'}$ is close to $\beta'$, before exiting $B(x,2^\ell)$. 
 \begin{prop}
\label{prop_1_step}
    Let $\delta > 0$ and $x \in \bbZ^d$. There exists $\ell_{\min}(\delta) \in \bbN$ such that for every $\ell, \ell' \in \bbN$ with $\ell_{\min}(\delta) \leq \ell' < \ell$ and $\delta \leq \beta' \wedge (1-\beta') \wedge \tfrac{1}{4}$, where $\beta' = (\sigma_{\ell'} )_{B(x,2^\ell)}$, the following holds
    \begin{equation}
    \inf_{\omega \in \Omega_\lambda} P^\omega_x [H_{ \{\sigma_{\ell'} \in [\beta' - \delta, \beta' + \delta] \} } < \tau_{2^\ell} ] \geq c_6(\delta).
    \end{equation}
\end{prop}
\begin{proof}

Given $\delta > 0$ we take $\delta' > 0$ solely depending on $\delta$ such that for all $\ell \geq \overline{\ell} (= \overline{\ell}(\delta)$),
\begin{equation}
\mu_{x,\ell}(\{ y \in \bbZ^d \, : \, |y-x|_\infty > 2^\ell (1-\delta') \}) \leq \tfrac{\delta}{4}.
\end{equation} 
We then choose $\tilde{\ell} (=\tilde{\ell}(\delta))$ large enough such that (with $\vartheta = \delta'$ and $R \geq 2^{\tilde{\ell}}$), the heat kernel lower bound~\eqref{eq:HK_bound} holds. We define
\begin{equation}
\label{eq:MinScaleDefinition}
\ell_{\text{min}}(\delta) = \tilde{\ell} \vee \overline{\ell} \vee \min \{ \ell \in \bbN \,  : \, 2^{-\ell} \leq \tfrac{\delta}{8}\}.
\end{equation}
The third member of the maximum is introduced for later use. Let us now assume we are in situation i) of the dichotomy~\eqref{eq:dichotomy}. In this case, we have that 
\begin{equation}
\label{eq:nondegeneracy_betadelta}
    \left\{ 
    \begin{aligned}
       & \mu_{x,\ell}(\{\sigma_{\ell'} > \beta' + \delta \} \cap B(x,2^\ell(1-\delta'))) \geq \tfrac{\delta}{4} , \\
        &\mu_{x,\ell}(\{\sigma_{\ell'} < \beta' - \delta \} \cap B(x,2^\ell(1-\delta'))) \geq \tfrac{\delta}{4}.
    \end{aligned}    
    \right.
\end{equation}
Then, using the killed heat kernel bounds, we see that for $\omega \in \Omega_\lambda$ (since $\ell \geq \tilde{\ell}$):
\begin{equation}
\begin{split}
P^\omega_x & [H_{\{\sigma_{\ell'} \geq \beta' + \delta \} \cap B(x,2^\ell(1-\delta')) } < T_{B(x,2^\ell)}] \\
 & \stackrel{\eqref{eq:KilledHeatKernel}}{\geq}  \sum_{z \in \{\sigma_{\ell'} \geq \beta' + \delta \} \cap B(x,2^\ell(1-\delta')) } \omega_z q^{\omega}_{ 4^{\ell}, B(x,2^\ell) }(x,z) 
 \geq C(\delta),
\end{split}
\end{equation}
combining~\eqref{eq:nondegeneracy_betadelta}, \eqref{eq:HK_bound} and $\omega_z \geq 2d \lambda$ in the second step. A similar argument shows that for all $y \in B(x,2^\ell(1-\delta'))$:
\begin{equation}
\begin{split}
P^\omega_y[H_{\{\sigma_{\ell'} \leq \beta' - \delta \} \cap B(x,2^\ell(1-\delta')) } < T_{B(x,2^\ell)}]
\geq C(\delta).
\end{split}
\end{equation}
By Lipschitz continuity of $\sigma_{\ell'}$, see~\eqref{eq:lip}, and the fact that $\ell > \ell' \geq \ell_{\text{min}}(\delta)$ we see (using the strong Markov property) that a random walk starting in $x$ has a probability of at least $C(\delta)^2$ to visit first $\{\sigma_{\ell'} > \beta' + \delta \}$ and then $\{\sigma_{\ell'} < \beta' - \delta \}$ before exiting $B(x,2^{\ell})$ and thus reaching $\{ \sigma_{\ell'} \in [\beta' - \delta, \beta' + \delta]\}$. On the other hand, if in the dichotomy~\eqref{eq:dichotomy} ii) holds, we have
\begin{equation}
\mu_{x,\ell}(\{ \sigma_{\ell'} \in [\beta' - \delta,\beta' + \delta] \}) \geq \tfrac{1}{4}- \tfrac{\delta}{2} \geq \tfrac{1}{8}.
\end{equation}
We can then argue in the same fashion as above that 
\begin{equation}
P^\omega_x[H_{ \{\sigma_{\ell'} \in [\beta' - \delta, \beta' + \delta] \} } < \tau_{2^\ell}] \geq c.
\end{equation}
Collecting the two bounds and taking the infimum over $\omega \in \Omega_\lambda$ yields the claim. 
\end{proof}

We now take well-separated decreasing scales $2^{\ell_0} > 2^{\ell_1} > ... > 2^{\ell_J}$ (see~\eqref{eq:separation}), and in a similar spirit to Proposition 1.4 of~\cite{nitzschner2017solidification}, we show that for a random walk starting in $x\in \bbZ^d$ with $\sigma_{\ell_0}(x)$ close to $\tfrac{1}{2}$, there is a non-degenerate probability to enter a region where all $\widetilde{\sigma}_{\ell_j}$, $0 \leq j \leq J$ lie in the fixed interval $[\widetilde{\alpha}, 1 - \widetilde{\alpha}]$, where $\widetilde{\alpha}$ is defined in~\eqref{eq:DefinitionAlpha} below, before the walk moves at a distance $\tfrac{3}{2} 2^{\ell_0}$.

Due to the discrete nature of our set-up, we are again forced to consider $\ell_J \geq \ell_{\text{min}}(\delta)$, defined in~\eqref{eq:MinScaleDefinition}, for some $\delta > 0$ depending on $J$. More precisely, let $J \geq 1$, $c_0 = d 2^{d-1}$ and define 
\begin{equation}
\label{eq:LJdef}
L(J) = \min \{ L \geq 5 \, : \, c_0 2^{-L} \leq \tfrac{1}{200J} \}.
\end{equation}
We then look at separated scales $2^{\ell_0} > 2^{\ell_1} > ... > 2^{\ell_J}$ with $\ell_J \geq \ell_{\text{min}}((200J)^{-1})$. 
\begin{equation}
\label{eq:separation}
\ell_{j+1} \leq \ell_j - L(J),\quad 0 \leq j < J.
\end{equation}
Next, we introduce the increasing sequence of compact sub-intervals of $(0,1)$, namely
\begin{equation}
I_j = \Big[\tfrac{1}{2} - \tfrac{j+1}{100J}, \tfrac{1}{2} + \tfrac{j+1}{100J} \Big],\quad 0 \leq j \leq J
\end{equation}
and the non-decreasing sequence of stopping times
\begin{equation}
\gamma_0 = H_{ \{\sigma_{\ell_0} \in I_0 \} }, \qquad \gamma_{j+1} = \gamma_j + H_{ \{\sigma_{\ell_{j+1}} \in I_{j+1} \} } \circ \theta_{\gamma_j},\quad 0 \leq j < J
\end{equation}
($(\theta_t)_{t \geq 0}$ denotes the family of canonical shift operators on the space of right-continuous functions from $[0,\infty)$ to $\bbZ^d$ with finitely many jumps on every finite interval).
\begin{prop}
\label{prop_multiscale_descent}
Assume $J \geq 1$, and that $\ell_j$, $0 \leq j \leq J$, satisfy the separation condition~\eqref{eq:separation} as well as $\ell_J \geq \ell_{\min}((200J)^{-1})$. We let $\cC$ stand for the event 
\begin{equation}
\cC = \{ \gamma_0 = 0 \} \cap \bigcap_{0 \leq j < J} \{ \gamma_{j+1} < \gamma_j + \tau_{2^{\ell_j}} \circ \theta_{\gamma_j} \}.
\end{equation}
We have for every $x \in \bbZ^d$ such that $\sigma_{\ell_0}(x) \in [\tfrac{1}{2} - \tfrac{1}{2^{\ell_{\min}((200J)^{-1})}}, \tfrac{1}{2} + \tfrac{1}{2^{\ell_{\min}((200J)^{-1})}}]$ that 
\begin{equation}
\inf_{\omega \in \Omega_\lambda} P^\omega_x[\cC] \geq c_7(J). 
\end{equation}
On the event $\cC$, we have 
\begin{equation}
\sup \{|X_s - X_{\gamma_j} |_\infty : \gamma_j \leq s \leq \gamma_J \}   \leq  \tfrac{3}{2} 2^{\ell_j}, \quad 0 \leq j < J,
\end{equation}
\begin{equation}
\label{eq:DefinitionAlpha}
\widetilde{\sigma}_{\ell_j}(X_{\gamma_J}) \in [\widetilde{\alpha}, 1- \widetilde{\alpha}], 0 \leq j \leq J, \text{ where } \widetilde{\alpha} = \tfrac{1}{3}4^{-d}.
\end{equation}
\end{prop}
\begin{proof}
We use the notation $\delta = \frac{1}{200J} (\leq \frac{1}{4})$, and use Proposition~\ref{prop_1_step} repeatedly with this choice of $\delta$. More precisely, we define for $0 \leq j \leq J$ the event 
\begin{equation}
\cC_j = \{ \gamma_0 = 0 \} \cap \bigcap_{0 \leq i < j} \{ \gamma_{i+1} < \gamma_i + \tau_{2^{\ell_i}} \circ \theta_{\gamma_i} \}
\end{equation}
(with $\cC_J = \cC$) and aim at showing via induction that for every $\omega \in \Omega_\lambda$, $P^\omega_x[\cC_j] \geq c_6(\delta)^j$. Since $2^{-\ell_0} \leq 2^{-\ell_{\text{min}}((200J)^{-1})} \leq \tfrac{1}{8 \cdot 200J}$, we immediately see that $\sigma_{\ell_0}(x) \in [\tfrac{1}{2} - \tfrac{1}{200J}, \tfrac{1}{2} + \tfrac{1}{200J}] = I_0$, implying that $\gamma_0 = 0$, $P^\omega_x$-a.s., so the induction hypothesis is true in the case $j = 0$. We now perform induction in $j$. To this end, let $\beta'_{j+1} = (\sigma_{\ell_{j+1}})_{B(X_{\gamma_j},2^{\ell_j}) }$. On $\cC_j$, one has (by definition) that $\sigma_{\ell_{j}}(X_{\gamma_j}) \in I_j$. Next, by~\eqref{eq:average} and~\eqref{eq:separation} we have that 
\begin{equation}
\begin{split}
& |\sigma_{\ell_j}(X_{\gamma_j}) - \beta'_{j+1}| \leq c_0 2^{-L(J)} \leq \tfrac{1}{200J}, \text{ thus} \\
& \beta'_{j+1} \in \left[\tfrac{1}{2} - \tfrac{j+1}{100} - \tfrac{1}{200J},\tfrac{1}{2} +\tfrac{j+1}{100} + \tfrac{1}{200J}  \right],
\end{split}
\end{equation}
which in turn implies that $[\beta'_{j+1} - (200J)^{-1}, \beta'_{j+1} + (200J)^{-1}] \subseteq I_{j+1}$. Applying the strong Markov property at time $\gamma_j$, we obtain that 
\begin{equation}
P^\omega_x[\cC_{j+1}] = E^\omega_x[\mathbbm{1}_{\cC_j} P^\omega_{X_{\gamma_j}} [H_{ \{\sigma_{\ell_{j+1}} \in I_{j+1} \} } < \tau_{2^{\ell_j}}  ]] \geq c_6(\delta)^{j+1}.
\end{equation}
Here we used that $\ell_{\text{min}}((200J)^{-1}) \leq \ell_{j+1} < \ell_{j}$ so that Proposition~\ref{prop_1_step} can be applied. Taking the induction up to $J$, we finally find $\inf_{\omega \in \Omega_\lambda} P^\omega_x[\cC] \geq c_6(\tfrac{1}{200J})^J = c_7(J)$. We now turn to the rest of the proof. First, by the triangle inequality, one has on the event $\cC$: 
\begin{equation}
\begin{split}
\sup \{|X_s - X_{\gamma_j} |_\infty : \gamma_j \leq s \leq \gamma_J \} & \leq 2^{\ell_j} + 2^{\ell_{j+1}} + ... + 2^{\ell_{J-1}} \\
& \leq 2^{\ell_j} \sum_{m \geq 0} 2^{-m L(J)} < \tfrac{3}{2} 2^{\ell_j}, 
\end{split}
\end{equation}
using~\eqref{eq:separation} in the penultimate step. This last part in particular implies that on $\cC$, for any $0 \leq j \leq J$, $B(X_{\gamma_j}, 2^{\ell_j}) \subseteq B(X_{\gamma_J}, 4 \cdot 2^{\ell_j})$, and $\sigma_{\ell_j}(X_{\gamma_j}) \in I_j \subseteq [\tfrac{1}{3}, \tfrac{2}{3}]$, and ultimately $\widetilde{\sigma}_{\ell_j}(X_{\gamma_J}) \in [\tfrac{1}{3}4^{-d}, 1 - \tfrac{1}{3}4^{-d}] = [\widetilde{\alpha}, 1- \widetilde{\alpha}]$. 
\end{proof}
\subsection{Resonance set}

In the present subsection, we define a discrete analogue of the resonance set of~\cite[Section 2]{nitzschner2017solidification} associated with a finite subset $A \subseteq \bbZ^d$ which is characterized by the presence of at least $J$ among a collection of local densities $\widetilde{\sigma}_\ell$ of $U_1 ( = \bbZ^d \setminus U_0)$ attaining non-degenerate values in $[\widetilde{\alpha}, 1- \widetilde{\alpha}]$. Crucially, we will show that for a simple random walk on the weighted graph $(\bbZ^d, \bbE^d, \omega)$ starting in $A$, the resonance set attached to $U_0$ is visited with high probability, provided the set $U_0$ is chosen such that $A$ is ``well inside'' $U_0$ as measured by local density functions $\sigma_\ell$ being small. These bounds are uniform in $\omega \in \Omega_\lambda$. 

The presentation and set-up of the proof is very similar to the case of Brownian motion, and the main difficulty  is again an instance of the lack of arbitrary precision in the discrete set-up, which forces certain constraints on the definition of the resonance set. Roughly speaking, we have to ensure that the objects under consideration pertain to scales above $\ell_{\text{min}}((200J)^{-1})$ in order for the results of the previous section to be applicable. 

Let $d \geq 3$ and $U_0 \subseteq \bbZ^d$ be bounded, non-empty with associated local densities $\sigma_\ell$, $\widetilde{\sigma}_\ell ( = \sigma_{\ell +2 })$, see~\eqref{eq:DefSigmaEll}. We will be interested in a scale $\ell_\ast$ controlling from above the scales under consideration, as well as the distance between $A$ and $U_1 (= \bbZ^d \setminus U_0)$. Moreover, $J \geq 1$ will control the ``strength'' of resonance, $L \geq L(J)$ (chosen as in~\eqref{eq:LJdef}) the separation of scales and $I$ the number of scales inspected. Let
\begin{align}
\ell_0 & = \sup \{ \ell \in (J+1)L \bbN \, : \, \ell \leq \ell_\ast \}, \\
\cA_\ast & = \{ \ell \in L\bbN \, : \, \ell_0 \geq \ell > \ell_0 - I(J+1) L \}, \\
\cA &= \{ \ell \in L(J+1)\bbN \, : \, \ell_0 \geq \ell > \ell_0 - I(J+1) L \}. 
\end{align}
We say that $\ell_\ast$ is \emph{$(I,J,L)$-compatible}, if with $\ell_0$ as above, we have
\begin{equation}
\label{eq:CompatibilityCond}
\ell_0 - (I+1)(J+1) L > \ell_{\text{min}}((200J)^{-1}),
\end{equation}
with $\ell_{\text{min}}((200J)^{-1})$ defined in~\eqref{eq:MinScaleDefinition} with $\delta  = \frac{1}{200J}$. This compatibility condition will ensure that the bounds from the previous subsection, in particular Proposition~\ref{prop_multiscale_descent}, apply when inspecting all relevant scales. Moreover this implies that $|\cA_\ast| = (J+1)I$ and $|\cA| = I$. For a given choice of $I, J \geq 1$, $L \geq L(J)$, $U_0$ and $\ell_\ast$ $(I,J,L)$-compatible, the \emph{resonance set} is now defined as 
\begin{equation}
\text{Res} = \text{Res}(U_0,I,J,L, \ell_\ast) = \{ x \in \bbZ^d \, : \, \sum_{\ell \in \cA_\ast} \mathbbm{1}_{\{ \widetilde{\sigma}_\ell(x) \in [\widetilde{\alpha}, 1 - \widetilde{\alpha}] \} }\geq J \},
\end{equation}
which is a bounded (possibly empty) subset of $\bbZ^d$. Note that the resonance set does not depend on the choice of $\omega \in \Omega_\lambda$. 

We now state the main result of this subsection (recall the definition of $\cU_{\ell_\ast,A}$ from~\eqref{eq:U0l_def}):
\begin{theorem}
\label{thm:Resonance}
Let $J, I \geq 1$ $L \geq L(J)$, $A \subseteq \bbZ^d$ non-empty, and 
\begin{equation}
\Phi_{J,I,L,A} = \sup_{\ell_\ast } \sup_{U_0 \in \cU_{\ell_\ast, A}} \sup_{\omega \in \Omega_\lambda} \sup_{x \in A} P^\omega_x[H_{\textnormal{Res}} = \infty],
\end{equation}
where the first supremum is over all $\ell_\ast$ which are $(I,J,L)$-compatible. Define also
\begin{equation}
\Phi_{J,I,L} =  \Phi_{J,I,L,\{ 0 \}}.
\end{equation}
Then, we have the maximality property
\begin{equation}
\label{eq:maximality}
\Phi_{J,I,L,A} \leq \Phi_{J,I,L},
\end{equation}
and as $I \rightarrow \infty$, 
\begin{equation}
\limsup_I I^{-1/2^{J-1}} \log \Phi_{J,I,L} \leq \log(1 - c_7(J)) (< 0). 
\end{equation}
\end{theorem}
\begin{rem}
It should be noted that as $I$ increases, the set of compatible $\ell_\ast$ decreases, so that for \textit{fixed} $A$, we can only make a statement about sets $U_0$ whose boundary is ``at distance'' $2^{\ell_\ast}$ from that of $A$, and so are the associated resonance sets. However, this result will ultimately be applied to $A = K_N = (NK) \cap \bbZ^d$, where $K$ is a compact subset of $\bbR^d$ with non-empty interior, in the limit $N \rightarrow \infty$. Therefore, for any given $I,J,L$ we can find a compatible $\ell_\ast$, for which we then choose $N \gg 2^{\ell_\ast}$ giving us a useful resonance set, since its ``distance'' from the boundary of $A$ is of a much smaller order than the size of $A$.
\end{rem}

The proof of this Theorem~\ref{thm:Resonance} is an adaptation Theorem 2.1 of~\cite{nitzschner2017solidification} and is given in the Appendix for completeness. It involves the use of the strong Markov property for random walks and Proposition~\ref{prop_multiscale_descent}, for which quenched killed heat kernel estimates were instrumental. The latter is the main difference compared to the version for Brownian motion, together with the aforementioned necessity of the $(I,J,L)$-compatibility condition in the discrete set-up. 

\subsection{Proof of Theorem~\ref{thm:SolidificationTheorem}}
In this subsection, we turn to the proof of the main solidification result~\eqref{eq:SolidificationStatement} of Theorem~\ref{thm:SolidificationTheorem}. Let us remark that the discrete nature of the problem is an impediment to the program only as far as we are forced to work with results asymptotic in $N$: As in the proof of the equivalent statement for Brownian motion, the crucial step is the introduction of the resonance set, and here we need $\ell_\ast$ to be $(I,J,L)$-compatible as defined in the previous subsection. Ultimately, this will be achieved by choosing $2^{\ell_\ast} \asymp N$, which ensures that for \emph{fixed} $I, J \geq 1$ and $L \geq L(J)$, $\ell_\ast$ will in fact be $(I,J,L)$-compatible for large $N$. 

We now make the above program precise: 
\begin{proof}[Theorem~\ref{thm:SolidificationTheorem}]
We start by introducing a discrete version of Lemma 3.2 in~\cite{nitzschner2017solidification}, utilizing again quenched killed heat kernel estimates to adapt the proof to our set-up.
\begin{lem}
\label{main_lemma}
Let $\omega \in \Omega_\lambda$. For $\Sigma \in \cS^\omega_{U_0, \varepsilon, \chi}$, $\ell \geq c_8$ such that $\varepsilon \leq \tfrac{1}{4}\cdot 2^{\ell}$, $x_0, y \in \bbZ^d$ with $\widetilde{\sigma}_\ell(x_0) \in [\widetilde{\alpha}, 1- \widetilde{\alpha}]$ and $|y - x_0| \leq \tfrac{1}{4}\cdot 2^\ell$, it holds that
\begin{equation}
P^\omega_y[H_\Sigma < T_{B(x_0, 5 \cdot 2^\ell )} ] \geq c_9(\chi).
\end{equation}
\end{lem}
\begin{proof} Define  $\widetilde{U}_0 = U_0 \cap B(x_0, 4 \cdot 2^\ell )$ and $\widetilde{U}_1 = U_1 \cap B(x_0, 4 \cdot 2^\ell )$, and note that 
\begin{equation}
\label{eq:LowerBoundLemma49}
\begin{split}
P^\omega_y&[X_{ \frac{81}{4} \cdot 4^\ell} \in U_0,  X_{\frac{81}{2} \cdot 4^\ell} \in U_1 \text{ and } \tfrac{81}{2} \cdot 4^\ell < T_{B(x_0, \frac{9}{2} \cdot 2^\ell )}]\\
& \stackrel{\eqref{eq:KilledHeatKernel},\eqref{eq:reversibMeasure} }{\geq} (2d\lambda)^2  \sum_{x \in \widetilde{U}_0  } \sum_{z \in\widetilde{U}_1} q^{\omega}_{\frac{81}{4} \cdot4^\ell, B(x_0, \frac{9}{2} \cdot 2^\ell )}(y,x) q^{\omega}_{\frac{81}{4} \cdot4^\ell, B(x_0, \frac{9}{2} \cdot 2^\ell )}(x,z) \geq c'.
\end{split}
\end{equation} 
In the last step, we used that both $\widetilde{U}_0$ and $\widetilde{U}_1$ have at least a volume of $\widetilde{\alpha} (4\cdot 2^\ell)^d$ due to the assumption, together with the heat kernel bound~\eqref{eq:HK_bound} (with $\vartheta = \tfrac{1}{9}$, $R = \tfrac{9}{2} \cdot 2^\ell$ so that $(1-\vartheta)R = 4 \cdot 2^\ell$, and $t = R^2$, choosing $c_8$ sufficiently large). On the event under the probability on the left-hand side of~\eqref{eq:LowerBoundLemma49}, $X$ visits $U_0$ and $U_1$ before exiting $B(x_0, \frac{9}{2} \cdot 2^\ell )$, which is to say that (with $S = \partial U_0$):
\begin{equation}
P^\omega_y[H_S < T_{B(x_0, \frac{9}{2} \cdot 2^\ell )}] \geq c'.
\end{equation}
Finally, by the strong Markov property, we obtain
\begin{equation}
\begin{split}
& P^\omega_y[H_\Sigma \circ \theta_S + H_S< T_{B(x_0, 5 \cdot 2^\ell )}] \geq \\
&E^\omega_y\big[ H_S <  T_{B(x_0, \frac{9}{2} \cdot 2^\ell )}, P^\omega_{X_{H_S}}[H_\Sigma < \tau_\varepsilon] \big] \geq c'\chi = c_9(\chi),
\end{split}
\end{equation}
using that $\varepsilon \leq \tfrac{1}{4}\cdot 2^\ell$. 
\end{proof}
We resume the proof of the main solidification estimate~\eqref{eq:SolidificationStatement}. To this end, pick $J \geq 1$, $I \geq 1$ and $L \geq L(J)$ (recall~\eqref{eq:LJdef} for the definition of $L(J)$). Now choose $N_0 (=N_0(I,J,L))$ big enough such that for all $N \geq N_0$, we have 
\begin{align}
&a_N \leq \tfrac{1}{4}2^{-(I+1)(J+1)L}, \\
\label{eq:TwoCondCompatibility}
&\ell_\ast \geq 0 \text{ is $(I,J,L)$-compatible  and $\ell_\ast - (I+1)(J+1)L \geq c_8$ if }1/2^{\ell_\ast} \leq a_{N_0},
\end{align}
where $c_8$ is the constant in Lemma~\ref{main_lemma}. Importantly, we have
\begin{align}
|\cA_\ast| = I(J+1), \ \min \cA_\ast \geq \ell_\ast - (I+1)(J+1)L \geq \ell_{\min}((200J)^{-1}), \ \max \cA_\ast \leq \ell_\ast.
\end{align}
We will now conclude the proof and argue that for $\varepsilon/2^{\ell_\ast} \leq a_N$ and $x_0 \in \text{Res}$, $\omega \in \Omega_\lambda$, $\Sigma \in \cS^\omega_{U_0,\varepsilon, \chi}$ with $U_0 \in \cU_{\ell_\ast, A_N}$:
\begin{equation}
P^\omega_{x_0}[H_\Sigma > T_{B(x_0, 5 \cdot 2^{\max \cA_\ast})}] \leq (1 - c_9(\chi))^J.
\end{equation}
Indeed, $\varepsilon \leq a_N 2^{\ell_\ast} \leq \tfrac{1}{4} 2^{\min \cA_\ast}$, and applying the strong Markov property at successive exit times of balls $B(x_0, 5 \cdot 2^\ell)$, $\ell \in \cA_\ast$, we find (because $5 \cdot 2^{\ell'} \leq \tfrac{1}{4}\cdot 2^\ell$ for $\ell' < \ell \in \cA_\ast$) upon repeated use of Lemma~\ref{main_lemma}, that 
\begin{equation}
P^\omega_{x_0}[H_\Sigma > T_{B(x_0, 5 \cdot 2^{\max \cA_\ast})}] \leq (1 - c_9(\chi))^{\sum_{\ell \in \cA_\ast} \mathbbm{1}_{ \{ \widetilde{\sigma}_\ell(x_0) \in [\widetilde{\alpha}, 1 - \widetilde{\alpha} ] \} } } \leq (1 - c_9(\chi))^J,
\end{equation}
having used $x_0 \in \text{Res}$ in the last step. So we obtain for $\omega \in \Omega_\lambda$, $x \in A_N$, and $N \geq N_0$:
\begin{equation}
\begin{split}
P^\omega_x[H_\Sigma = \infty] & \leq P^\omega_x[H_{\text{Res}} = \infty] + E^\omega_x[H_{\text{Res}} < \infty, P^\omega_{X_{H_{\text{Res}}}}[H_\Sigma = \infty] ] \\
& \leq P^\omega_x[H_{\text{Res}} = \infty] + (1 - c_9(\chi))^J \\
& \leq \Phi_{J,I,L}+ (1 - c_9(\chi))^J.
\end{split}
\end{equation}
For $N \geq N_0$, we then take the supremum over $x \in A_N$, $\Sigma \in \cS^\omega_{U_0, \varepsilon, \chi}$, $\omega \in \Omega_\lambda$, $U_0 \in \cU_{\ell_\ast, A_N}$ and $\varepsilon/ 2^{\ell_\ast} \leq a_N$, and let $N$ tend to $\infty$ to obtain:
\begin{equation}
\varlimsup_{N \rightarrow \infty} \sup_{\varepsilon/ 2^{\ell_\ast} \leq a_N} \sup_{U_0 \in \cU_{\ell_\ast, A_N}}  \sup_{\omega \in \Omega_\lambda} \sup_{\Sigma \in \cS^\omega_{U_0, \varepsilon, \chi} } \sup_{x \in A_N} P^\omega_x[H_\Sigma = \infty] \leq  \Phi_{J,I,L}+ (1 - c_9(\chi))^J  
\end{equation}
where we used that the set of $(\varepsilon, \ell_\ast) \in \bbN^2$, $\varepsilon \geq 1$, such that $\varepsilon / 2^{\ell_\ast} \leq a_N$ fulfills~\eqref{eq:TwoCondCompatibility} since $N \geq N_0$. We first let $I \rightarrow \infty$, then $J \rightarrow \infty$ and obtain the claim~\eqref{eq:SolidificationStatement}. Now~\eqref{eq:Solidification_From_Outside} follows from standard arguments, which we only sketch briefly (see also Lemma 2.1 of~\cite{chiarini2019entropic}). Note that for $x \in \bbZ^d$, we have
\begin{equation}
\label{eq:ProofSolidFromOutsidePenultimate}
\begin{split}
P^\omega_x[H_\Sigma = \infty] = E^\omega_x[ P^\omega_x[H_\Sigma = \infty, H_{A_N} < \infty| \cF_{H_A}]] + P^\omega_x[H_{A_N } = \infty].
\end{split} 
\end{equation}
Using the strong Markov property at $H_{A_N}$, together with and fact that on $\{H_{A_N} < H_\Sigma, H_{A_N} < \infty\}$, it holds that $H_\Sigma = H_\Sigma \circ \theta_{H_{A_N}} + H_{A_N}$, the first term on the right-hand side of~\eqref{eq:ProofSolidFromOutsidePenultimate} is bounded by $  \sup_{x \in A_N} P^\omega_x[H_{\Sigma} = \infty]$, and \eqref{eq:Solidification_From_Outside} readily follows. 
\end{proof}

\section{Quenched entropic lower bound on disconnection}
\label{sec:LowerBound}

In this section we consider levels $\alpha < \alpha_{**}$ and prove a quenched asymptotic lower bound on the decay rate for the probability of disconnection $\bbP^\omega[\cD^\alpha_N]$ in the limit as $N$ goes to infinity. For the derivation we follow a traditional approach in large deviation theory. We tilt the Gaussian free field measure in such a way that disconnection becomes typical. Provided that one can compute the Radon-Nikodym derivative of the tilted measure with respect to the original measure, all that is left to do is to use the relative entropy lower bound~\eqref{eq:RelEntropyInequality}. The estimate so obtained is roughly in terms of $N^{2-d}\capa^\omega(A_N)$ which, with the aid of Corollary~\ref{cor:Capacity_convergence_NoKilling} below, converges for $P$-a.e.\ $\omega\in \Omega_\lambda$ to $\capa^{\homo}(A)$, which is introduced in~\eqref{eq:CapacityContinuumDef}. 

\begin{theorem}
    \label{thm:MainLowerBound}
     Let $\alpha < \alpha_{**}$, then for $P$-a.e.\ $\omega\in \Omega_\lambda$ and $A \subseteq \bbR^d$ compact, it holds that
    \begin{equation}
    \label{eq:QuenchedDiscLowerBound}
        \liminf_{N\to\infty} \frac{1}{N^{d-2}} \log \bbP^\omega[\cD_N^\alpha] \geq -\frac{1}{2} (\alpha_{**}- \alpha)^2 \capa^{\homo}(A).
    \end{equation}
\end{theorem}
We now introduce  $\capa^{\homo}(A)$ and show in Corollary~\ref{cor:Capacity_convergence_NoKilling} the $P$-a.s.\ convergence  of the discrete capacity $N^{2-d} \capa^\omega(A_N)$ to this quantity, employing a $\Gamma$-convergence result in Corollary 3.4 of~\cite{neukamm2017stochastic}.

It is well known that if $P$ is a stationary and ergodic probability on $\Omega_\lambda$, then for $P$-a.e.\ $\omega$ the scaled variable-speed random walk $(\overline{X}_{t N^2}/N)_{t\geq 0}$ under $P_0^\omega$ converges in law to a Brownian motion  $(Z_t)_{t\geq 0}$ with deterministic and non-degenerate covariance $a^\homo\in \bbR^{d\times d}$. Its law started at $x\in \bbR^d$ is denoted by $W_x$. The Dirichlet form on $L^2(\bbR^d)$ with domain $W^{1,2}(\bbR^d)$ associated with it is given by
\begin{equation}
    \bfD(u,u) =\frac{1}{2} \int_{\bbR^d} \nabla u(x) \cdot a^\homo \nabla u(x) \,\De x,\qquad u\in W^{1,2}(\bbR^d),
\end{equation}
where $W^{1,2}(\bbR^d)$ is the classical Sobolev space of weakly differentiable functions that are square-integrable and have square integrable first weak derivatives. We then define the harmonic potential of a closed or open bounded set $A \subseteq \bbR^d$ as
\begin{equation}
\label{eq:HarmonicPotContinuumDef}
\mathscr{h}_A(x) = W_x[(Z_t)_{t \geq 0} \text{ hits } A], \qquad x \in \bbR^d,
\end{equation}
and for $A \subseteq B \subseteq \bbR^d$ both bounded and closed or open we define the harmonic potential of $A$ with respect to $B$ as 
\begin{equation}
\label{eq:HarmonicPotContinuumKilledDef}
\mathscr{h}_{A,B}(x) = W_x[(Z_t)_{t \geq 0} \text{ hits } A \text{ before exiting }B], \qquad x \in \bbR^d.
\end{equation}
Note that $\mathscr{h}_{A,B} \in W^{1,2}(\bbR^d)$. Combining Theorem 4.3.3, p. 171 of~\cite{fukushima2010dirichlet} with Theorem 2.1.5, p. 72 of the same reference, one knows that for any closed or open bounded set $A \subseteq \bbR^d$, $\mathscr{h}_A$ is in the extended Dirichlet space of $(\bfD, W^{1,2}(\bbR^d))$ (see Example 1.5.3 in~\cite{fukushima2010dirichlet} for a characterization of this space). We can therefore define
\begin{equation}
\label{eq:CapacityContinuumDef}
    \capa^\homo(A) = \bfD(\mathscr{h}_A,\mathscr{h}_A),\qquad \capa^\homo_B(A) = \bfD(\mathscr{h}_{A,B},\mathscr{h}_{A,B}).
\end{equation}

\begin{rem}\label{rem:on_regularity} 1) The matrix $a^\homo$ is typically not directly accessible. If the law $P$ of the conductances is invariant under symmetries of $\bbZ^d$, then $a^\homo = \sigma^2 I_d$ for some $\sigma^2 > 0$. In this case, $\capa^\homo(A)$ is a scalar multiple of the Brownian capacity of $A$. In particular, the ``regularity condition'' $\capa^\homo(A) = \capa^\homo(\mathring{A})$ used in Section~\ref{sec:UpperBounds} is equivalent to the equality of the Brownian capacities of $A$ and $\mathring{A}$, which is the regularity condition used in~\cite{chiarini2019entropic,chiarini2020entropic,nitzschner2017solidification,nitzschner2018entropic}.

2) If $P$ is a stationary and ergodic probability on $\Omega_\lambda$, then for $P$-a.e.\ $\omega$, also the scaled constant-speed random walk $(X_{t N^2}/N)_{t\geq 0}$ under $P_0^\omega$ converges in law to a Brownian motion. Its covariances are given by $E_P[\omega_0]^{-1} a^\homo$, where $E_P$ is the expectation associated with $P$.
    
\end{rem}
In Proposition~\ref{prop:homo_killed} and Corollary~\ref{cor:Capacity_convergence_NoKilling} below we show that, additionally to the quenched invariance principle for the random walk among random conductances, one also has the $P$-a.s.\ convergence of the associated capacities. These results will be instrumental in the proofs of Theorems~\ref{thm:MainLowerBound},~\ref{thm:UpperBound} and~\ref{thm:EntropicRepulsion}. In what follows, we say that $D \subseteq \bbR^d$ is a Lipschitz domain if $D$ is bounded and the boundary of $D$
can be represented locally as a graph of a Lipschitz function.


\begin{prop}\label{prop:homo_killed} Let $D\subseteq \bbR^d$ be a compact Lipschitz domain and $B$ the open Euclidean ball of radius $R>0$ centered at the origin such that $D\subseteq B$. Then, for $P$-a.e.\ $\omega\in \Omega_\lambda$,
    \begin{equation}\label{eq:minima_convergence_killed}
        \frac{1}{N^{d-2}}\capa^\omega_{B_N}(D_N) \to \capa^\homo_B(D),\qquad \text{as $N\to\infty$,}
    \end{equation}
    and for all $f:\bbR^d\to \bbR$ continuous and compactly supported one has
    \begin{equation}\label{eq:minima_convergence}
        \frac{1}{N^{d}}\sum_{x\in\bbZ^d} h^\omega_{D_N,B_N}(x) f(\tfrac{x}{N}) \to \int_{\bbR^d} \mathscr{h}_{D,B}(x) f(x)\,\De x, \qquad \text{as $N\to\infty$.}
    \end{equation}
    \end{prop}
    \begin{proof} The proof of this Proposition follows from Corollary 3.4 in~\cite{neukamm2017stochastic} which is a quenched $\Gamma$-convergence result for discrete random energies stemming from a potential much more general than quadratic and possibly degenerate. Under our assumption of bounded and uniformly elliptic conductances, their results apply to the potential $V(\omega,\{x,y\},r) = \omega_{x,y} r^2$, $q=p=2$ and $\beta =\infty$. 
    In Corollary 3.4 of~\cite{neukamm2017stochastic}, the authors impose boundary conditions with a boundary layer which we wish to remove in our statement. The discussion below is to address the different boundary condition.
    
    In this proof we will keep close to the notation of~\cite{neukamm2017stochastic}, so we work with $\varepsilon =1/N$, we identify each $u:\varepsilon\bbZ^d \to \bbR$ with its unique, canonical piecewise affine interpolation on $\bbR^d$, and denote by $\cA^\varepsilon$ the collection of all such piecewise affine interpolations.

    We proceed by defining our Dirichlet boundary conditions, which is the main difference with respect to Corollary 3.4 in~\cite{neukamm2017stochastic}. That is, we set
    \begin{equation}
        \cA_\varepsilon^{b.c.} = \{u\in \cA^\varepsilon: u \equiv 0 \text{ on $(\bbR^d\setminus B)\cap \varepsilon \bbZ^d $ and }u\equiv 1\text{ on $D \cap \varepsilon \bbZ^d$} \}.
    \end{equation}
Next we consider the energy functional
    \begin{equation}
        J^\omega_\varepsilon(u) = \begin{cases}
            \displaystyle \frac{\varepsilon^{d-2}}{2} \sum_{x,y \in \bbZ^d,  x  \sim y} \omega_{x,y} \big(u(\varepsilon x) - u(\varepsilon y)\big)^2, & \textrm{if $u\in \cA^{b.c.}_\varepsilon$},\\
            \displaystyle+\infty, & \textrm{otherwise.}
        \end{cases}
    \end{equation}
Observe that from~\eqref{eq:variationalchar} we have
\begin{equation}
    \inf_{u\in  \cA^{b.c.}_{1/N}} J^\omega_{1/N}(u) = N^{2-d}\capa_{B_N}^\omega(D_N)
\end{equation}
and the infimum is achieved uniquely at $\tilde{h}_{N} \in \cA^{b.c.}_{1/N}$ such that $\tilde{h}_{N}(x/N) = h^\omega_{D_N,B_N}(x)$ for all $x\in \bbZ^d$. 

Define the set $U = B \setminus D$ and note that it is a bounded domain with Lipschitz boundary. We also let $g\in C^\infty(\bbR^d)$ be such that $g|_D = 1$ and $g|_{\bbR^d\setminus B} = 0$. Observe that for all $\delta > \varepsilon\sqrt{d}$, if $u\in \cA_\varepsilon^{b.c.}$, then $u \equiv g$ on  $\bbR^d\setminus U^\delta$. We claim that for $P$-a.e.\ $\omega\in \Omega_\lambda$ and all $\delta>0$ small enough:
\begin{itemize}
    \item \emph{(Coercivity)} Any sequence $(u_\varepsilon)$ with finite energy, that is,
    \begin{equation}
        \limsup_{\varepsilon \downarrow 0} J^\omega_\varepsilon(u_\varepsilon)<\infty,
    \end{equation}
    admits a subsequence that strongly converges in $L^2(U^\delta)$ to a limit  $u\in g+W_0^{1,2}(U)$, where $W_0^{1,2}(U)$ denotes the closure of $C^\infty_0(U)$ in $W^{1,2}(\bbR^d)$.
    \item \emph{($\Gamma$-convergence)} The sequence of functionals $(J^\omega_\varepsilon)$  $\Gamma$-converges with respect to the strong convergence in $L^2(U^\delta)$
    to the functional $J_\homo$ given by
    \begin{equation}
    \label{eq:Jhom_Def}
        J_\homo(u)=
        \begin{cases}
            \displaystyle \frac{1}{2}\int_{U} \nabla u(x) \cdot a^\homo\nabla u(x)\,\De x, & \textrm{if $u\in g + W_0^{1,2}(U)$},\\
            \displaystyle+\infty, & \textrm{otherwise.}
        \end{cases}
    \end{equation}
    where $a^\homo\in \bbR^{d\times d}$ is a symmetric and positive definite matrix. Moreover, in view of the variational characterization of the capacity (see Theorem 2.1.5 and Theorem 4.3.3 in~\cite{fukushima2010dirichlet}), we have $\capa^\homo_B(D) = \inf J_\homo$ and the infimum is achieved uniquely at $\mathscr{h}_{D,B}$.
\end{itemize}

We will only discuss the coercivity, as the $\Gamma$-convergence follows from the exact same argument as in Corollary 3.4 of~\cite{neukamm2017stochastic} once the coercivity is established. 
Fix $\delta>0$, the existence of a subsequential limit $u\in g+ W_0^{1,2}(U^\delta)$ of $(u_\varepsilon)$ with respect to strong convergence in $L^2(U^\delta)$, follows from Lemma 3.3 of~\cite{neukamm2017stochastic}, from the observation that $u_\varepsilon-g\in W_0^{1,2}(U^\delta)$ for all $\varepsilon$ small enough, and from
\begin{equation}
\begin{split}
    \limsup_{\varepsilon\downarrow 0} \int_{U^\delta} |\nabla u_\varepsilon-\nabla g|^2\,\De x & \leq  c\limsup_{\varepsilon\downarrow 0} \int_{U^\delta} |\nabla u_\epsilon|^2 \,\De x + c' \\
    & \leq c \limsup_{\varepsilon\downarrow 0} J_\varepsilon^\omega(u_\varepsilon) + c' < \infty,
    \end{split}
\end{equation}
with the penultimate inequality due to Lemma 2.1 of~\cite{neukamm2017stochastic}.
Note that by considering  $0<\delta'<\delta$ and looking at sub-subsequences we can deduce $u\in g+ W_0^{1,2}(U^{\delta'})$. As $\delta'$ is arbitrary, we get $u\in g+ W_0^{1,2}(U)$.

As a result of the $\Gamma$-convergence, the coercivity, and the uniqueness of the minimizers in our setup, one has the convergence of the minima and minimizers of $J^\omega_\varepsilon$ to those of $J_\homo$ as $\varepsilon\downarrow 0$. That is, we obtain that for $P$-a.e.\ $\omega\in \Omega_\lambda$
    \begin{equation}
       \lim_{N\to\infty} \frac{1}{N^{d-2}} \capa^\omega_{B_N}(D_N) = \lim_{N\to\infty} \min J^\omega_{1/N} = \min J_\homo = \capa^\homo_B(D),
    \end{equation}
    which proves~\eqref{eq:minima_convergence_killed}.
    Furthermore, the affine interpolation $\tilde{h}_N$ of $h^\omega_{D_N,B_N}(\cdot/N)$ converges in $L^2(U^\delta)$ to $\mathscr{h}_{D,B}$ for some $\delta>0$ as $N\to\infty$. Since $\tilde{h}_N$ and $\mathscr{h}_{D,B}$ coincide on $\bbR^d\setminus U_{\delta}$ for all $N$ large enough, and since an application of Proposition A.1 in~\cite{alicandro2000finite} allows to replace affine interpolations with piecewise constant interpolations and still retain convergence $L^2(U^\delta)$, we finally get, with the help of Cauchy-Schwartz's inequality,~\eqref{eq:minima_convergence} for any continuous test function with compact support. 
    \end{proof}

    In the following corollary we extend the above homogenization result to the case without imposing zero boundary conditions outside of a big ball.
    
    \begin{cor}
    \label{cor:Capacity_convergence_NoKilling}
        Let $D\subseteq \bbR^d$ be a bounded Lipschitz domain.
        Then, for $P$-a.e.\ $\omega\in \Omega$,
    \begin{equation}\label{eq:convergence_capacity}
        \frac{1}{N^{d-2}}\capa^\omega(D_N) \to \capa^\homo(D),\qquad \text{as $N\to\infty$,}
    \end{equation}
    and for all $f:\bbR^d\to \bbR$ continuous and compactly supported one has
    \begin{equation}\label{eq:convergence_minimizer}
        \frac{1}{N^{d}}\sum_{x\in\bbZ^d} h^\omega_{D_N}(x) f(\tfrac{x}{N}) \to \int_{\bbR^d} \mathscr{h}_{D}(x) f(x)\,\De x,  \qquad \text{as $N\to\infty$.}
    \end{equation}
    \end{cor}
    \begin{proof} The basic idea is to show that it is possible to take $R\to\infty$ before taking the limit $N\to \infty$ in Proposition~\ref{prop:homo_killed}.
    First, observe that we have the following decay for the hitting probability of $D_N$ if we start from $x\in \bbZ^d\setminus D_N$, which is an application of the decay of the Green function, see~\eqref{eq:QuenchedGFEstimate},
    \begin{equation}\label{eq:harmonicdecay}
        \begin{aligned}
            h^\omega_{D_N}(x)  &\leq  \frac{c}{d_\infty(\{x\},D_N)^{d-2}}  \sum_{y\in \partial D_N} e^{\omega}_{D_N}(y) = c \frac{\capa^\omega(D_N)}{d_\infty(\{x\},D_N)^{d-2}}.
        \end{aligned}
    \end{equation}
    Note that from the random walk characterization of the capacity one has
    \begin{equation}
        \begin{aligned}
            \capa^\omega_{B_N}(D_N) -  \capa^\omega(D_N) &= \sum_{x\in\partial D_N} \omega_x \Big(P_x^\omega[\widetilde{H}_{D_N}>T_{B_N}] - P_x^\omega[\widetilde{H}_{D_N} = \infty]\Big)\\
                       &\leq\sum_{x\in\partial D_N} \omega_x E_x^\omega[\widetilde{H}_{D_N}>T_{B_N}, P^\omega_{X_{T_{B_N}}}[\widetilde{H}_{D_N}<\infty]].
        \end{aligned}
    \end{equation}
    Since $\widetilde{H}_{D_N} = H_{D_N}$ if the walk starts from $\partial B_N$, in view of~\eqref{eq:harmonicdecay} and $\capa^\omega(D_N)\leq \capa^\omega_{B_N}(D_N)$
    \begin{equation}
        0\leq \capa^\omega_{B_N}(D_N) -  \capa^\omega(D_N) \leq c \frac{\capa^\omega_{B_N} (D_N)^2}{d_\infty(B_N^c,D_N)^{d-2}}.
    \end{equation}
    We can now finish the proof of~\eqref{eq:convergence_capacity}. Fix $R>0$ arbitrary and such that $A\subseteq B$. Then,
    \begin{equation}
        \begin{aligned}
            \varlimsup_{N\to\infty} &\Big|N^{2-d}\capa^\omega(D_N) - \capa^\homo(D)\Big| \leq \varlimsup_{N\to\infty} \Big|N^{2-d}\capa^\omega_{B_N}(D_N) - \capa_B^\homo(D)\Big|\\ &+ c\varlimsup_{N\to\infty} \bigg(\frac{\capa^\omega_{B_N}(D_N)}{N^{d-2}}\bigg)^2 \frac{N^{d-2}}{d_\infty(B_N^c,D_N)^{d-2}} + |\capa_B^\homo(D)-\capa^\homo(D)|.
        \end{aligned}
    \end{equation}
    An application of Proposition~\ref{prop:homo_killed} and taking $R\to\infty$ completes the proof of~\eqref{eq:convergence_capacity}.
    
    The convergence \eqref{eq:convergence_minimizer} follows from the following simple observation
    \begin{equation}
        \begin{aligned}
            0\leq h^\omega_{D_N}(x) - h^\omega_{D_N,B_N}(x) &= P_x^\omega[H_{D_N}<\infty, T_{B_N}\leq H_{D_N}]\\
            & \leq \sup_{y\notin B_N}P^\omega_y[H_{D_N}<\infty] \leq c\frac{\capa^\omega(D_N)}{(NR)^{d-2}},
        \end{aligned}
    \end{equation}
    together with the fact that $\mathscr{h}_{D,B} \uparrow \mathscr{h}_D$ as $R\to \infty$. 
    \end{proof}

    \begin{proof}[Theorem~\ref{thm:MainLowerBound}] Let $f : \bbZ^d \to \bbR$ be finitely supported, and consider for a given $\omega \in \Omega_\lambda$ the probability measure $\widetilde{\bbP}^\omega$ on $\bbR^{\bbZ^d}$ defined by
        \begin{equation}
            \De \widetilde{\bbP}^\omega = \exp\Big\{\cE^\omega(f,\varphi) - \frac{1}{2}\cE^\omega(f,f)\Big\}\, \De \bbP^\omega.
        \end{equation}
        Then, as in (2.4) of~\cite{sznitman2015disconnection}, $\widetilde{\bbP}^\omega$ is equivalently characterized by
        \begin{equation}
            \begin{minipage}{0.8\linewidth}\begin{center}
                $\varphi$ under $\widetilde{\bbP}^\omega$ has the same law as $(\varphi_x + f(x))_{x\in \bbZ^d}$ under $\bbP^\omega$.\end{center}
            \end{minipage}
        \end{equation}
        We now fix three parameters $\epsilon, R,\delta > 0$ such that $A^\delta \subseteq B$, where $B$ is the open Euclidean ball of radius $R$ centered at the origin and $A^\delta$ denotes the closed $\delta$-neighborhood of $A$. By slightly modifying $A^\delta$, if necessary, we can assume it has Lipschitz boundary. We also set  $(N A^\delta)\cap \bbZ^d = (A^\delta)_N$. We consider the functions
        \begin{equation}
        f^\omega_N = -(\alpha_{\ast\ast} - \alpha + \epsilon)\,h^\omega_{(A^\delta)_N, B_N},
        \end{equation}
        and we denote the corresponding probability measures by $\widetilde{\bbP}^\omega_N$.
        
        As a consequence of the relative entropy inequality~\eqref{eq:RelEntropyInequality}, we find that 
        \begin{equation}\label{eq:after_relative_entropy}
            \log \bbP^\omega[\cD^\alpha_N] \geq \log \widetilde{\bbP}_N^\omega[\cD^\alpha_N] -\frac{1}{\widetilde{\bbP}^\omega_N[\cD^\alpha_N]} \Big(\cE^\omega(f_N^\omega,f_N^\omega) +\frac{1}{\mathrm{e}}\Big),
        \end{equation}
        where we used that $H(\widetilde{\bbP}^\omega_N|\bbP^\omega) = \widetilde{\bbE}^\omega_N[\cE^\omega(f_N^\omega,\varphi)] - \tfrac{1}{2}\cE^\omega(f_N^\omega,f_N^\omega) = \tfrac{1}{2} \cE^\omega(f_N^\omega,f_N^\omega)$.
        Because of our choice of $f_N^\omega$, we have 
        \begin{equation}
            \cE^\omega(f_N^\omega,f_N^\omega) = - (\alpha_{**} -\alpha + \epsilon)^2 \capa^\omega_{B_N}((A^\delta)_N).
        \end{equation}
        We now show that $\cD_N^\alpha$ is typical under $\widetilde{\bbP}^\omega_N$ in the limit $N\to\infty$, for $P$-a.e.\ $\omega \in \Omega_\lambda$.
        Indeed,
        \begin{equation}\label{eq:tilted_disconnection}
                \widetilde{\bbP}^\omega_N[\cD^\alpha_N]  = \bbP^\omega\Big[A_N \stackrel{\geq \alpha - f_N^\omega}{\centernot \longleftrightarrow} S_N\Big] \geq \bbP^\omega\Big[A_N \stackrel{\geq \alpha_{\ast\ast} + \epsilon}{\centernot \longleftrightarrow} \partial_i (A^\delta)_N\Big],
        \end{equation}
     where for $U,V \subseteq \bbZ^d$ and a function $g: \bbZ^d \rightarrow \bbR$, the event $\{U \stackrel{\geq g}{\centernot \longleftrightarrow} V\}$ denotes the absence of a nearest-neighbor path in $\{ x \in \bbZ^d \, : \, \varphi_x \geq g(x) \}$ starting in $U$ and ending in $V$. In the last step we used that $f_N^\omega = -(\alpha_{\ast\ast} - \alpha + \epsilon)$ on $(A^\delta)_N$. As a result of a union bound, we have for some constant $c'>0$, depending only on the dimension, and constants $C, c, \kappa > 0$, possibly depending on $\omega$, 
        \begin{equation}\label{eq:complementaryevent}
            \begin{aligned}
                \bbP^\omega\Big[A_N \stackrel{\geq \alpha_{\ast\ast} + \epsilon}{\longleftrightarrow} \partial_i (A^\delta)_N\Big] & = \bbP^\omega \bigg[\bigcup_{x\in A_N} \big\{ x \stackrel{\geq \alpha_{\ast\ast} + \epsilon}{\longleftrightarrow} \partial_i (A^\delta)_N \big \}\bigg] \\
                & \leq |A_N| \sup_{x\in A_N } \bbP^\omega \Big[ x \stackrel{\geq \alpha_{\ast\ast} + \epsilon}{\longleftrightarrow} \partial B(x, c'\delta N) \Big] \\
                & \leq C N^{2d - 1} e^{- c N^{\kappa}} \to 0,\quad \text{as $N\to \infty$,} 
            \end{aligned}
        \end{equation}
        where in the last step we used a union bound together with the stretched exponential estimate in Theorem~\ref{thm:StretchedExpDecay}.
        Therefore, combining~\eqref{eq:tilted_disconnection} and~\eqref{eq:complementaryevent},
        \begin{equation}\label{eq:typical}
            \widetilde{\bbP}^\omega_N[\cD^\alpha_N] \geq 1 - \bbP^\omega\Big[A_N \stackrel{\geq \alpha_{\ast\ast} + \epsilon}{\longleftrightarrow} \partial_i (A^\delta)_N\Big] \to 1,\quad\text{as $N\to \infty$.}
        \end{equation}
        Finally, by means of~\eqref{eq:after_relative_entropy},~\eqref{eq:typical}, and Proposition~\ref{prop:homo_killed} above for $P$-a.e.\ $\omega \in \Omega_\lambda$
        \begin{equation}
            \begin{aligned}
            \liminf_{N\to\infty} \frac{1}{N^{d-2}} \log \bbP^\omega[\cD_N^\alpha] & \geq - \frac{1}{2} (\alpha_{**}- \alpha + \epsilon)^2 \limsup_{N\to\infty} \frac{1}{N^{d-2}}\, \capa^\omega_{B_N}((A^\delta)_N) \\
            & =  - \frac{1}{2} (\alpha_{**}- \alpha + \epsilon)^2 \capa^{\homo}_B(A^\delta).
            \end{aligned}
        \end{equation}
        We now take successively the limits $\epsilon,\delta \to 0$ and $R\to\infty$, and the claim is proved. 
        \end{proof}

\section{Quenched upper bounds for Gaussian functionals}
\label{sec:Quenched}

In this section, we derive some bounds on certain Gaussian functionals similar in spirit to~\cite[Section 4]{sznitman2015disconnection}, see also~\cite[Sections 3 and 4]{chiarini2019entropic}. These bounds, which are presented in Theorem~\ref{thm:ZmMainBound} below, will capture in the next section the main asymptotic decay rates for the probability of the disconnection event $\cD^\alpha_N$, possibly intersected with the ``entropic repulsion'' event  $\big\{\big\vert \langle \bbX_N, \eta \rangle - \langle \mathscr{H}^\alpha_{\mathring{A}}, \eta \rangle \big\vert \geq \Delta \big\}$.  The bounds we derive in this section again utilize quenched Green function estimates for the weighted graph $(\bbZ^d,\bbE_d,\omega)$ and are uniform in $\omega \in \Omega_\lambda$. 

We start by introducing some notation and the set-up. Let $L \geq 1$ and $K \geq 100$ be integers. We consider the lattice
\begin{equation}
\label{eq:LatticeL_Def}
\bbL = L \bbZ^d,
\end{equation}
 and the collection of boxes for $z \in \bbZ^d$:
\begin{equation}
\label{eq:BoxesDefinition}
\begin{split}
B_z &= z + [0,L)^d \cap \bbZ^d \subseteq D_z = z + [-3L,4L)^d \cap \bbZ^d \\
& \subseteq U_z = z + [-KL+1,KL-1)^d \cap \bbZ^d.
\end{split}\end{equation}
For a given $\omega \in \Omega_\lambda$, the collection of boxes $U_z$ is used to obtain an $\omega$-dependent decomposition of the Gaussian free field into an harmonic average $\xi^{\omega,U_z}$ and a local field  $\psi^{\omega,U_z}$ (recall~\eqref{eq:HarmonicAverageDef} and~\eqref{eq:LocalFieldDef} for the respective definitions). For $z \in \bbL$, we write
\begin{align}
\label{eq:Sec6_Decomposition_first}
\xi^{\omega,z}_x & = \xi^{\omega, U_z}_x, \qquad x \in \bbZ^d, \\
\label{eq:Sec6_Decomposition_second}
\psi^{\omega,z}_x & = \psi^{\omega, U_z}_x, \qquad x \in \bbZ^d.
\end{align}
Evidently, one has the following property, which follows directly from~\eqref{eq:DomainMP} and the fact that under $\bbP^\omega$, all considered fields are jointly Gaussian:
 \begin{equation}
 \label{eq:CollectionC}
\begin{minipage}{0.9\linewidth}
 If $\cC \subseteq \bbL$ is a collection of sites with mutual $| \cdot |_\infty$-distance at least $(4K+1)L$, the centered Gaussian fields $\psi^{\omega,z}$, $z \in \cC$ are independent, and also independent from the collection $(\xi^{\omega,z}_{x})_{x \in U_z, z \in \cC }$. 
\end{minipage}
\end{equation}
We attach to such a collection $\cC$ a set of boxes $C \subseteq \bbZ^d$ having side-length $L$ as well as a collection of functions $\cM$ which correspond at an informal level to assigning to each box in $B_z \subseteq C$ a point in its $3L$-neighborhood. Specifically, we define
\begin{equation}
\label{eq:DefinitionCandM}
\begin{cases}
& C = \bigcup_{z \in \cC} B_z, \\
& \cM = \{ m \in  (\bbZ^d)^\cC \, : \, m(z) \in D_z \text{ for all }z \in \cC\}.
\end{cases}
\end{equation}

We now introduce two Gaussian fields that track the behavior of the harmonic averages associated with the collection $\cC$, and develop controls on their variance and expectation that are uniform in $\omega \in \Omega_\lambda$. To simplify notation, we introduce the quantity
\begin{equation}
\nu^\omega(z) = \frac{e_C^\omega(B_z)}{\capa^\omega(C)}, \qquad z \in \cC.
\end{equation}
Throughout the remainder of this section, we consider 
\begin{equation}\begin{split}
\eta : \bbR^d \rightarrow \bbR, & \text{ continuous, with compact support}, \\
\eta_N(x) & = \eta\left(\frac{x}{N} \right),\qquad x \in\bbZ^d
\end{split}
\end{equation}
and suppress dependence on $\eta$ in the notation. 

The relevant Gaussian fields associated with $m \in \cM$ are 
\begin{equation}
\label{eq:ZmDefinition}
\begin{split}
Z^\omega_m & = \sum_{z \in \cC} \nu^\omega(z)\xi^{\omega,z}_{m(z)}, \\
Z^\omega_{m,\beta,\rho} & = (1+\rho) Z_{m}  - \beta \langle \bbX_N, \eta \rangle, \qquad \beta\geq 0, \rho \in \bbR
\end{split}
\end{equation}
(see~\eqref{eq:EmpiricalGFFMeasureDef} for the definition of $\bbX_N$).
Moreover, we consider
\begin{equation}
\label{eq:ZomegaDefinition}
Z^\omega = \inf_{m \in \cM} Z^{\omega}_m.
\end{equation}
In the next theorem, which generalizes Theorem 4.2 of~\cite{sznitman2015disconnection} and is similar in spirit to Proposition 3.3 of~\cite{chiarini2019entropic}, we develop the relevant uniform controls on the variance of $Z_{m,\beta,\rho}^\omega$ and the expectation of $Z^\omega$. 
\begin{theorem}
Let $L = L(N)$ be an increasing sequence of integers with $L = o(N)$, and $\rho = \rho(N)$ a sequence of reals with $|\rho| < 1$. It holds that
\label{thm:ZmMainBound} 
\begin{equation}
\label{eq:VarZmBound}
\limsup_{N \rightarrow \infty} \sup_{\omega\in \Omega_\lambda} \sup_{\cC} \sup_{m \in \cM} \Big[ N^{d-2} \text{\normalfont Var}^\omega(Z^\omega_{m,\beta,\rho}) - \frac{N^{d-2}}{\capa^\omega(C)}\widetilde{\alpha}_{K,L,\beta,\rho} \Big]\leq \beta^2 c_{10}(\eta), 
\end{equation}
where $\widetilde{\alpha}_{K,L,\beta,\rho}$ is given by: 
\begin{align}
\label{eq:alphaKLbetarho}
\widetilde{\alpha}_{K,L,\beta,\rho} & = (1+\rho)^2  - 2\beta(1+\rho) \tfrac{1}{N^d} \langle \eta_N, h^\omega_C \rangle_{\bbZ^d} + \cR_{K,L,\beta,\rho}, \text{ where } \\
|\cR_{K,L,\beta,\rho}| & \leq (1+\rho)^2 U(K,L)+  \beta c_{11}(\eta)  \big(U(K,L) + \tfrac{(KL)^d}{N^d} |\cC| + \tfrac{(KL)^d}{N^d}\capa^\omega(C) \big),
\end{align}
and $\lim_K \limsup_L U(K,L) = 0$.
Furthermore, one has the upper bound
\begin{equation}
\label{eq:ExpInfimumBound}
\sup_{\cC} \sup_{\omega \in \Omega_\lambda} | \bbE^\omega[Z^\omega]|\left( \frac{\capa^\omega(C)}{|\cC|} \right)^{\frac{1}{2}} \leq \frac{c_{12}}{K^{c_{13}}}.
\end{equation} 
\end{theorem}
\begin{proof}
We start with the proof of~\eqref{eq:VarZmBound}. Note that for $\omega \in \Omega_\lambda$, $\beta \geq 0$, $|\rho| <1$ and a collection $\cC$ as in~\eqref{eq:DefinitionCandM}, the field $Z^\omega_{m,\beta,\rho}$ (indexed by $m \in \mathcal{M}$) is a centered Gaussian field, so one has 
\begin{equation}
\label{eq:DecompositionVarZ}
\text{Var}^\omega(Z^\omega_{m,\beta,\rho}) = (1+\rho)^2 \text{Var}^\omega(Z_m^\omega) + \beta^2 \cH_N^\omega - 2(1+\rho)\beta \cG_N^\omega,
\end{equation}
where we defined the shorthand notation
\begin{equation}
\cG_N^\omega = \bbE^\omega[Z_m^\omega \langle \bbX_N,\eta \rangle], \qquad \cH_N^\omega = \bbE^\omega[\langle \bbX_N, \eta \rangle^2].
\end{equation}
We now develop upper bounds on the first two summands in~\eqref{eq:DecompositionVarZ} and a lower bound on the last summand.

By the quenched upper bound~\eqref{eq:QuenchedGFEstimate}, we find that
\begin{equation}
\label{eq:H_NBound}
\begin{split}
\limsup_{N \rightarrow \infty} \sup_{\omega \in \Omega_\lambda} N^{d-2}  \bbE^\omega[\langle \bbX_N,\eta \rangle^2] & \leq \limsup_{N \rightarrow \infty} \frac{c_3}{N^{d+2}} \sum_{x,y \in \bbZ^d} \frac{\eta\left(\frac{x}{N} \right)\eta\left(\frac{y}{N} \right)}{|x-y|^{d-2} \vee 1} \\
& \leq \lim_{N \rightarrow \infty} \frac{c_3 c'}{N^{2d}} \sum_{\stackrel{x, y \in \bbZ^d}{x \neq y}} \eta\left(\frac{x}{N} \right)\eta\left(\frac{y}{N} \right)g_{BM}\left(\frac{x}{N}, \frac{y}{N}\right) \\
& = c_3 c' E(\eta) = c_{10}(\eta),
\end{split} 
\end{equation}
where $g_{BM}$ denotes the Green function of standard Brownian motion and 
\begin{equation}
E(\eta) = \iint \eta(x)\eta(y)g_{BM}(x,y)\, \De x \De y
\end{equation}
is the energy associated with the function $\eta$. 

We now turn to the bound on $\text{Var}^\omega(Z^\omega_m)$, which proceeds along the lines of (4.18)--(4.25) in~\cite{sznitman2015disconnection}. 
Let $\cC$ be a collection as in~\eqref{eq:CollectionC} and assume $z \neq z'$ are elements in $\cC$. We see that for $x \in D_z$, $x' \in D_{z'}$:
\begin{equation}
\begin{split}
\bbE^\omega[\xi^{\omega,z}_x \xi^{\omega,z'}_{x'}] & = \sum_{y,y' \in \bbZ^d} P^\omega_x[X_{T_{U_{z}}} = y]P^\omega_{x'}[X_{T_{U_{z'}}} = y] g^\omega(y,y') \\
& \stackrel{\eqref{eq:DecompositionGreenKilledGreen}}{=} g^\omega(x,x'),
\end{split}
\end{equation}
whereas for $z = z' \in \cC$, and $x,x' \in D_z$ one has
\begin{equation}
\label{eq:SecondCaseHarmonicCovariance}
\bbE^\omega[\xi^{\omega,z}_x\xi^{\omega,z}_{x'}] = \sum_{y \in \bbZ^d} P_{x'}^\omega[X_{T_{U_z}} = y]g^\omega(x,y),
\end{equation} 
see also (4.18) and (4.19) in~\cite{sznitman2015disconnection} for a similar argument. We obtain from the definition~\eqref{eq:ZmDefinition} that
\begin{equation}\label{eq:variance_continue}
\text{Var}^\omega(Z_m^\omega) = \sum_{z \in \cC} \nu^\omega(z)^2\bbE^\omega[(\xi^{\omega,z}_{m(z)})^2] + \sum_{z\neq z' \in \cC} \nu^\omega(z)\nu^\omega(z') g^\omega(m(z),m(z')).
\end{equation}
To control the last term and similar expressions, we introduce the quantities
\begin{equation}
\begin{split}
\gamma^\omega(K,L) & = \sup_{\substack{z,z' \in \bbL \\ |z-z'|_\infty \geq KL}} \sup_{\substack{y \in D_z, y' \in D_{z'} \\ x \in B_z, x' \in B_{z'}} }  \frac{g^\omega(y,y')}{g^\omega(x,x')} (\geq 1), \\
\widetilde{\gamma}^\omega(K,L) & = \inf_{\substack{z,z' \in \bbL \\ |z-z'|_\infty \geq KL}} \inf_{\substack{y \in D_z, y' \in D_{z'} \\ x \in B_z, x' \in B_{z'}} }  \frac{g^\omega(y,y')}{g^\omega(x,x')} (\leq 1).
\end{split}
\end{equation}
We shall now prove that 
\begin{align}
\label{eq:ClaimGamma}
\lim_{K \rightarrow \infty} \limsup_{L \rightarrow \infty} \sup_{\omega \in \Omega_\lambda} \gamma^\omega(K,L) & = 1, \text{ and }\\
\label{eq:ClaimTildeGamma}
\lim_{K \rightarrow \infty} \liminf_{L \rightarrow \infty} \inf_{\omega \in \Omega_\lambda} \widetilde{\gamma}^\omega(K,L) &= 1.
\end{align}
To this end, fix $\omega \in \Omega_\lambda$, $L \geq 1$, $K \geq c$, $z, z' \in \bbL$ with $|z-z'|_\infty \geq KL$ and $y \in D_z$, $y' \in D_{z'}$, $x \in B_z$, $x' \in B_{z'}$. We write, using the symmetry of $g^\omega(\cdot,\cdot)$ and the fact that $g^\omega > 0$: 
\begin{equation}
\label{eq:ProofConvergenceGammaStep1}
\frac{g^\omega(y,y')}{g^\omega(x,x')} -1 = \frac{g^\omega(x',y) - g^\omega(x',x)}{g^\omega(x,x')} + \frac{g^\omega(y,y') - g^\omega(y,x')}{g^\omega(x,x')}.
\end{equation}
By~\eqref{eq:harmonicityOfGreen}, $g^\omega(x', \cdot)$ is $\omega$-harmonic in $B^{(1)}(z,  \frac{1}{2} |z-z'|_\infty)$ (recall that $B^{(1)}(x_0,r)$ denotes the closed ball with $r \geq 0$ around $x_0 \in \bbZ^d$ with respect to the graph distance on $\bbZ^d$). Also, $B^{(1)}(z,  \frac{1}{4} |z-z'|_\infty)$ contains both $x$ and $y$ if $K \geq c$. By~\eqref{eq:HoelderRegProperty} there exists a constant $\tau > 0$ only depending on $\lambda$, for which one has the $\tau$-H\"older bound for $g^\omega(x',\cdot)$:
\begin{equation}
\label{eq:ProofConvergenceGammaStep2}
\begin{split}
& \bigg\vert\frac{g^\omega(x',y) - g^\omega(x',x)}{g^\omega(x,x')}\bigg\vert  \leq C {\bigg(\frac{|x-y|_\infty}{\tfrac{1}{4} |z-z'|_\infty}\bigg)}^\tau \cdot \frac{\sup_{w \in B^{(1)}(z,\frac{1}{4} |z-z'|_\infty)} g^\omega(x',w) }{g^\omega(x,x')} \\
&  \qquad\qquad\qquad  \stackrel{\eqref{eq:QuenchedGFEstimate}}{\leq} C' \frac{1}{K^\tau}  (|z-z'|_\infty + 2L)^{d-2} \sup_{w \in B^{(1)}(z,\frac{1}{4} |z-z'|_\infty)} g^\omega(x',w) \\
& \qquad\qquad\qquad \stackrel{\eqref{eq:QuenchedGFEstimate}}{\leq}  C'' \frac{1}{K^\tau} \bigg( \frac{|z-z'|_\infty + 2L}{|z-z'|_\infty} \bigg)^{d-2},
\end{split} 
\end{equation}
having used in the last step that for $x' \in B_{z'}$, $w \in B^{(1)}(z, \frac{1}{4}|z-z'|_\infty)$, we have $|x'-w| \geq \frac{1}{2}|z-z'|_\infty$. We therefore obtain that for any $L \geq 1$ and $K \geq c$, one has
\begin{equation}
\label{eq:ProofConvergenceGammaStep3}
\sup_{ \omega \in \Omega_\lambda}\sup_{\substack{z,z' \in \bbL \\ |z-z'|_\infty \geq KL}} \sup_{\substack{y \in D_z, y' \in D_{z'} \\ x \in B_z, x' \in B_{z'}} } \bigg\vert\frac{g^\omega(x',y) - g^\omega(x',x)}{g^\omega(x,x')}\bigg\vert \leq  \frac{C''}{K^\tau} \bigg(1 + \frac{2}{K} \bigg)^{d-2},
\end{equation}
which vanishes upon taking $\limsup_L$, and $\lim_K$. By the same argument, one has the upper bound
\begin{equation}
\label{eq:ProofConvergenceGammaStep4}
\sup_{ \omega \in \Omega_\lambda}\sup_{\substack{z,z' \in \bbL \\ |z-z'|_\infty \geq KL}} \sup_{\substack{y \in D_z, y' \in D_{z'} \\ x \in B_z, x' \in B_{z'}} } \bigg\vert\frac{g^\omega(y,y') - g^\omega(y,x')}{g^\omega(x,x')}\bigg\vert \leq  \frac{\widetilde{C}}{K^\tau} \bigg(1 + \frac{2}{K} \bigg)^{d-2}.
\end{equation}
The claim~\eqref{eq:ClaimGamma} then follows upon taking the respective suprema of the absolute values in~\eqref{eq:ProofConvergenceGammaStep1}, using the triangle inequality and~\eqref{eq:ProofConvergenceGammaStep3} and~\eqref{eq:ProofConvergenceGammaStep4}, and sending first $L$ and then $K$ to infinity. The claim~\eqref{eq:ClaimTildeGamma} follows in a similar manner. 

We now resume estimating~\eqref{eq:variance_continue}. Note that for $\omega \in \Omega_\lambda$, $z \in \cC$ and any $y \in D_z$, we have
\begin{equation}\label{eq:OneBoxVarianceBound}
\bbE^\omega[(\xi_{y}^{\omega,z})^2] \stackrel{\eqref{eq:SecondCaseHarmonicCovariance}}{ =} \sum_{y \in \bbZ^d} P_x^\omega[X_{T_{U_z}} = y] g^\omega(x,y) \stackrel{\eqref{eq:QuenchedGFEstimate}}{\leq} \frac{c}{(KL)^{d-2}},
\end{equation}
since $d_\infty(\partial U_z, D_z) \geq (K-3)L$. Moreover, for any $z \in \cC$, we have the upper bound
\begin{equation}
\begin{split}
\nu^\omega(z)& = \frac{1}{\capa^\omega(C)} \sum_{x \in B_z} P^\omega_x[\widetilde{H}_C = \infty] \omega_x \\
& \leq \frac{1}{\capa^\omega(C)} \sum_{x \in B_z} P^\omega_x[\widetilde{H}_{B_z} = \infty] \omega_x = \frac{\capa^\omega(B_z)}{\capa^\omega(C)},
\end{split}
\end{equation}
using the definition of the equilibrium measure~\eqref{eq:DefEqMeasure}. We therefore infer the upper bound 
\begin{equation}
\label{eq:UpperBoundVarianceZ}
\begin{split}
\text{Var}^\omega(Z^\omega_m) & \leq \frac{c}{K^{d-2}} \frac{1}{\capa^\omega(C)} + \frac{\gamma^\omega(K,L)}{\capa^\omega(C)^2} \sum_{x,x' \in C, x \neq x'} e_C^\omega(x)e_C^\omega(x')g^\omega(x,x') \\
& = \frac{1}{\capa^\omega(C)} \Big( \frac{c}{K^{d-2}} + \gamma^\omega(K,L) \Big).
\end{split}
\end{equation} 
where for the first term we used that $\sum_{z \in \cC} \nu^\omega(z) = 1$ together with~\eqref{eq:QuenchedBoxCapacityEstimate}. 

We now derive a lower bound on $\cG^\omega_N$. We decompose $\eta = \eta^+ - \eta^-$ and consequently obtain a decomposition $\cG^{\omega}_N =  \cG^{+,\omega}_N - \cG^{-,\omega}_N$, where
\begin{equation}
\cG_N^{\pm,\omega} = \bbE^\omega[Z_m^\omega \langle \bbX_N,\eta^\pm \rangle].
\end{equation}
We first give a lower bound on $\cG^{+,\omega}_N$. We introduce $\cU = \bigcup_{z\in \cC} U_z$. For every $x\notin \cU$ we have $d_\infty(\{x\},\cC)\geq KL$ and $\bbE^\omega [\xi^{\omega,z}_{m(z)}\varphi_x] = g^\omega(m(z),x)$. Thus,
\begin{equation}
\label{eq:EstimationGN}
\begin{split}
  \cG_N^{+,\omega} & = \tfrac{1}{N^d} \sum_{x\in \bbZ^d} \eta^+\big(\tfrac{x}{N}\big) \sum_{z\in \cC} \bbE^\omega [\xi^{\omega,z}_{m(z)}\varphi_x] \nu^\omega(z)\\
  & \geq \tfrac{1}{N^d \capa^\omega(C)} \sum_{x\in \cU^c} \eta^+\big(\tfrac{x}{N}\big) \sum_{z\in \cC} g^\omega(m(z),x) e^\omega_C(B_z)\\
  & \geq \tfrac{\widetilde{\gamma}^\omega(K,L)}{N^d \capa^\omega(C)} \sum_{x\in \cU^c} \eta^+\big(\tfrac{x}{N}\big) \sum_{x'\in C} g^\omega(x,x') e^\omega_C(x')\\
  & \stackrel{\eqref{eq:LastExDecomp}}{\geq} \tfrac{\widetilde{\gamma}^\omega(K,L)}{\capa^\omega(C)}\Big(\tfrac{1}{N^d} \sum_{x \in \bbZ^d}  \eta^+\big(\tfrac{x}{N}\big) h^\omega_C(x)
  - \tfrac{1}{N^d}\sum_{x\in \cU} \eta^+\big(\tfrac{x}{N}\big)\Big).
\end{split}
\end{equation}
Observe that $|\cU| \leq c (KL)^d |\cC|$. Next, we give an upper bound for $\cG^{-,\omega}_N$. Note that for $x \in \cU$, there is a unique $z_x \in \cC$ such that $|z_x-z|_\infty \leq KL$. Therefore, we have
\begin{equation}
\label{eq:EstimationGN_Minus}
\begin{split}
  \cG_N^{-,\omega} & = \tfrac{1}{N^d} \sum_{x\in \bbZ^d} \eta^-\big(\tfrac{x}{N}\big) \sum_{z\in \cC} \bbE^\omega [\xi^{\omega,z}_{m(z)}\varphi_x] \nu^\omega(z)\\
  & \leq \tfrac{\gamma^\omega(K,L)}{N^d \capa^\omega(C)} \sum_{x\in \bbZ^d} \eta^-\big(\tfrac{x}{N}\big) \sum_{x'\in C} g^\omega(x,x') e^\omega_C(x') \\
  & + \tfrac{1}{N^d \capa^\omega(C) } \sum_{x \in \cU} \eta^- \big(\tfrac{x}{N}\big)  g^\omega(m(z_x),x) e^\omega_C(B_{z_x}) \\
    & \stackrel{\eqref{eq:QuenchedGFEstimate}}{\leq} \tfrac{\gamma^\omega(K,L)}{ \capa^\omega(C)} \Big( \tfrac{1}{N^d} \sum_{x\in \bbZ^d} \eta^-\big(\tfrac{x}{N}\big)h^\omega_C(x)+ \tfrac{c_3}{N^d} \sum_{x \in \cU} \eta^- \big(\tfrac{x}{N}\big) e^\omega_C(B_{z_x}) \Big).
\end{split}
\end{equation}
Note that $\tfrac{c_3}{N^d} \sum_{x \in \cU} \eta^- \big(\tfrac{x}{N}\big) e^\omega_C(B_{z_x})\leq c \|\eta\|_\infty\tfrac{(KL)^d}{N^d} \capa^\omega(C)$. We readily get the claim~\eqref{eq:alphaKLbetarho}, upon combining~\eqref{eq:DecompositionVarZ},~\eqref{eq:H_NBound},~\eqref{eq:UpperBoundVarianceZ},~\eqref{eq:EstimationGN} and~\eqref{eq:EstimationGN_Minus}, collecting the error terms,  and setting
\begin{equation}
\begin{split}
U(K,L) & = \bigg(\frac{c}{K^{d-2}}+ \sup_{\omega \in \Omega_\lambda} \gamma^\omega(K,L) - 1 \bigg) \vee \bigg( 1- \inf_{\omega \in \Omega_\lambda} \widetilde{\gamma}^\omega(K,L) \bigg),
\end{split}
\end{equation}
in view of~\eqref{eq:ClaimGamma},~\eqref{eq:ClaimTildeGamma}.  

It remains to prove~\eqref{eq:ExpInfimumBound}, for which we use an argument based on the metric entropy method. We claim that there exists $\tau > 0$ such that for $m,m' \in \mathcal{M}$:
\begin{equation}
\label{eq:ClaimedUpperBoundGaussianDistance}
\bbE^\omega[(Z^\omega_m - Z^\omega_{m'})^2]\capa^\omega(C) \leq c_{12} \left(\frac{| m - m'|_\infty^2}{(KL)^2}\right)^\tau,
\end{equation}
where $| m - m'|_\infty =\sup_{z \in \cC} |m(z) - m'(z)|_\infty$. Indeed, let $\cC$ be as above and $m,m' \in \cM$. By direct calculation, we find
\begin{equation}
\label{eq:UpperBoundGaussianDist}
\begin{split}
\bbE^\omega[(Z^\omega_m - Z^\omega_{m'})^2] = & \sum_{z,z'\in \cC} \nu^\omega(z)\nu^\omega(z') \sum_{x,x'}(P^\omega_{m(z)}[X_{T_{U_z}} = x] - P^\omega_{m'(z)}[X_{T_{U_z}} = x]) \\
&(P^\omega_{m(z')}[X_{T_{U_{z'}}} = x'] - P^\omega_{m'(z')}[X_{T_{U_{z'}}} = x'])g^\omega(x,x').
\end{split}
\end{equation}
Since for a given $x \in \partial U_z$, $P^\omega_\cdot[X_{T_{U_z}} = x]$ is $\omega$-harmonic and non-negative in $U_z$, we can combine the Harnack inequality (see Theorem 1 in~\cite{delmotte1997inegalite}) with the $\tau$-H\"older regularity (see~\eqref{eq:HoelderRegProperty}) to obtain for $y,y' \in D_z$ (and $K \geq c$):
\begin{equation}
|P^\omega_y[X_{T_{U_z}} = x] - P^\omega_{y'}[X_{T_{U_z}} = x]|\leq  c \left(\frac{|y-y'|_\infty}{KL}\right)^\tau P^\omega_y[X_{T_{U_z}} = x].
\end{equation}
Similarly, one can apply the analogous bound to $P_\cdot[X_{T_{U_{z'}}} = x']$ where $x' \in \partial U_{z'}$, which upon insertion into~\eqref{eq:UpperBoundGaussianDist} yields
\begin{equation}
\bbE^\omega[(Z^\omega_m - Z^\omega_{m'})^2] \leq c \left(\frac{|m-m'  |_\infty^2}{(KL)^2}\right)^\tau \bbE^\omega[(Z_m^\omega)^2] \leq \widetilde{c}\left(\frac{|m-m'  |_\infty^2}{(KL)^2}\right)^\tau \frac{1}{\capa^\omega(C)},
\end{equation}
similarly as in (4.29) of~\cite{sznitman2015disconnection}, which is precisely~\eqref{eq:ClaimedUpperBoundGaussianDistance}.

If we introduce now the scaled Gaussian process $\widetilde{Z}^
\omega_m = \sqrt{\capa^\omega(C)} Z^\omega_m$ for $m \in \cM$, we obtain the upper bound (for each $\omega \in \Omega_\lambda$)
\begin{equation}
\bbE^\omega[(\widetilde{Z}^\omega_m - \widetilde{Z}^\omega_{m'})^2]^{\frac{1}{2}} \leq \frac{c}{(KL)^\tau} |m-m' |_{\infty}^\tau \leq \frac{7^\tau c}{K^\tau},
\end{equation}
for $m, m' \in \cM$. By a straightforward adaptation of the arguments leading up to (4.33) of~\cite{sznitman2015disconnection}, which uses Theorem 1.3.3 of~\cite{adler2007random}, we therefore infer 
\begin{equation}
|\sqrt{\capa^\omega(C)} \bbE^\omega[Z^\omega]| \leq \frac{c_{12}}{K^\tau}\sqrt{|\cC|},
\end{equation}
Rearranging and taking the suprema over $\omega \in \Omega_\lambda$ and all $\cC$ yields~\eqref{eq:ExpInfimumBound} (with $c_{13} = \tau$). 
\end{proof}

\begin{rem}
Theorem~\ref{thm:ZmMainBound} is crucial in the proof of the asymptotic upper bounds in the next section. For the upper bound on the disconnection event, we  take $\rho = \beta = 0$, see~\eqref{eq:BoundInTermsofZ} and~\eqref{eq:Borell-TIS_onlyDisc}. In the proof of Proposition~\ref{prop:MainPropProofSec7}, which is pivotal for treating the case of both disconnection and the entropic repulsion-type event under the probability in~\eqref{eq:EntropicRepulsionIntro}, we instead consider $\rho = \beta\frac{1}{N^d}  \langle \eta_N, h^\omega_C\rangle_{\bbZ^d}$, with $\beta$ small enough.
\end{rem}

\section{Asymptotic upper bound on disconnection and entropic repulsion}
\label{sec:UpperBounds}

In this section we state and prove in the main Theorem~\ref{thm:UpperBound} a quenched asymptotic upper bound on the probability of the disconnection event $\cD^\alpha_N$ from~\eqref{eq:DefDisconnection}, in the strongly percolative regime $\alpha < \overline{\alpha}$. This result complements the quenched asymptotic lower bound~\eqref{eq:QuenchedDiscLowerBound}, and if $\overline{\alpha} = \alpha_\ast = \alpha_{\ast\ast}$ holds (which is plausible but open), the respective rates would in fact be matching for sets $A$ that are regular in the sense that $\capa^\homo(\mathring{A}) = \capa^\homo(A)$, see Remark~\ref{rem:on_regularity} concerning this condition. Moreover, we derive in Theorem~\ref{thm:EntropicRepulsion} a quenched asymptotic upper bound on the probability that $\cD^\alpha_N$ occurs and that $\langle \bbX_N, \eta \rangle$ (the ``average'' of the Gaussian free field over some fixed continuous function $\eta$ with compact support) deviates from $\langle \mathscr{H}^\alpha_{\mathring{A}}, \eta \rangle$ by a certain amount $\Delta > 0$. This bound comes with an additional cost $c_1(\Delta,\alpha,\eta)$ to the right hand side of~\eqref{eq:QuenchedDiscUpperBound} in Theorem~\ref{thm:UpperBound}, see~\eqref{eq:EntropicRepulsionSec7}, which is uniform in $\omega \in \Omega_\lambda$.

These results are obtained by adapting a coarse-graining procedure taken from~\cite{nitzschner2017solidification}, see also~\cite{chiarini2019entropic}, together with the quenched Gaussian bounds of the previous section, the solidification estimates of Section~\ref{sec:Solidification}, and the homogenization results in Corollary~\ref{cor:Capacity_convergence_NoKilling}. 

\begin{theorem}
\label{thm:UpperBound}
 Let $\alpha < \overline{\alpha}$, then for $P$-a.e.\ $\omega\in \Omega_\lambda$ and $A \subseteq \bbR^d$ compact with non-empty interior, it holds that
\begin{equation}
\label{eq:QuenchedDiscUpperBound}
    \limsup_{N\to\infty} \frac{1}{N^{d-2}} \log \bbP^\omega[\cD_N^\alpha] \leq -\frac{1}{2} (\overline{\alpha}- \alpha)^2 \capa^{\homo}(\mathring{A}).
    \end{equation}
\end{theorem}

\begin{theorem}
 Let $\alpha < \overline{\alpha}$, and $\eta : \bbR^d \rightarrow \bbR$ continuous with compact support, then for $P$-a.e.\ $\omega\in \Omega_\lambda$ and $A \subseteq \bbR^d$ compact with non-empty interior, it holds that
\label{thm:EntropicRepulsion}

\begin{equation}
\label{eq:EntropicRepulsionSec7}
    \begin{split}
\limsup_{N \rightarrow \infty} \frac{1}{N^{d-2}} \log \, & \bbP^\omega\Big[ \big\vert \langle \bbX_N, \eta \rangle - \langle \mathscr{H}^\alpha_{\mathring{A}}, \eta \rangle \big\vert \geq \Delta ; \cD^\alpha_N \Big] \\
& \leq -\frac{1}{2}(\overline{\alpha}-\alpha)^2 \capa^{\homo}(\mathring{A}) - c_1(\Delta,\alpha,\eta).
\end{split}
    \end{equation}
\end{theorem}

 If the equalities $\overline{\alpha} = \alpha_\ast = \alpha_{\ast\ast}$ hold and the set $A$ is regular, one obtains by combining Theorems~\ref{thm:MainLowerBound} and~\ref{thm:EntropicRepulsion} that for $\alpha < \alpha_\ast$ and $P$-a.e.\ $\omega \in \Omega_\lambda$ 
 \begin{equation}
 \limsup_{N \rightarrow \infty} \frac{1}{N^{d-2}} \log \,  \bbP^\omega\Big[ \big\vert \langle \bbX_N, \eta \rangle - \langle \mathscr{H}^\alpha_{A}, \eta \rangle \big\vert \geq \Delta \, \big\vert \, \cD^\alpha_N \Big] \\
 \leq  - c_1(\Delta,\alpha,\eta),
 \end{equation}
 using the same methods as for Corollary 3.2 of~\cite{chiarini2019entropic}.  In particular~\eqref{eq:EntropicRepulsionConditioned} would hold, meaning that the average of the Gaussian free field is  pinned with high probability to the profile function $\mathscr{H}^\alpha_{A}$, conditionally on disconnection.

Recall from the previous section that, for $L \geq 1$, $K \geq 100$ integers and $B_z$, $z \in \bbL$, an $L$-box, we have for a fixed $\omega\in \Omega_\lambda$ the decomposition $\varphi = \xi^{\omega,z} + \psi^{\omega,z}$ into ($\omega$-)harmonic average and ($\omega$-)local field, see~\eqref{eq:Sec6_Decomposition_first},~\eqref{eq:Sec6_Decomposition_second}. We now introduce the notion of $B_z$ being ``$\psi^\omega$-good'' at levels $\gamma>\delta$, where $\delta < \gamma < \overline{\alpha}$. A box $B_z$ is called $\psi^\omega$-good at levels $\gamma > \delta$ if 
\begin{equation}
\label{eq:PsiGoodCond1}
 \begin{minipage}{0.9\linewidth}\begin{center} $B_z \cap \{x \in \bbZ^d \, : \, \psi^{\omega,z}_x \geq \gamma \}$  contains a component of diameter at least  $\frac{L}{10}$,\end{center}
    \end{minipage}
\end{equation}
and for any adjacent box $B_{z'}$ with $z' \in \bbL$, $|z-z'| = L$, one has
\begin{equation}
\label{eq:PsiGoodCond2}
 \begin{minipage}{0.9\linewidth}
two components of $B_z \cap \{x\in \bbZ^d  :  \psi^{\omega,z}_x \geq \gamma\}$ and $B_{z'} \cap \{x\in \bbZ^d : \psi^{\omega,z'}_x \geq \gamma\}$ with diameter at least $\frac{L}{10}$ are connected in $D_z \cap \{x\in \bbZ^d :  \psi^{\omega,z}_x \geq \delta\}$
    \end{minipage}
\end{equation}
($D_z$ is defined in~\eqref{eq:BoxesDefinition}). A box $B_z$ that is not $\psi^\omega$-good at levels $\gamma > \delta$ is then called $\psi^\omega$-bad. Note that these definitions generalize their counterparts for the Gaussian free field with constant conductances, see (5.7) and (5.8) in~\cite{sznitman2015disconnection}.

We also introduce the notion of $B_z$ being ``$\xi^\omega$-good'' at level $a > 0$, by which we mean that
\begin{equation}
\label{eq:XiGoodCond}
\inf_{x \in D_z} \xi^{\omega,z}_x > -a.
\end{equation}
Similarly, a box $B_z$ that is not $\xi^\omega$-good will be called $\xi^\omega$-bad. Both definitions are again analogues of the notion of $h$-good and $h$-bad boxes in the case of the Gaussian free field with constant conductances, see (5.9) of~\cite{sznitman2015disconnection}. In the next proposition, we develop bounds on the probability of the event that a given $L$-box is $\psi^\omega$-bad or $\xi^\omega$-bad.

\begin{prop}
\label{prop:PsiBadUniformMesoscopic}
Let $L \geq 1$, $K \geq 100$, $a > 0$. Then,
\begin{equation}
\label{eq:QuenchedBoundOneBox}
\sup_{\omega \in \Omega_\lambda} \sup_{z \in \bbL} \bbP^\omega\Big[ \sup_{x \in D_z} |\xi^{\omega,z}_x| > a \Big] \leq 2\exp\Big\{- c(KL)^{d-2}\Big(a- \tfrac{c}{K^{c_{13}} L^{\frac{d-2}{2}}} \Big)_+^2 \Big\}.
\end{equation}
For $\overline{\alpha} > \gamma > \delta$ and $P$-a.e.\ $\omega \in \Omega_\lambda$ there exists $\rho > d-1$ such that
\begin{equation}
\label{eq:PsiBadUniformMesoscopic}
\lim_{L\to \infty} \sup_{z \in B(0,L^{\rho}) \cap \bbL }  \frac{1}{\log L} \log \bbP^\omega \big[B_z \textnormal{ is $\psi^\omega$-bad at levels $\gamma > \delta$}\big] = - \infty.
\end{equation} 
\end{prop}

\begin{proof}
We first prove~\eqref{eq:QuenchedBoundOneBox}. Fix $\omega \in \Omega_\lambda$ and $z \in \bbL$. By~\eqref{eq:ExpInfimumBound} (with $|\mathcal{C}| = 1$), one has $|\bbE^\omega[\inf_{x \in D_z} \xi^{\omega,z}_x]| \leq c (K^{c_{13}} L^{\frac{d-2}{2}})^{-1}$, and the same bound holds for $\bbE^\omega[\sup_{x \in D_z} \xi^{\omega,z}_x]$. By~\eqref{eq:OneBoxVarianceBound}, we see that $\text{Var}^\omega(\xi^{\omega,z}_x) \leq c (KL)^{2-d}$. By the Borell-TIS inequality, one obtains the bound
\begin{equation}
\bbP^\omega\Big[ \sup_{x \in D_z} |\xi^{\omega,z}_x| > a \Big] \leq 2\exp\Big\{- c(KL)^{d-2}\Big(a- \tfrac{c}{K^{c_{13}} L^{\frac{d-2}{2}}} \Big)_+^2 \Big\},
\end{equation}
and since the right-hand side does not depend on $z \in \bbL$ and $\omega \in \Omega_\lambda$, we may take the respective suprema, giving~\eqref{eq:QuenchedBoundOneBox}. 

We now turn to the proof of~\eqref{eq:PsiBadUniformMesoscopic}. Let $\overline{\alpha} > \beta_+ > \beta_- > \gamma > \delta$. Note that for $P$-a.e.\ $\omega \in \Omega_\lambda$ there exists some $\rho'  = \rho'(\omega)> d-1$ such that \eqref{eq:NoLargeComponent} holds. Moreover, by~\eqref{eq:QuenchedBoundOneBox}, we see that
\begin{equation}
\label{eq:QuenchedBoundBadnessApplication}
\lim_{L\to\infty}\sup_{\omega \in \Omega_\lambda} \sup_{z \in \bbL}  \frac{1}{\log L} \log \bbP^\omega \Big[ \sup_{x \in D_z} | \varphi_x - \psi^{\omega,z}_x | > (\beta_+ - \gamma) \wedge \tfrac{\gamma-\delta}{4} \Big] = -\infty.
\end{equation}
If for $z \in B(0,L^{\rho'}) \cap \bbL$, $B_z \cap \{\psi^{\omega,z} \geq \gamma\}$ has no component of diameter at least $\frac{L}{10}$, then either $B_z \cap E^{\geq \beta_+}$ has no component of diameter at least $\frac{L}{10}$, or the event under the probability in~\eqref{eq:QuenchedBoundBadnessApplication} occurs. By a union bound we therefore find that for $P$-a.e.\ $\omega \in \Omega_\lambda$, 
 \begin{equation}
 \label{eq:PsiBadUnlikely1}
\lim_{L\to \infty} \sup_{z \in B(0,L^{\rho'}) \cap \bbL }  \frac{1}{\log L}\log  \bbP^\omega \Bigg[  \begin{minipage}{0.4\textwidth}
    $B_z  \cap \{x \in \bbZ^d \, : \, \psi^{\omega,z}_x \geq \gamma \}$  contains no component of diameter at least $\tfrac{L}{10}$
\end{minipage} \Bigg] = -\infty.
\end{equation}
 Now let $B_z$ and $B_{z'}$ be neighboring $L$-boxes with $z,z' \in \bbL$. Suppose that both $\sup_{x \in D_z} |\varphi_x - \psi^{\omega,z}_x| \leq \frac{1}{4}(\gamma-\delta)$ and $\sup_{x \in D_{z'}} |\varphi_x - \psi^{\omega,z'}_x| \leq \frac{1}{4}(\gamma-\delta)$ hold, so connected sets in $B_z \cap \{\psi^{\omega,z} \geq \gamma\}$ and $B_{z'} \cap \{\psi^{\omega,z'} \geq \gamma\}$ with diameter each at least $\frac{L}{10}$ are connected sets in $B_z \cap E^{\geq \frac{\gamma+\delta}{2}}$ and $B_{z'} \cap E^{\geq \frac{\gamma+\delta}{2}}$, respectively.
 
 Using~\eqref{eq:NoLargeComponentsDisconn} with $\beta_+$ replaced by $\frac{\delta+\gamma}{2}$ and $\beta_-$ replaced by $\frac{\gamma+3\delta}{4}$, for $P$-a.e.\ $\omega \in \Omega_\lambda$ there exists some $\widetilde{\rho} = \widetilde{\rho}(\omega)> d-1$ such that one has 
     \begin{equation}
     \label{eq:NoConnectionPhi}
        \lim_{L\to \infty}\sup_{z\in B(0,L^{\widetilde{\rho}}) \cap \bbL } \frac{1}{\log L} \log \bbP^{\omega}\left[
        \begin{minipage}{0.5\textwidth}
            there exist components of $B_z\cap E^{\geq \frac{\gamma+\delta}{2} }$ and $B_{x+z}\cap E^{\geq \frac{\gamma+\delta}{2}}$ with diameter at least $\tfrac{L}{10}$ which are not connected in $D_x\cap E^{\geq  \frac{\gamma+3\delta}{4} }$
        \end{minipage}
        \right] = -\infty.
    \end{equation}
    
If there exist two connected components of $B_z \cap \{\psi^{\omega,z} \geq \gamma\}$ and $B_{z'} \cap \{\psi^{\omega,z'} \geq \gamma\}$ with diameter each at least $\frac{L}{10}$ that are not connected in $D_z \cap \{\psi^{\omega,z} \geq \delta\}$, then either the event under the probability in~\eqref{eq:NoConnectionPhi} occurs, or $\sup_{x \in D_z} |\varphi_x - \psi^{\omega,z}_x| > \frac{\gamma - \delta}{4}$ or $\sup_{x \in D_{z'}} |\varphi_x - \psi^{\omega,z'}_x| > \frac{\gamma - \delta}{4}$ occurs. By~\eqref{eq:NoConnectionPhi},~\eqref{eq:QuenchedBoundBadnessApplication} and a union bound, we see that that for $P$-a.e.\ $\omega \in \Omega_\lambda$

\begin{equation}
\label{eq:PsiBadUnlikely2}
\lim_{L\to\infty}  \sup_{z \in B(0,L^{\widetilde{\rho}}) \cap \bbL } \frac{1}{\log L} \log  \bbP^\omega  \left[
    \begin{minipage}{0.50\textwidth}
        there are components of $B_z \cap \{x\in \bbZ^d :  \psi^{\omega,z}_x \geq \gamma\}$  and $B_{z'} \cap \{x\in \bbZ^d\, : \,\psi^{\omega,z'}_x \geq \gamma\}$  with diameter at least $\tfrac{L}{10}$ not connected in $D_z \cap \{x\in \bbZ^d :  \psi^{\omega,z}_x \geq \delta\}$ 
    \end{minipage}\right] = -\infty.
\end{equation}  
  
Setting $\rho = \rho' \wedge \widetilde{\rho}$, we infer from~\eqref{eq:PsiBadUnlikely1} and~\eqref{eq:PsiBadUnlikely2}, that for $P$-a.e.\ $\omega \in \Omega_\lambda$ the super-polynomial bound~\eqref{eq:PsiBadUniformMesoscopic} holds. 
\end{proof}

 We will now recall two facts from~\cite{sznitman2015disconnection} concerning a connectivity property of the level-set involving $\psi^\omega$-good boxes and a super-exponential bound on the probability that many boxes in a certain range are $\psi^\omega$-bad. 
 
 We choose $\gamma > \delta $ and $a > 0$ rationals such that 
 \begin{equation}
  a + \alpha = \delta < \gamma < \overline{\alpha}
 \end{equation}
(we will eventually let $a$ tend to $\overline{\alpha}-\alpha$). We also let $\overline{K} = 4K + 1$ for a given integer $K \geq 100$. Now assume that for $L \geq 1$,
\begin{equation}
\text{ $\cC \subseteq \bbL$ is a subset of points with mutual distance at least $\overline{K}L$.}
\end{equation}
For $\omega \in \Omega_\lambda$, one has the following result, which follows in the exact same way as Lemma 5.3 of~\cite{sznitman2015disconnection} from the independence (under $\bbP^\omega$) of the fields $(\psi^{\omega,z})_{z \in \cC}$, see~\eqref{eq:CollectionC} and the definition of $\psi^\omega$-good and $\xi^\omega$-good boxes, see~\eqref{eq:PsiGoodCond1}--\eqref{eq:XiGoodCond}.

\begin{lem}
\label{lem:Independence}
The events 
\begin{equation}
\{\textnormal{$B_z$ is $\psi^\omega$-good at levels $\gamma > \delta$}\}_{z \in \cC}
\end{equation}
are independent under $\bbP^\omega$. If $B_{z_i}$, $0 \leq i \leq n$ is a sequence of neighboring $L$-boxes, i.e. $(z_i)_{0 \leq i \leq n}$ forms a nearest-neighbor path in $\bbL$, all $\psi^\omega$-good at levels $\gamma > \delta$ and $\xi^\omega$-good at level $a$, then 
there exists a path in $E^{\geq \delta - a} \cap \big(\bigcup_{i = 0}^n D_{z_i} \big)$, starting in $B_{z_0}$ and ending in $B_{z_n}$.
\end{lem} 

As a next step we prove a super-exponential bound as $L \rightarrow \infty$ on the occurrence of a ``large'' number of $\psi^\omega$-bad boxes in a box of size $N_L (\gg L)$ (see~\eqref{eq:NLDefinition}), which is similar in character to Proposition 5.4 in~\cite{sznitman2015disconnection}. 
This will be an instrumental tool for decoupling the ``wavelet part'' encoded by the fields $\psi^\omega$ from the ``undertow'' encoded by the fields $\xi^\omega$ in the study of disconnection events. However, some care is required due to the inhomogeneous nature of the field $\varphi$ under $\bbP^\omega$. Specifically, we need to rely on the fact that for $P$-a.e.\ $\omega$, the probability of a box $B_z$ being $\psi^\omega$-bad at levels $\gamma>\delta$ where $\delta < \gamma < \overline{\alpha}$, has a super-exponential decay \textit{uniformly} over $z$ in a box of size $L^\rho$, centered at the origin, where $\rho > d-1$. This property~\eqref{eq:PsiBadUniformMesoscopic} of the supercritical regime is the motivation in the definition of $\overline{\alpha}$, see also Remark~\ref{rem:OtherCritParam}.

\begin{prop}
\label{prop:BadEventProposition}
For $P$-a.e.\ $\omega \in \Omega_\lambda$, there exists a positive function $\varrho^\omega(L)$ (depending also on $\gamma,\delta,K$) with 
\begin{equation}
\lim_{L\to \infty} \varrho^\omega(L) = 0,
\end{equation}
such that setting 
\begin{align}
\label{eq:NLDefinition}
N_L = L^{d-1} / \log L, \text{ for }L > 1,
\end{align}
one has for $P$-a.e.\ $\omega \in \Omega_\lambda$
\begin{equation}
\label{eq:SuperExpDecay}
\lim_{L \rightarrow \infty} \; N_L^{-(d-2)} \log \bbP^\omega \left[
    \begin{minipage}{0.47\textwidth} \textnormal{there are at least $\varrho^\omega(L) (\frac{N_L}{L})^{d-1}$ columns in
$[-N_L,N_L]^d$ in the direction $e$ containing
a $\psi^\omega$-bad box at levels $ \gamma > \delta$} \end{minipage}\right] = -\infty, 
\end{equation}
where for $e$ vector of the canonical basis of $\bbR^d$ a column in $[-N_L,N_L]^d$ in the direction $e$ refers to the collection of $L$-boxes $B_z$ located at $z \in \bbL$, intersecting $[-N_L,N_L]^d$ with same
projection on the discrete hyperplane $\{x \in \bbZ^d\, : \, x \cdot e = 0\}$, and~\eqref{eq:SuperExpDecay} holds for all vectors $e$ of the canonical basis.
\end{prop}

\begin{proof}
Choose $\omega \in \Omega_\lambda$ such that~\eqref{eq:PsiBadUniformMesoscopic} holds (by Proposition~\ref{prop:PsiBadUniformMesoscopic}, the set of such $\omega$ has full $P$-measure) and denote the event under the probability in~\eqref{eq:SuperExpDecay} by $A^\omega$. Following the proof of Proposition 5.4 in~\cite{sznitman2015disconnection} we denote by $m$ the total number of $L$-boxes belonging to the union of all columns in $[-N_L,N_L]^d$. Then one has for large $L$ that 
\begin{equation}
\label{eq:UpperLowerBoundNumberBoxesColumns}
c_{14} \big( \tfrac{N_L}{L} \big)^d \leq m \leq c_{15}\big( \tfrac{MN_L}{L} \big)^d,
\end{equation}
where $M > 0$ is the fixed number introduced above~\eqref{eq:BlowUpBoxDef}.
Decompose $\bbL$ into the disjoint union 
\begin{equation}
\bbL = \bigcup_{y \in \{0,L,...,(\overline{K}-1)L\}^d } \cL_y, \qquad \cL_y = y + \overline{K}\bbL,
\end{equation}
where $|\{0,L,...,(\overline{K}-1)L\}^d| = \overline{K}^d$. We introduce the quantities (for $L > 1$ and $\rho = \rho(\omega) > d-1$ as above~\eqref{eq:PsiBadUniformMesoscopic})
\begin{align}
\label{eq:etaDef}
\eta^\omega(L) & =  \sup_{z \in B(0,L^{\rho}) \cap \bbL }   \bbP^\omega \big[B_z \textnormal{ is $\psi^\omega$-bad at levels $\gamma > \delta$}\big], \\
\label{eq:rhoDef}
\varrho^\omega(L) & = \sqrt{\tfrac{\log L }{\log(1/\eta^\omega(L))}}, \\
\widetilde{\varrho}^\omega(L) & = \tfrac{\varrho^\omega(L)}{m\overline{K}^d} \big(\tfrac{N_L}{L} \big)^{d-1}
\label{eq:varrhoDef} \stackrel{\eqref{eq:UpperLowerBoundNumberBoxesColumns}}{\in} \big(\tfrac{\varrho^\omega(L)}{c_{15}(\overline{K}M)^d}\tfrac{L}{N_L} ,\tfrac{\varrho^\omega(L)}{c_{14}\overline{K}^d}\tfrac{L}{N_L} \big).
\end{align} 
By~\eqref{eq:PsiBadUniformMesoscopic}, $\lim_{L\rightarrow \infty} \varrho^\omega(L) = 0$, and by~\eqref{eq:NLDefinition}, we immediately obtain that also $\lim_{L\rightarrow \infty} \widetilde{\varrho}^\omega(L) = 0$. Moreover, since $\rho > d-1$, one has $\lim_{L \rightarrow \infty} \frac{N_L}{L^\rho} = 0$, and so
\begin{equation}
\label{eq:N_LBoxContainedInRhoBox}
[-N_L,N_L]^d \cap \bbZ^d \subseteq B(0,L^\rho), \qquad \text{for }L \geq c(\rho).
\end{equation}
 By Lemma~\ref{lem:Independence}, the events $\{\text{$B_z$ is $\psi^\omega$-bad at levels $\gamma > \delta$}\}_{z \in \cL_y}$ are independent under $\bbP^\omega$. We introduce the random variables
\begin{equation}
\mathfrak{X}_z^\omega = \mathbbm{1}_{\{\text{$B_z$ is $\psi^\omega$-bad at levels $\gamma > \delta$}\}}, \qquad z \in \bbL,
\end{equation}
and for fixed $y \in \{0,L,...,(\overline{K}-1)L\}^d$, the collection $\{\mathfrak{X}_z^\omega \}_{z \in \cL_y}$ consists of independent Bernoulli random variables. It follows that for $L$ large, one has
\begin{equation}
\begin{split}
\bbP^\omega[A^\omega] & \leq \bbP^\omega\bigg[\bigcup_{y \in \{0,L,...,\overline{K}L \}^d } \bigg\{\sum_{z \in \cL_y \cap [-N_L,N_L]^d } \mathfrak{X}_z^\omega \geq m\widetilde{\varrho}^\omega(L)  \bigg\} \bigg] \\
& \leq \overline{K}^d \bbP^\omega[\mathfrak{Z}^\omega \geq m\widetilde{\varrho}^\omega(L) ],
\end{split}
\end{equation}
where $\mathfrak{Z}^\omega$ under $\bbP^\omega$ is a Binomial($m,\eta^\omega(L)$)-random variable. Here we used the fact that the cardinality of $\cL_y \cap [-N_L,N_L]^d$ is bounded from above by $m$, see~\eqref{eq:UpperLowerBoundNumberBoxesColumns}, and that $(\mathfrak{X}^\omega_z)_{z \in \cL_y}$ can be stochastically dominated by a set of Bernoulli random variables (indexed by $\cL_y$) with success parameter $\eta^\omega(L)$, using~\eqref{eq:etaDef} and~\eqref{eq:N_LBoxContainedInRhoBox}. 

Using standard bounds for sums of Bernoulli random variables one obtains 
\begin{equation}
\begin{split}
\bbP^\omega[A^\omega] & \leq \overline{K}^d \exp\big(-m I_L^\omega \big),\\ I_L^\omega & = \widetilde{\varrho}^\omega(L) \log \tfrac{\widetilde{\varrho}^\omega(L)}{\eta^\omega(L)} + (1 - \widetilde{\varrho}^\omega(L)) \log \tfrac{1-\widetilde{\varrho}^\omega(L)}{1-\eta^\omega(L)}.
\end{split}
\end{equation}
The claim follows by using that
\begin{equation}
N_L^{d-2} = o(mI_L^\omega), \text{ as } L \rightarrow \infty,
\end{equation}
which follows from the choice of the scale $N_L$ in~\eqref{eq:NLDefinition} and the definitions~\eqref{eq:etaDef} and~\eqref{eq:varrhoDef} (see (5.23)--(5.27) in~\cite{sznitman2015disconnection} for the details of this calculation).  
\end{proof}

We will now move to the proofs of Theorems~\ref{thm:UpperBound} and~\ref{thm:EntropicRepulsion}. We heavily rely on a coarse-graining procedure which appeared first in Section 4 of~\cite{nitzschner2017solidification}, and was used to develop the corresponding versions of the above mentioned theorems in the case of the Gaussian free field with constant, non-random conductances in~\cite{chiarini2019entropic,nitzschner2018entropic}. We introduce some more notation in order to make this coarse-graining procedure explicit. 

For $\delta < \gamma < \overline{\alpha}$, one can select a sequence $(\gamma_N)_{N \geq 1}$ of numbers in $(0,1]$ as in (4.18) of~\cite{nitzschner2017solidification}, in particular,
\begin{equation}\label{eq:gamma_N}
\gamma_N \rightarrow 0,\quad \gamma_N^{\frac{d+1}{2}}/(N^{2-d}\log N) \to \infty,\quad \text{as }N \rightarrow \infty.
\end{equation}
The sequence $(\gamma_N)_{N \geq 1}$ depends implicitly on $\varrho^\omega$ from~\eqref{eq:rhoDef}, but we suppress this dependence in the notation. For $N \geq 1$, we define the two scales
\begin{equation}
\label{eq:ScalesDefinition}
L_0 = L_0(N) = \Big\lfloor (\gamma_N^{-1} N \log N)^{\tfrac{1}{d-1}} \Big\rfloor, \qquad \widehat{L}_0 = \widehat{L}_0(N) = 100 d \left\lfloor \sqrt{\gamma_N} N \right\rfloor,
\end{equation}
together with the corresponding lattices
\begin{equation}
\bbL_0  = L_0 \bbZ^d,  \qquad \widehat{\bbL}_0 = \tfrac{1}{100d} \widehat{L}_0\bbZ^d = \left\lfloor \sqrt{\gamma_N} N \right\rfloor \bbZ^d.
\end{equation}
The lattice $\bbL_0$ corresponds to $\bbL$ in~\eqref{eq:LatticeL_Def} with the choice $L = L_0$. We also recall the definition of boxes $B_z$, $D_z$ and $U_z$ (again with $L = L_0$) from~\eqref{eq:BoxesDefinition}. The scale $\widehat{L}_0$ corresponds to a ``nearly macroscopic'' scale, whereas the scale $L_0$ is the size of the microscopic boxes constituting a discrete porous interface for the discrete blow-up $A'_N$ (of a subset $A' \subseteq \mathring{A}$) in the sense of Section~\ref{sec:Solidification}, which roughly speaking is present when the disconnection event $\cD^\alpha_N$ occurs. 
\vspace{0.3\baselineskip}

\noindent
\emph{Outline of the proof.} The proofs of Theorems~\ref{thm:UpperBound} and~\ref{thm:EntropicRepulsion} will require a multi-step procedure. We outline the set-up in an informal fashion and sketch where the main ingredients of the proof, Corollary~\ref{cor:CapacityDirichletSolidification}, Corollary~\ref{cor:Capacity_convergence_NoKilling}, Theorem~\ref{thm:ZmMainBound}, and Proposition~\ref{prop:BadEventProposition} enter the argument. The first two steps are a routine extension of the corresponding methods pertaining to the Gaussian free field with constant conductances that appeared in Section 3 of~\cite{nitzschner2018entropic} and Section 3 of~\cite{chiarini2019entropic} (see also~\cite{chiarini2020entropic,nitzschner2017solidification,sznitman2019macroscopic}) and are recalled here for the convenience of the Reader.

\vspace{0.3\baselineskip}

\noindent \textbf{1.} \emph{Reduction to an effective disconnection event}:

We consider $\omega$ in a subset of $\Omega_\lambda$ with full $P$-measure and $\gamma < \delta < \alpha (< \overline{\alpha}$). By Lemma~\ref{lem:Independence}, for large $N$ on the disconnection event $\cD^\alpha_N$, there exists an interface of ``blocking'' $L_0$-boxes located between $A_N$ and the complement of $B(0,(M+1)N)$, which are all $\xi^\omega$-bad at level $a = \delta - \alpha$ or $\psi^\omega$-bad at levels $\delta < \gamma$. 

In the first step, we define in~\eqref{eq:BadEventForDisconnection} a ``bad'' event $\cB_N^\omega$ that corresponds to the existence of many $\psi^\omega$-bad boxes at levels $\delta < \gamma$ within a box $B(0,10(M+1)N)$. On account of Proposition~\ref{prop:BadEventProposition}, we show that $\cB_N^\omega$ has a negligible probability for large $N$ (see~\eqref{eq:SuperExponentialBound}). We then introduce the \textit{effective disconnection event} $ \widetilde{\cD}^{\omega,\alpha}_N = \cD^\alpha_N \setminus \cB_N^\omega$, allowing us to restrict our attention to interfaces of blocking $L_0$-boxes that are $\xi^\omega$-bad.   
\vspace{0.3\baselineskip}

\noindent \textbf{2.} \emph{Coarse graining procedure}:

In a next step, we introduce a random variable $\kappa^\omega_N$ on $ \widetilde{\cD}^{\omega,\alpha}_N $, see~\eqref{eq:kappaN_Definition}, which encodes the position of the $\xi^\omega$-bad boxes. The corresponding coarse-grained events are denoted $\cD^\omega_{N,\kappa}$. This coarse-graining, which is similar in spirit to the one developed in~\cite{nitzschner2017solidification} (see also~\cite{chiarini2019entropic} and~\cite{nitzschner2018entropic}) is of small combinatorial complexity $\exp\{o(N^{d-2}) \}$ and reduces the problem to finding an asymptotic upper bound on the probability of $\cD^\omega_{N,\kappa}$ or $\{| \langle \bbX_N, \eta \rangle - \langle \mathscr{H}^\alpha_{\mathring{A}},\eta \rangle| > \Delta\} \cap \cD^{\omega}_{N,\kappa}$.

As a result of this coarse-graining procedure, we obtain a collection of boxes of size $2\widehat{L}_0$ located at points $\widehat{\cS}_N^\omega$. Within these boxes the ``interface'' of $\xi^\omega$-bad boxes of (the smaller) size $L_0$ has substantial presence. The larger boxes located at the points $\widehat{\cS}_N^\omega$ are ultimately used to define a segmentation $U_0$ of the porous interface we construct. The porous interface $\Sigma$ itself is made out of $L_0$-boxes  at a mutual distance bigger or equal to $(4K+1){L}_0$, located within a sparsified subcollection of $\widehat{L}_0$-boxes at $\widetilde{\cS}^\omega_N \subseteq \widehat{\cS}^\omega_N$. The relevant definitions of $U_0$ and $\Sigma$ appear in~\eqref{eq:DefOfSets}. The geometric picture is illustrated in Figure~\ref{fig:segmentation} below.

\begin{figure}[htbp]
\begin{center}
\includegraphics[scale=.55]{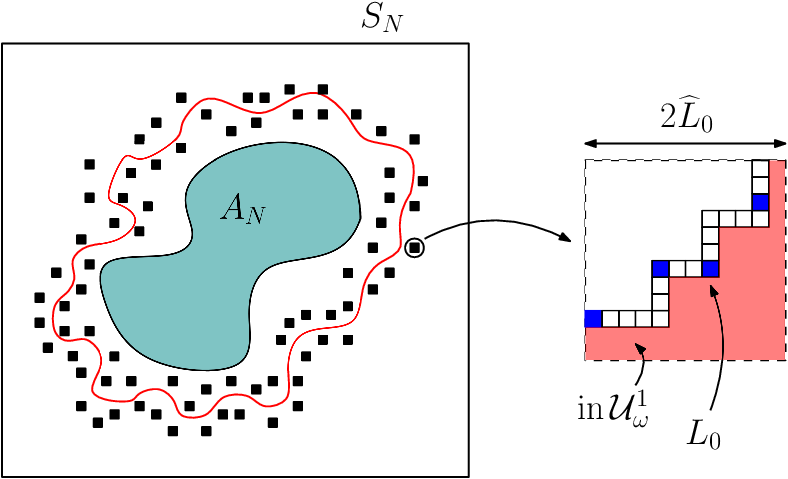}
\end{center}
\caption{Informal depiction of the geometry corresponding to the choice of $\kappa \in \cK_N$: Within each of the  boxes of size $2\widehat{L}_0$ located at $\widetilde{S}_N^\omega$ (picture on the left-hand side), one extracts a non-degenerate number of boxes of size $L_0$ (in blue, picture on the right-hand side). The smaller boxes are located on the interface of $\cU^1_\omega$.}
\label{fig:segmentation}
\end{figure}

\noindent \textbf{3.} \emph{Encoding into a Gaussian functional and application of quenched Gaussian bounds}:

In this step we employ the quenched Gaussian upper bounds of Section~\ref{sec:Quenched}. On the event $\cD^\omega_{N,\kappa}$, we consider the interface of $\xi^\omega$-bad boxes constructed in the previous step and use the Gaussian functionals $Z^\omega_m$ introduced in~\eqref{eq:ZmDefinition} to capture the event $\cD^\omega_{N,\kappa}$, see~\eqref{eq:badboxes}. In the proof of the asymptotic upper bound on $\bbP^\omega[\cD^\alpha_N]$, we can then directly employ Theorem~\ref{thm:ZmMainBound} and bring into play a bound in terms of $\frac{1}{N^{d-2}} \capa^\omega(\Sigma)$, see~\eqref{eq:UpperBoundBTIS_onlyDisc}. This term will subsequently be bounded using the quenched discrete solidification estimates and a homogenization result. 

The situation pertaining to $ \{ |\langle \bbX_N, \eta \rangle - \langle \mathscr{H}^\alpha_{\mathring{A}},\eta \rangle| > \Delta\}\cap \cD^\omega_{N,\kappa} $ requires more care: For a given number $\mu > 0$ and a fixed compact subset $A' \subseteq \mathring{A}$, we devise a dichotomy for $\kappa \in \cK_N$, leading to a set of ``good configurations'' $\cK^{\mu,\omega}_N$. Heuristically, the harmonic potential associated with the boxes present in a ``good configuration'' is close to that of $A_N'$, as measured by the Dirichlet form, see~\eqref{eq:GoodConfigurationsDef}. Asymptotic upper bounds are then obtained separately, uniformly over $\kappa \in \cK^{\mu,\omega}_N$ or $\kappa \in \cK_N \setminus \cK^{\mu,\omega}_N$, see~\eqref{eq:UpperBoundEntropicRepulsionSplitted}.

\begin{itemize}
\item In the case of ``bad configurations'' $\kappa \in \cK_N \setminus \cK^{\mu,\omega}_N$, we bound the probability of $\cD^\omega_{N,\kappa} \cap \{|\langle \bbX_N, \eta \rangle - \langle \mathscr{H}^\alpha_{\mathring{A}},\eta \rangle| > \Delta\}$ by $\bbP^\omega[\cD^\omega_{N,\kappa}]$, which we already treated. 
\item In the case of good configurations $\kappa \in \cK^{\mu,\omega}_N$, we first introduce an approximation of $\mathscr{H}^\alpha_{\mathring{A}}$ by $\cH^{\omega,\overline{\alpha}-a}_\Sigma$,  see~\eqref{eq:DiscreteShiftedHarmonicPotential}. The corresponding error term~\eqref{eq:DeltaOmegaDef} obtained when testing against $\eta$, will be bounded in terms of $\mu$, using the definition of good configurations and homogenization later. We can then use the functional $Z^\omega_{m,\beta,\rho}$ and Theorem~\ref{thm:ZmMainBound} to control control the probability of \textit{both} the occurrence of a deviation of $|\langle \bbX_N, \eta \rangle - \frac{1}{N^{d}} \langle \cH^{\omega,\overline{\alpha}-s}_\Sigma,\eta_N \rangle_{\bbZ^d}|$ and of the $\xi^\omega$-bad boxes, for any $s > 0$, see Proposition~\ref{prop:MainPropProofSec7}. The resulting bound in terms of $\frac{1}{N^{d-2}} \capa^\omega(\Sigma)$ captures the \textit{additional cost} of a deviation of the empirical field $\bbX_N$ from the approximate profile function $\cH^{\omega,\overline{\alpha}-a}_\Sigma$, in presence of the interface of $\xi^\omega$-bad boxes. 
\end{itemize}

\noindent \textbf{4.} \emph{Application of (quenched) discrete solidification results}:

Having brought into play bounds that involve the discrete capacity $\capa^\omega(\Sigma)$ as well as the approximation $\cH^{\omega,\overline{\alpha}-a}_\Sigma$, we now argue that the set $\Sigma$ can be viewed as a porous interface in the sense of Section~\ref{sec:Solidification} for the ``segmentation'' $U_0$ associated with $\kappa \in \cK_N$ (see~\eqref{eq:DefOfSets}). This allows us to use appropriate  bounds in the following ways:
\begin{itemize}
\item By applying the capacity lower bound~\eqref{eq:Capacity_Lower_Bound} of Corollary~\ref{cor:CapacityDirichletSolidification}, we uniformly bound the term $\frac{1}{N^{d-2}} \capa^\omega(\Sigma)$ from below by $\frac{1}{N^{d-2}} \capa^\omega(A_N')$, with $A' \subseteq \mathring{A}$ compact, see~\eqref{eq:SolidificationCapacityLowerBoundStep4}.  
\item We furthermore use the solidification-type bound of the difference between $\capa^\omega(\Sigma)$ and $\capa^\omega(A_N')$ by the Dirichlet energy of $h^\omega_{A_N'} - h^\omega_\Sigma$, see~\eqref{eq:Dirichlet_Form_Bound} of Corollary~\ref{cor:CapacityDirichletSolidification}, to obtain an ``additional cost term'' controlled by $\mu > 0$ for $\bbP^\omega[\cD^\omega_{N,\kappa}]$ in the case of bad configurations $\kappa \in \cK_N \setminus \cK^{\mu,\omega}_N$, see~\eqref{eq:BadConfigurationsUpperBound}. 
\item Finally, in the situation of good configurations $\kappa \in \cK^{\mu,\omega}_N$ we are able to control the approximation error between $ \langle \mathscr{H}^\alpha_{\mathring{A}} , \eta \rangle$ and $\tfrac{1}{N^d} \langle \cH^{\omega,\overline{\alpha}-a}_\Sigma, \eta_N \rangle_{\bbZ^d}$ by a constant depending on $\mu$ and the difference between $\langle \mathscr{H}^\alpha_{\mathring{A}} , \eta \rangle$ and $\tfrac{1}{N^d} \langle \cH^{\omega,\overline{\alpha}-a}_{A_N'}, \eta_N \rangle_{\bbZ^d}$, see~\eqref{eq:SplittingDelta_Hom_Solid_Part} and~\eqref{eq:ClosenessOfL2SolidificationPart}. 
\end{itemize}

\noindent \textbf{5.} \emph{Homogenization}:

In this last step, we finish the proof by using Corollary~\ref{cor:Capacity_convergence_NoKilling}. In a first simpler step, we use the convergence of $\frac{1}{N^{d-2}} \capa^\omega(A_N')$ to $\capa^\homo(A')$ as $N\rightarrow \infty$, see~\eqref{eq:HomogenizationCapacityApplSection7}. On the other hand, on account of the same Corollary, we see that the difference between $\langle \mathscr{H}^\alpha_{\mathring{A}} , \eta \rangle$ and $\tfrac{1}{N^d} \langle \cH^{\omega,\overline{\alpha}-a}_{A_N'}, \eta_N \rangle_{\bbZ^d}$ can be controlled in terms of $A'$ in the limit $N \rightarrow \infty$, see~\eqref{eq:ErrorTermFinalBound}. By taking $A' \uparrow \mathring{A}$, we finish the proof.

\begin{proof}[Theorems~\ref{thm:UpperBound} and~\ref{thm:EntropicRepulsion}] Both results are proved in a step-by-step manner, and rely on the same effective disconnection event and coarse-graining procedure. We choose $\omega$ in a subset of $\Omega_\lambda$ of full $P$-measure such that the statements of Corollary~\ref{cor:Capacity_convergence_NoKilling} and~ Proposition~\ref{prop:BadEventProposition} are fulfilled.
\vspace{\baselineskip}

\noindent\textbf{Step 1: Effective disconnection event.} We start by introducing a local density function $\widehat{\sigma}_\omega$ (see~\eqref{eq:sigmaHatDef} below) to extract in scale $\widehat{L}_0$ the interface of blocking $L_0$-boxes that we are interested in. To this end, consider the (random) subset
\begin{equation}
\cU_\omega^1 = \ \begin{minipage}{0.8\linewidth}
  the union of all $L_0$-boxes $B_z$ that are either contained in\\  $B(0,(M+1)N)^c$ or connected to an $L_0$-box in $B(0,(M+1)N)^c$ by
a path of $L_0$-boxes $B_{z_i}$, $0 \leq i \leq n$, all (except possibly the last one) $\psi^\omega$-good at levels $\delta < \gamma$ and $\xi^\omega$-good at level $a = \delta - \alpha$.
\end{minipage}
\end{equation}
One introduces the function
\begin{equation}
\label{eq:sigmaHatDef}
\widehat{\sigma}_\omega(x) = |\cU_\omega^1 \cap B(x,\widehat{L}_0) | / |B(x,\widehat{L}_0)|, \quad x \in \bbZ^d,
\end{equation}
together with the set $\widehat{\cS}^\omega_N$, that provides a ``segmentation'' of the interface of blocking $L_0$-boxes, namely 
\begin{equation}
\widehat{\cS}^\omega_N = \left\{x \in \widehat{\bbL}_0\, : \, \widehat{\sigma}_\omega(x) \in \left[ \tfrac{1}{4}, \tfrac{3}{4} \right] \right\}.
\end{equation} 
We can extract from $\widehat{\cS}^\omega_N$ a further (random) subset $\widetilde{\cS}^\omega_N$ such that
\begin{equation} 
\begin{minipage}{0.8\linewidth}$\widetilde{\cS}^\omega_N$ is a maximal subset of $\widehat{\cS}^\omega_N$ with the property that the $B(x,2\widehat{L}_0)$, $x \in \widetilde{\cS}^\omega_N$, are pairwise disjoint.
\end{minipage}
\end{equation}
We then introduce the ``bad'' event $\cB_N^\omega$ (similar to (3.19) in~\cite{nitzschner2018entropic}), which is defined as 
\begin{equation}
\label{eq:BadEventForDisconnection}
\cB_N^\omega = \bigcup_{e \in \{e_1,...,e_d \} } \left\{ \begin{minipage}{0.6\textwidth}
  there are at least $\varrho^\omega(L_0)(N_{L_0}/L_0)^{d-1}$ columns of $L_0$-boxes in the direction $e$ in $B(0,10(M+1)N)$ that contain a $\psi^\omega$-bad $L_0$-box at levels $\delta < \gamma$
\end{minipage}\right\},
\end{equation}
with $\varrho^\omega(L)$ as in Proposition~\ref{prop:BadEventProposition},  and $\{e_1,...,e_d \}$ is the canonical basis of $\bbR^d$. Let us argue that
\begin{equation}
\label{eq:SuperExponentialBound}
\lim_{N \rightarrow \infty} \frac{1}{N^{d-2}} \log \bbP^\omega[\cB_N^\omega] = -\infty.
\end{equation}
Indeed, by inserting the definitions of $L_0$ and $N_{L_0}$ (see~\eqref{eq:ScalesDefinition} and~\eqref{eq:NLDefinition}), one can show that for large $N$, it holds that
\begin{equation}
N_{L_0} = \tfrac{L_0^{d-1}}{\log L_0} \geq c \gamma_N^{-1} N \geq 10(M+1)N,
\end{equation}
see also Lemma 4.2 of~\cite{nitzschner2017solidification}. Therefore, the event $\cB^\omega_N$ is contained in the union over $e$ of the events under the probability in~\eqref{eq:SuperExpDecay} (with $L$ replaced by $L_0$), moreover $N_{L_0} \geq N$. The claim~\eqref{eq:SuperExponentialBound} follows immediately by~\eqref{eq:SuperExpDecay}. This motivates the introduction of the \emph{effective} disconnection event 
\begin{equation}
\label{eq:EffectiveDiscEvent}
\widetilde{\cD}^{\omega,\alpha}_N = \cD^\alpha_N \setminus \cB_N^\omega.
\end{equation}

\noindent\textbf{Step 2: Coarse graining.} We now define the coarse-graining procedure for $\widetilde{\cD}^{\omega,\alpha}_N$.

Similar to (4.39)--(4.41) of~\cite{nitzschner2017solidification} (see also below (3.21) of~\cite{nitzschner2018entropic} and (3.16) in~\cite{chiarini2019entropic}), one can show that
\begin{equation}\label{eq:ProjectionsExistOnEffDiscEvent}
  \begin{minipage}{0.85\textwidth}
    for large $N$, on $\widetilde{\cD}^{\omega,\alpha}_N$, for each $x \in \widetilde{\cS}^\omega_N$, one can find a collection $\widetilde{\cC}^\omega_x$ of points in $\bbL_0$ and $\widetilde{i}_x^\omega \in \{ 1,...,d \}$, such that the $L_0$-boxes $B_z$, $z\in \widetilde{\cC}^\omega_x$, intersect $B(x,\widehat{L}_0)$ and have $\widetilde{\pi}^\omega_x$-projection at mutual distance at least $\overline{K}L_0$, where $\widetilde{\pi}^\omega_x$ is the orthogonal projection on the set of points in $\bbZ^d$ with vanishing $\widetilde{i}^\omega_x$-coordinate. Moreover, $\widetilde{\cC}^\omega_x$ has cardinality $\big[\big( \tfrac{c'(\omega)}{K} \tfrac{\widehat{L}_0}{L_0} \big)^{d-1} \big]$ and for each $z \in \widetilde{\cC}^\omega_x$, $B_z$ is $\psi^\omega$-good at level $\gamma < \delta$ and $\xi^\omega$-bad at level $a = \delta- \alpha$.\end{minipage}
\end{equation}
 The existence of $\widetilde{\pi}_x^\omega$ and $\widetilde{\cC}^\omega_x$ for each $x \in \widetilde{\cS}^\omega_N$ on $\widetilde{\cD}^{\omega,\alpha}_N$ follows from the isoperimetric controls (A.3) -- (A.6), p. 480--481 of~\cite{deuschel1996surface}, paired with the fact that\\
   $\varrho^\omega(L_0) \big(\frac{N_{L_0}}{L_0}\big)^{d-1} / (\widehat{L}_0 / L_0)^{d-1} \rightarrow 0$ as $N \rightarrow \infty$ (see Lemma 4.2 of~\cite{nitzschner2017solidification}).

One then introduces a random variable $\kappa_N^\omega$ defined on $\widetilde{\cD}^{\omega,\alpha}_N$ with range $\cK_N$, 
\begin{equation}
\label{eq:kappaN_Definition}
\kappa_N^\omega = \big(\widehat{\cS}^\omega_N, \widetilde{\cS}^\omega_N, (\widetilde{\pi}^\omega_x, \widetilde{\cC}^\omega_x)_{x \in \widetilde{\cS}^\omega_N} \big),
\end{equation}
as in (4.41) of~\cite{nitzschner2017solidification}, see also (3.22) of~\cite{nitzschner2018entropic} (the range $\cK_N$ does not depend on $\omega$). A counting argument shows that 
\begin{equation}
\label{eq:smallCombComplexity}
|\cK_N| = \exp\{ o(N^{d-2}) \},
\end{equation} 
see (4.43) in~\cite{nitzschner2017solidification}. For large $N$, this gives a coarse-graining 
\begin{equation}\label{eq:coarseGraining}
\widetilde{\cD}^{\omega,\alpha}_N = \bigcup_{\kappa \in \cK_N} \cD^\omega_{N,\kappa}, \ \text{ where } \cD^\omega_{N,\kappa} = \widetilde{\cD}^{\omega,\alpha}_N \cap \{ \kappa_N^\omega = \kappa \}.
\end{equation}

For later use, we associate to a choice of $\kappa = (\widehat{\cS},\widetilde{\cS}, (\widetilde{\pi}_x,\widetilde{\cC}_x)_{x \in \widetilde{\cS}}) \in \cK_N$ the following sets
\begin{equation}
\label{eq:DefOfSets}
\begin{cases}
\cC & = \bigcup_{x \in \widetilde{\cS}} \widetilde{\cC}_x, \\
\Sigma & = \bigcup_{z \in \cC} B_z \subseteq \bbZ^d, \\
U_1 & = \text{ the unbounded component of }\bbZ^d \setminus   \bigcup_{x \in \widehat{\cS}} B\Big(x, \tfrac{1}{50d} \widehat{L}_0 \Big), \\
U_0 & = \bbZ^d \setminus U_1. 
\end{cases}
\end{equation}

The sets $U_0$ and $\Sigma$ correspond to a ``segmentation'' and ``porous interface'' in the sense of \eqref{eq:U0l_def} and \eqref{eq:ClassOfPorousInterfaces} (with the choice $\varepsilon = 10 \widehat{L}_0$), and are chosen similarly to (3.19) in~\cite{chiarini2019entropic}, with the important distinction that the present set-up is formulated entirely in terms of subsets of $\bbZ^d$. Furthermore, we define for a real number
\begin{equation}
\mu > 0
\end{equation}
and a compact subset $A' \subseteq \mathring{A}$ chosen later, a set of ``good'' configurations $\kappa \in \cK^{\mu,\omega}_N$, in which the set of boxes $\Sigma$ (depending on $\kappa$) has a harmonic potential $h^\omega_\Sigma$ close to $h^\omega_{A'_N}$ as measured by the Dirichlet form $\cE^\omega$, namely 
\begin{equation}
\label{eq:GoodConfigurationsDef}
\cK^{\mu,\omega}_N = \{ \kappa \in \cK_N \, : \, \cE^\omega(h^\omega_{\Sigma} - h^\omega_{A'_N}) \leq \mu N^{d-2} \}.
\end{equation} 
The configurations $\kappa \in \cK_N \setminus \cK^{\mu,\omega}_N$ form the ``bad'' configurations.


\vspace{\baselineskip}

\noindent\textbf{Step 3: Encoding into a Gaussian functional and application of quenched Gaussian bounds.} We  now apply the coarse-graining procedure to study the disconnection event and the intersection of the disconnection event with the entropic repulsion event. 

By the coarse graining~\eqref{eq:coarseGraining}, the super-exponential estimate~\eqref{eq:SuperExponentialBound} and the combinatorial bound~\eqref{eq:smallCombComplexity} we obtain 
\begin{equation}
\label{eq:UpperBoundOnlyDiscStep1}
\limsup_{N\to\infty} \frac{1}{N^{d-2}}\log \bbP^\omega[\cD^\alpha_N] \leq \limsup_{N\to\infty} \sup_{\kappa \in \cK_N} \frac{1}{N^{d-2}} \log \bbP^\omega[\cD^\omega_{N,\kappa}].
\end{equation}
Let $\kappa \in \cK_N$. On the event $\cD^\omega_{N,\kappa}$, in view of~\eqref{eq:ProjectionsExistOnEffDiscEvent}, we find that all $B_z$ with $z\in \cC$ are $\xi^\omega$-bad at level $a$ and at mutual distance $\geq \overline{K}L_0$, for large $N$. Thus, 
\begin{equation}\label{eq:badboxes}
  \cD^\omega_{N,\kappa} \subseteq \bigcap_{z \in \cC} \{\inf_{x\in D_z} \xi^{\omega,z}_x \leq -a \}.
\end{equation}
Moreover, $L_0 = o(N)$. For $\kappa \in \cK_N$, we also associate the set $\cM$ in~\eqref{eq:DefinitionCandM} to $\cC$ and observe that $\cap_{z \in \cC} \{\inf_{x\in D_z} \xi^{\omega,z}_x \leq -a \}\subseteq \{Z^\omega\leq -a\}$ (recall the definition of $Z^\omega$ in~\eqref{eq:ZomegaDefinition}). Combining~\eqref{eq:UpperBoundOnlyDiscStep1} and~\eqref{eq:badboxes}, we find that
\begin{equation}
\label{eq:BoundInTermsofZ}
\limsup_{N\to\infty} \frac{1}{N^{d-2}}\log \bbP^\omega[\cD^\alpha_N] \leq - \liminf_{N\to\infty} \inf_{\kappa \in \cK_N} \frac{1}{N^{d-2}} \log \bbP^\omega[Z^\omega \leq - a].
\end{equation}
 By the Borell-TIS inequality, we obtain 
\begin{equation}
\label{eq:Borell-TIS_onlyDisc}
\begin{split}
\bbP^\omega[Z^\omega \leq - a] & \leq \exp\Big\{-\frac{1}{2\sigma_\omega^2} \Big( a - | \bbE^\omega[Z^\omega]| \Big)_+^2 \Big\}, \qquad \text{where} \\
\sigma_\omega^2 & = \sup_{m \in \cM} \text{Var}^\omega(Z^\omega_m).
\end{split}
\end{equation}
In view of~\eqref{eq:VarZmBound} with the choice $\rho = \beta = 0$ and $\eta = 0$, and~\eqref{eq:ExpInfimumBound} of Theorem~\ref{thm:ZmMainBound} (with $C = \Sigma$), we obtain 
\begin{equation}
\label{eq:UpperBoundBTIS_onlyDisc}
\begin{split}
\limsup_{N\to\infty} &\frac{1}{N^{d-2}}\log \bbP^\omega[\cD^\alpha_N]  \\ & \leq - \liminf_{N\to\infty} \frac{1}{N^{d-2}} \inf_{\kappa \in \cK_N} \frac{1}{2} \Big\{ \Big(a - \tfrac{c_{12}}{K^{c_{13}}} \sqrt{\tfrac{|\cC|}{\capa^\omega(\Sigma)}} \Big)_+^2 \frac{\capa^\omega(\Sigma)}{1+\widetilde{U}(K)} \Big\},
\end{split}
\end{equation}
where $\widetilde{U}(K) = \limsup_L U(K,L)$, with $U(K,L)$ as in Theorem~\ref{thm:ZmMainBound}. In particular, $\widetilde{U}(K) \rightarrow 0$ as $K \rightarrow \infty$.
We now turn to the more intricate case involving the repulsion event. For $\mu > 0$, we obtain as in~\eqref{eq:UpperBoundOnlyDiscStep1} that
\begin{equation}
\label{eq:UpperBoundEntropicRepulsionSplitted}
\begin{split}
\limsup_{N\to\infty} & \frac{1}{N^{d-2}} \log \bbP^\omega[ | \langle \bbX_N, \eta \rangle - \langle \mathscr{H}^\alpha_{\mathring{A}},\eta \rangle| > \Delta \,;\, \cD^\alpha_N] \\
 & \leq \Big( \limsup_{N\to\infty} \sup_{\kappa \in \cK^{\mu,\omega}_N} \frac{1}{N^{d-2}} \log \bbP^\omega[ | \langle \bbX_N, \eta \rangle - \langle \mathscr{H}^\alpha_{\mathring{A}},\eta \rangle| > \Delta \, ; \, \cD^\omega_{N,\kappa}] \Big) \\
& \vee \Big( \limsup_{N\to\infty} \sup_{\kappa \in \cK_N \setminus \cK^{\mu,\omega}_N} \frac{1}{N^{d-2}} \log \bbP^\omega[\cD^\omega_{N,\kappa}] \Big)
\end{split}
\end{equation}

The second member of the maximum can be treated as in~\eqref{eq:badboxes}--\eqref{eq:Borell-TIS_onlyDisc}, yielding the right-hand side of~\eqref{eq:UpperBoundBTIS_onlyDisc} with $\cK_N$ replaced by $\cK_N \setminus \cK^{\mu,\omega}_N$ in the infimum. We therefore turn our attention to the first member of the maximum. To every $\kappa \in \cK^{\mu,\omega}_N$, we associate a ``discrete approximation'' to $\mathscr{H}^\alpha_{\mathring{A}}$, given by the function $\cH^{\omega,\overline{\alpha}-a}_\Sigma$, where we set
\begin{equation}
\label{eq:DiscreteShiftedHarmonicPotential}
\cH^{\omega,\theta}_{F}(x) = -(\overline{\alpha}-\theta) h^\omega_F(x), \qquad x \in \bbZ^d,\,\theta\in \bbR,\, F \subseteq \bbZ^d \text{ finite}.
\end{equation}
(For us, $\theta = \overline{\alpha}- a$ will be a natural choice). With the quantity
\begin{equation}
\label{eq:DeltaOmegaDef}
\overline{\Delta}^\omega_N(\mu) = \sup_{\kappa \in \cK^{\mu,\omega}_N} | \langle \mathscr{H}^\alpha_{\mathring{A}} , \eta \rangle - \tfrac{1}{N^d} \langle \cH^{\omega,\overline{\alpha}-a}_\Sigma, \eta_N \rangle_{\bbZ^d}|,
\end{equation}
where $\eta_N = \eta(\cdot/N)$, one obtains (for $\kappa \in \cK^{\mu,\omega}_N$)
\begin{equation}
\begin{split}
\label{eq:UpperBoundDiscretizedForBTIS}
\bbP^\omega &[ | \langle \bbX_N, \eta \rangle - \langle \mathscr{H}^\alpha_{\mathring{A}},\eta \rangle| > \Delta  \,;\, \cD^{\omega}_{N,\kappa}] \\
& \leq \bbP^\omega[ | \langle \bbX_N, \eta \rangle -  \tfrac{1}{N^d}\langle \cH^{\omega,\overline{\alpha}-a}_\Sigma, \eta_N \rangle_{\bbZ^d} | > \Delta - \overline{\Delta}^\omega_N(\mu) \,;\, \cD^{\omega}_{N,\kappa}].
\end{split}
\end{equation}
For large $N$, we will eventually show that, $\overline{\Delta}^\omega_N(\mu)$ can be controlled by $\mu$, using $\kappa \in \cK^{\mu,\omega}_N$ and the $L^2$-convergence of $h^\omega_{A'_N}$ to $\mathscr{h}_{A'}$ as a result of homogenization. The right-hand side of~\eqref{eq:UpperBoundDiscretizedForBTIS} will now be treated using the Borell-TIS inequality. 

\begin{prop}
\label{prop:MainPropProofSec7}
Let $s, \widetilde{\Delta},\mu > 0$. For small $\beta \in (0,c_{16}(\eta))$, there exists a function $K \mapsto \varepsilon(K)$, dependent on $\beta$ and $\eta$, with $\varepsilon(K) \rightarrow 1$ as $K \rightarrow \infty$, and such that
\begin{equation}
\label{eq:GaussianBoundBTIS_with_repulsionClaim}
\begin{split}
&\limsup_{N \rightarrow \infty}  \sup_{\kappa \in \cK^{\mu,\omega}_N} \frac{1}{N^{d-2}}\log \bbP^\omega\big[ \{ | \langle \bbX_N, \eta \rangle -  \tfrac{1}{N^d}\langle  \cH^{\omega,\overline{\alpha}- s}_\Sigma, \eta_N \rangle_{\bbZ^d} | > \widetilde{\Delta} \}  \\ & \quad\qquad\cap \bigcap_{z \in \cC} \{\inf_{x \in D_z} \xi^{\omega,z}_x \leq -s \} \big] \\
& \quad \leq - \liminf_{N \rightarrow \infty}  \inf_{\kappa \in \cK^{\mu,\omega}_N}   \frac{1}{2} \Big(s + \beta \widetilde{\Delta} - \tfrac{c}{K^{c_{13}}} \sqrt{\tfrac{|\cC|}{\capa^\omega(\Sigma)}} \Big)_+^2 \tfrac{N^{2-d} \capa^\omega(\Sigma)}{\varepsilon(K) + \beta^2c_{10}(\eta)N^{2-d}\capa^\omega(\Sigma) } 
\end{split}
\end{equation}
\end{prop}
\begin{proof}
By replacing $-\eta$ by $\eta$, the claim will follow if we show that
\begin{equation}
\label{eq:Step1ProofBTIS_with_repulsion}
\begin{split}
&\limsup_{N \rightarrow \infty}  \sup_{\kappa \in \cK^{\mu,\omega}_N} \frac{1}{N^{d-2}} \log \bbP^\omega\big[ \{ \langle \bbX_N, \eta \rangle -  \tfrac{1}{N^d}\langle  \cH^{\omega,\overline{\alpha}- s}_\Sigma, \eta_N \rangle_{\bbZ^d}  \geq \widetilde{\Delta} \} \\ &\quad\qquad \cap \bigcap_{z \in \cC} \{\inf_{x \in D_z} \xi^{\omega,z}_x \leq -s \} \big] \\
& \quad \leq - \liminf_{N \rightarrow \infty}  \inf_{\kappa \in \cK^{\mu,\omega}_N} \frac{1}{2} \Big(s + \beta \widetilde{\Delta} - \tfrac{c}{K^{c_{13}}} \sqrt{\tfrac{|\cC|}{\capa^\omega(\Sigma)}} \Big)_+^2 \tfrac{N^{2-d} \capa^\omega(\Sigma)}{\varepsilon(K) + \beta^2c_{10}(\eta)N^{2-d}\capa^\omega(\Sigma) }.
\end{split}
\end{equation}
Note that the event under the probability in~\eqref{eq:Step1ProofBTIS_with_repulsion} is contained in the event $\{\inf_{m \in \cM} Z^\omega_{m,\beta,\rho_N} \leq -s - \beta \widetilde{\Delta}\}$, where $\rho_N = \beta \frac{1}{N^d} \langle \eta_N, h^\omega_\Sigma \rangle_{\bbZ^d}$, and $Z^\omega_{m,\beta,\rho}$ is defined in~\eqref{eq:ZmDefinition}. We used the fact that $ \sup_{\kappa \in \cK_N} |\rho_N| \leq \beta c(\eta)$ (using that $h^\omega_\Sigma \in [0,1]$), so we can choose $\beta < \frac{1}{c(\eta)} = c_{16}(\eta)$ to ensure that $|\rho_N| < 1$. Note furthermore that
\begin{equation}
\label{eq:InfOfShiftedGaussianField}
\bbE^\omega\Big[ \inf_{m \in \cM} Z^\omega_{m,\beta,\rho_N} \Big] = (1+\rho_N)\bbE^\omega[Z^\omega].
\end{equation}
By the Borell-TIS inequality we obtain 
\begin{equation}
\label{eq:UpperBoundBTISforZmbetarho}
\begin{split}
\bbP^\omega&\Big[\inf_{m \in \cM} Z^\omega_{m,\beta,\rho_N}  \leq -s - \beta\widetilde{\Delta} \Big] \\
& \leq \exp\Big\{-\frac{1}{2\sigma_{\omega,\beta,\rho_N}^2} \Big( s + \beta \widetilde{\Delta} - | \bbE^\omega[\inf_{m \in \cM} Z^\omega_{m,\beta,\rho_N}]| \Big)_+^2 \Big\}, \\
& \text{where}\ \ \sigma_{\omega,\beta,\rho_N}^2  = \sup_{m \in \cM} \text{Var}^\omega(Z^\omega_{m,\beta,\rho_N}).
\end{split}
\end{equation}
We can therefore apply Theorem~\ref{thm:ZmMainBound} to obtain 
\begin{equation}
\label{eq:VarianceUniformBound}
\begin{split}
&\varlimsup_{N\to\infty}  \sup_{\kappa \in \cK^{\mu,\omega}_N} \sup_{m \in \cM} N^{d-2} \sigma^2_{m,\beta,\rho_N}  \leq \varlimsup_{N\to\infty} \sup_{\kappa \in \cK^{\mu,\omega}_N} \frac{N^{d-2}}{\capa^\omega(\Sigma)} \widetilde{\alpha}_{K,L_0,\beta,\rho_N} + \beta^2c_{10}(\eta), \\
& \widetilde{\alpha}_{K,L_0,\beta,\rho_N}  \leq 1 +4U(K,L_0)+ c_{11}(\eta) \beta \left(U(K,L_0)+  \tfrac{(KL_0)^d}{N^d}\big(|\cC| + \capa^\omega(\Sigma)\big)\right).
\end{split}
\end{equation}
A simple counting argument (see for instance (3.31) and (3.32) of~\cite{chiarini2019entropic}) shows that 
\begin{equation}
\label{eq:SmallnessVolumeNumber}
|\cC| \leq c\sqrt{\gamma_N} \frac{N^{d-2}}{\log N}, \quad \text{ implying }\quad  \lim_{N \rightarrow \infty} \sup_{\kappa \in \cK_N} \frac{L_0^d |\cC|}{N^d} = 0.
\end{equation}
Also, for every $\kappa \in \cK_N$, it holds that $\capa^\omega(\Sigma) \leq c N^{d-2}$, so
\begin{equation}
\limsup_{N \rightarrow \infty} \sup_{\kappa \in \cK_N}\frac{L_0^d\capa^\omega(\Sigma)}{N^d} \leq \limsup_{N \rightarrow \infty} c\frac{(\gamma_N^{-1} N \log N)^{\frac{d}{d-1}}}{N^2} = 0.
\end{equation}
We are thus able to define 
\begin{equation}
\varepsilon(K) =1+ 4\widetilde{U}(K)+ \beta c_{11}(\eta) \widetilde{U}(K),
\end{equation}  
with $\widetilde{U}(K) = \limsup_L U(K,L)$ and $\varepsilon(K) \rightarrow 1$ as $K \rightarrow \infty$. The claim~\eqref{eq:Step1ProofBTIS_with_repulsion} now follows by inserting~\eqref{eq:ExpInfimumBound},~\eqref{eq:InfOfShiftedGaussianField} and~\eqref{eq:VarianceUniformBound} into~\eqref{eq:UpperBoundBTISforZmbetarho}. 
\end{proof}

\noindent\textbf{Step 4: Application of the (quenched) discrete solidification results.}
We now apply the capacity bounds from Section~\ref{sec:Solidification} to control uniformly in $\kappa \in \cK_N$ the capacity of $\Sigma$ (see~\eqref{eq:DefOfSets}). This will allow us to control the right-hand sides of~\eqref{eq:UpperBoundBTIS_onlyDisc} and~\eqref{eq:Step1ProofBTIS_with_repulsion}. We also give a bound for the capacity of $\Sigma$, uniformly in $\kappa \in \cK_N \setminus \cK^{\mu,\omega}_N$. Furthermore, we perform a major step in showing the smallness of the approximation error term $\overline{\Delta}^\omega_N(\mu)$ (see~\eqref{eq:DeltaOmegaDef}), using that $\kappa \in \cK^{\mu,\omega}_N$.
\vspace{\baselineskip}

We now show that $\Sigma$ is a ``porous interface'' for the segmentation $U_0$ (both associated with $\kappa \in \cK_N$ or $\cK^{\mu,\omega}_N$) in the sense of~\eqref{eq:ClassOfPorousInterfaces} and~\eqref{eq:U0l_def}. Let us therefore consider a closed Lipschitz domain $A' \subseteq \mathring{A}$ with non-empty interior. For large $N$ and all $\kappa \in \cK_N$, on $\cD^\omega_{N,\kappa}$ we have that $\{z \in A_N \, : \, d_\infty(\{z\},\partial A_N) \geq \widehat{L}_0+L_0 +1 \}$ does not intersect $U_1$, see also Lemma 4.3 in~\cite{nitzschner2017solidification} for a similar argument. We therefore find a sequence $\ell_{\ast,N} = \ell_\ast(A_N,A'_N) \geq 0$ such that, for large $N$ and $\kappa \in \cK_N$, 
\begin{equation}
d_\infty(A'_N, U_1) \geq 2^{\ell_{\ast,N}}, \qquad \text{(so $U_{0} \in \mathcal{U}_{\ell_\ast,A'_N}$)},
\end{equation}
and moreover $2^{\ell_{\ast,N}} \geq c_{17}N$. A simple projection argument together with the (quenched) estimates~\eqref{eq:QuenchedBoxCapacityEstimate} shows that for large $N$, $\kappa \in \cK_N$, and $x \in \widetilde{\cS}$, we have $\capa^\omega \big(\bigcup_{z \in \widetilde{\cC}_x} B_z \big) \geq c(K) \widehat{L}_0^{d-2}$. We can then conclude that 
\begin{equation}
P^\omega_x[H_\Sigma < \tau_{10 \widehat{L}_0}] \geq c(K), \qquad \text{for all }x \in \partial U_0. 
\end{equation}
In other words, we see that $\Sigma \in \cS^\omega_{U_0, 10\widehat{L}_0, c(K)}$ and furthermore $a_N = \frac{10\widehat{L}_0}{2^{\ell_{\ast,N}}} \rightarrow 0$ as $N \rightarrow \infty$. This brings us into the framework of Corollary~\ref{cor:CapacityDirichletSolidification}. We obtain that for fixed $K$ large enough, one has 
\begin{equation}
\label{eq:SolidificationCapacityLowerBoundStep4}
\liminf_{N \rightarrow \infty} \inf_{\kappa \in \cK_N} \frac{1}{N^{d-2}} \capa^\omega(\Sigma) \geq \liminf_{N\to\infty} \frac{1}{N^{d-2}} \capa^\omega(A'_N).
\end{equation}
Since $\mathring{A}' \neq \emptyset$, we see that by~\eqref{eq:QuenchedBoxCapacityEstimate}, $\capa^\omega(A'_N) \geq cN^{d-2}$, and combining~\eqref{eq:SolidificationCapacityLowerBoundStep4} and~\eqref{eq:SmallnessVolumeNumber}, we obtain
\begin{equation}
\label{eq:ExpectationTermBTISVanishes}
\lim_{N \rightarrow \infty} \sup_{\kappa \in \cK_N} \frac{| \cC|}{\capa^\omega(\Sigma)} = 0. 
\end{equation}
Combining~\eqref{eq:SolidificationCapacityLowerBoundStep4},~\eqref{eq:ExpectationTermBTISVanishes} and~\eqref{eq:UpperBoundBTIS_onlyDisc} we find that (upon taking $\liminf_{K}$)
\begin{equation}
\label{eq:UpperBoundOnlyDisc_beforeHom}
\limsup_{N\to\infty} \frac{1}{N^{d-2}} \log \bbP^\omega[\cD^\alpha_N] \leq -  \frac{a^2}{2} \liminf_{N\to \infty} \frac{\capa^\omega(A'_N)}{N^{d-2}}. 
\end{equation}
This concludes the discussion of the solidification estimates for upper bounds on the probability of the disconnection event $\cD^\alpha_N$ alone. We now turn to the more delicate case involving the intersection of the deviation event $\{|\langle \bbX_N, \eta \rangle - \langle \mathscr{H}^\alpha_{\mathring{A}}, \eta \rangle | > \Delta \}$ with $\cD^\alpha_N$. Recall that in~\eqref{eq:UpperBoundEntropicRepulsionSplitted} we devised a splitting in ``good'' configurations $\kappa \in  \cK^{\mu,\omega}_N$ and ``bad'' configurations $\kappa \in \cK_N \setminus  \cK^{\mu,\omega}_N$, for which we derive separate bounds. We start by discussing the second member of the maximum in~\eqref{eq:UpperBoundEntropicRepulsionSplitted}, corresponding to the ``bad'' configurations. By~\eqref{eq:badboxes} and the arguments leading up to~\eqref{eq:UpperBoundBTIS_onlyDisc}, we find that (recall that $\widetilde{U}(K) \geq 0$),
\begin{equation}
\label{eq:BadConfigurationsUpperBound}
\begin{split}
& \varlimsup_N  \sup_{\kappa \in \cK_N \setminus \cK^{\mu,\omega}_N}  \frac{1}{N^{d-2}} \log \bbP^\omega[\cD^\omega_{N,\kappa}] \\
 & \leq - \varliminf_N \frac{1}{N^{d-2}} \inf_{\kappa \in \cK_N \setminus \cK^{\mu,\omega}_N} \frac{1}{2} \Big\{ \Big(a - \tfrac{c}{K^{c_{13}}} \sqrt{\tfrac{|\cC|}{\capa^\omega(\Sigma)}} \Big)_+^2 \frac{\capa^\omega(\Sigma)}{1+\widetilde{U}(K)} \Big\} \\
& \stackrel{\eqref{eq:ExpectationTermBTISVanishes}}{\leq} - \varliminf_N \frac{a^2}{2}  \frac{ \capa^\omega(A_N')}{N^{d-2}(1+ \widetilde{U}(K))} - \varliminf_N \frac{a^2}{2} \frac{\inf_{\kappa \in \cK_N \setminus \cK^{\mu,\omega}_N} (\capa^\omega(\Sigma) - \capa^\omega(A_N'))}{N^{d-2}(1+ \widetilde{U}(K))} \\
& \stackrel{\eqref{eq:Capacity_and_Dirichlet},\eqref{eq:Dirichlet_Form_Bound}}{\leq} - \varliminf_N \frac{a^2}{2} \frac{\capa^\omega(A_N')}{N^{d-2}(1+\widetilde{U}(K))} - \varliminf_N \frac{a^2}{2} \frac{\inf_{\kappa \in \cK_N \setminus \cK^{\mu,\omega}_N}  \cE^\omega(h^\omega_{A_N'} - h^\omega_{\Sigma})}{N^{d-2}(1+\widetilde{U}(K))} \\
& \stackrel{\eqref{eq:GoodConfigurationsDef}}{\leq} - \varliminf_N \frac{a^2}{2} \frac{\capa^\omega(A_N')}{N^{d-2}(1+\widetilde{U}(K))} -  \frac{a^2}{2(1+\widetilde{U}(K))} \mu .
\end{split}
\end{equation}
This concludes the application of the solidification estimate to the second member of the maximum in~\eqref{eq:UpperBoundEntropicRepulsionSplitted}. We now discuss the first member in this maximum.

As a first step, we want to argue that $\overline{\Delta}^\omega_N(\mu)$ from~\eqref{eq:DeltaOmegaDef} can be controlled by a term which becomes small due to homogenization and a constant depending on $\mu$ and $\eta$. More specifically, it holds that
\begin{equation}
\begin{split}
\label{eq:SplittingDelta_Hom_Solid_Part}
\overline{\Delta}^\omega_N(\mu) & \leq  | \langle \mathscr{H}^\alpha_{\mathring{A}},\eta \rangle - \tfrac{1}{N^d} \langle \cH^{\omega,\overline{\alpha}-a}_{A_N'},\eta_N\rangle_{\bbZ^d} | \\
& +  \tfrac{1}{N^d} \sup_{\kappa \in  \cK^{\mu,\omega}_N} | \langle \cH^{\omega,\overline{\alpha}-a}_\Sigma,\eta_N \rangle_{\bbZ^d} - \langle\cH^{\omega,\overline{\alpha}-a}_{A_N'},\eta_N \rangle_{\bbZ^d} |,
\end{split}
\end{equation}
and we now show that 
\begin{equation}
\label{eq:ClosenessOfL2SolidificationPart}
\limsup_{N \rightarrow \infty} \frac{1}{N^d} \sup_{\kappa \in  \cK^{\mu,\omega}_N} \left\vert \langle \cH^{\omega,\overline{\alpha}-a}_\Sigma,\eta_N \rangle_{\bbZ^d} - \langle\cH^{\omega,\overline{\alpha}-a}_{A_N'},\eta_N \rangle_{\bbZ^d} \right\vert \leq \overline{\alpha} c_{18}(\eta)\sqrt{\mu}.
\end{equation}
To that end, we insert the definition of $\cH^{\omega,\overline{\alpha}-a}_\Sigma$ and $\cH^{\omega,\overline{\alpha}-a}_{A_N'}$ (see~\eqref{eq:DiscreteShiftedHarmonicPotential}) and use~\eqref{eq:EnergyDirichletFormBound} (applied with $h = \eta_N^+$ and $h = \eta_N^-$) to obtain that for $\kappa \in \cK^{\mu,\omega}_N$
\begin{equation}
\label{eq:ClosenessOfL2SolidificationPartProof}
\begin{split}
\Big\vert \langle \cH^{\omega,\overline{\alpha}-a}_\Sigma,\eta_N \rangle_{\bbZ^d} &- \langle \cH^{\omega,\overline{\alpha}-a}_{A_N'},\eta_N \rangle_{\bbZ^d} \Big\vert  \leq 2\overline{\alpha} (W^\omega(\eta^+_N) \vee W^\omega(\eta^-_N))^{\frac{1}{2}} \cE^\omega(h^\omega_{\Sigma} - h^\omega_{A_N'})^{\frac{1}{2}} \\
& \stackrel{\eqref{eq:GoodConfigurationsDef}}{\leq} 2\overline{\alpha} \|\eta\|_\infty (c'(\eta) N^{d+2})^{\frac{1}{2}} (\mu N^{d-2})^{\frac{1}{2}} = \overline{\alpha} c_{18}(\eta)\sqrt{\mu}N^d,
\end{split}
\end{equation}
where we used the quenched Green function estimate~\eqref{eq:QuenchedGFEstimate} and the standard bound 
\begin{equation}
\sum_{x,y \in B(0,N)} \frac{1}{|x-y|^{d-2} \vee 1} \leq cN^2
\end{equation}
(and the support of $\eta^{\pm}_N$ is contained in a ball $B(0,c'(\eta)N)$). Upon taking the supremum over $\kappa \in \cK^{\mu,\omega}_N$ and dividing by $N^d$, the claim~\eqref{eq:ClosenessOfL2SolidificationPart} follows readily from~\eqref{eq:ClosenessOfL2SolidificationPartProof}. 

As a second step, we apply the quenched capacity estimate of Corollary~\ref{cor:CapacityDirichletSolidification} also to the right-hand side of~\eqref{eq:GaussianBoundBTIS_with_repulsionClaim} (noting that the function $(t,x) \mapsto \frac{x}{t+\beta^2 c_{10}(\eta)x}$ for $t,x > 0$ is increasing in $x$ and jointly continuous in $t$ and $x$) to obtain that for every $s, \widetilde{\Delta}, \mu > 0$, $\beta \in (0,c_{16}(\eta))$, one has
\begin{equation}
\label{eq:UpperBoundGoodCaseAlmostFinished}
\begin{split}
\limsup_{N \rightarrow \infty} \sup_{\kappa \in \cK^{\mu,\omega}_N} &\frac{1}{N^{d-2}}\log \bbP^\omega\big[ \{  | \langle \bbX_N, \eta \rangle -   \tfrac{1}{N^d}\langle  \cH^{\omega,\overline{\alpha}- s}_\Sigma, \eta_N \rangle_{\bbZ^d} | > \widetilde{\Delta} \} \\ &\qquad\cap \bigcap_{z \in \cC} \{\inf_{x \in D_z} \xi^{\omega,z}_x \leq -s \} \big] \\
&  \leq -\frac{1}{2} \big(s + \beta \widetilde{\Delta}  \big)^2 \liminf_{N \rightarrow \infty}   \frac{ N^{2-d} \capa^\omega(A_N') }{\varepsilon(K) + \beta^2c_{10}(\eta)N^{2-d} \capa^\omega(A_N') }.
\end{split}
\end{equation}
To finish the argument, we will need to pair this result with the smallness of $\overline{\Delta}^\omega_N(\mu)$ from~\eqref{eq:ClosenessOfL2SolidificationPart} for which the $L^2$-convergence of $h^\omega_{A_N'}$ to $\mathscr{h}_{A'}$ will be pivotal. This argument will be given in the final step.

\vspace{\baselineskip}
\noindent\textbf{Step 5: Homogenization.} In this step we finalize the proof of both Theorem~\ref{thm:UpperBound} and~\ref{thm:EntropicRepulsion}. We start with the upper bound on $\cD^\alpha_N$. By combining Corollary~\ref{cor:Capacity_convergence_NoKilling} with~\eqref{eq:UpperBoundOnlyDisc_beforeHom}, we find that 
\begin{equation}
\label{eq:HomogenizationCapacityApplSection7}
\limsup_{N \rightarrow \infty} \frac{1}{N^{d-2}} \log \bbP^\omega[\cD^\alpha_N] \leq - \frac{a^2}{2} \capa^\homo(A'). 
\end{equation}
The claim~\eqref{eq:QuenchedDiscUpperBound} now follows upon letting $a \uparrow \overline{\alpha} - \alpha$ and $A' \uparrow \mathring{A}$, using the convergence of capacities, see Theorem 2.1.1 in~\cite{fukushima2010dirichlet}.

We now turn to the more intricate situation of Theorem~\ref{thm:EntropicRepulsion}. As a first step, we see that 
\begin{equation}
\label{eq:ErrorTermFinalBound}
\begin{split}
\limsup_{N \rightarrow \infty}  |\langle \mathscr{H}^\alpha_{\mathring{A}}, \eta \rangle &- \tfrac{1}{N^d} \langle \cH^{\omega,\overline{\alpha}-a}_{A_N'},\eta_N \rangle_{\bbZ^d}|  \leq |\langle\mathscr{H}^\alpha_{\mathring{A}}, \eta \rangle - \langle\mathscr{H}^{\overline{\alpha}-a}_{A'}, \eta \rangle | \\
& + \limsup_{N \rightarrow \infty} |\langle\mathscr{H}^{\overline{\alpha}-a}_{A'}, \eta \rangle - \tfrac{1}{N^d} \langle \cH^{\omega,\overline{\alpha}-a}_{A_N'},\eta_N \rangle_{\bbZ^d}| \\
& \stackrel{\eqref{eq:convergence_minimizer}}{\leq} c(\eta) (\overline{\alpha}-\alpha - a) + |\langle \eta, \mathscr{H}^\alpha_{\mathring{A}} - \mathscr{H}^\alpha_{A'} \rangle|.
\end{split}
\end{equation}
We choose $\mu > 0$ such that $\overline{\alpha}c_{18}(\eta)\sqrt{\mu} < \frac{\Delta}{6}$, $a$ sufficiently close to $\overline{\alpha}-\alpha$ such that $c(\eta)(\overline{\alpha}-\alpha-a) < \frac{\Delta}{6}$ and $A' \subseteq \mathring{A}$ such that $|\langle \eta, \mathscr{H}^\alpha_{\mathring{A}} - \mathscr{H}^\alpha_{A'} \rangle| < \frac{\Delta}{6}$. Upon combining the previous display with~\eqref{eq:SplittingDelta_Hom_Solid_Part} and~\eqref{eq:ClosenessOfL2SolidificationPart} we therefore obtain for large enough $N$
\begin{equation}
\label{eq:ErrorSmallerHalfDelta}
\overline{\Delta}^\omega_N(\mu) \leq \overline{\alpha} c_{18}(\eta)\sqrt{\mu} + c(\eta) (\overline{\alpha}-\alpha - a) + |\langle \eta, \mathscr{H}^\alpha_{\mathring{A}} - \mathscr{H}^\alpha_{A'} \rangle| \leq \frac{\Delta}{2},
\end{equation}
By~\eqref{eq:UpperBoundDiscretizedForBTIS} and~\eqref{eq:UpperBoundGoodCaseAlmostFinished}, we thus obtain for every $\mu > 0$ small enough and $s = a$ that 
\begin{equation}
\label{eq:GoodConfigurationsUpperBound}
\begin{split}
\limsup_{N \rightarrow \infty} & \sup_{\kappa \in \cK^{\mu,\omega}_N} \frac{1}{N^{d-2}}\log\bbP^\omega[ | \langle \bbX_N, \eta \rangle - \langle \mathscr{H}^\alpha_{\mathring{A}},\eta \rangle| > \Delta \,;\, \cD^{\omega}_{N,\kappa}] \\
&  \stackrel{\eqref{eq:ErrorSmallerHalfDelta}}{\leq} \limsup_{N \rightarrow \infty} \sup_{\kappa \in \cK^{\mu,\omega}_N}  \frac{1}{N^{d-2}}\log \bbP^\omega\big[ \{ | \langle \bbX_N, \eta \rangle -  \tfrac{1}{N^d}\langle  \cH^{\omega,\overline{\alpha}- a}_\Sigma, \eta_N \rangle_{\bbZ^d} | > \tfrac{\Delta}{2} \} \\
& \qquad \qquad \qquad \cap \bigcap_{z \in \cC} \{\inf_{x \in D_z} \xi^{\omega,z}_x \leq -a \} \big] \\
&  \stackrel{\eqref{eq:UpperBoundGoodCaseAlmostFinished}}{\leq} -\frac{1}{2} \big(a + \beta \tfrac{\Delta}{2}  \big)^2 \liminf_{N \rightarrow \infty}   \frac{ N^{2-d} \capa^\omega(A_N') }{\varepsilon(K) + \beta^2c_{10}(\eta)N^{2-d} \capa^\omega(A_N') }.
\end{split}
\end{equation}

We thus find by inserting the results~\eqref{eq:BadConfigurationsUpperBound} and~\eqref{eq:GoodConfigurationsUpperBound} into~\eqref{eq:UpperBoundEntropicRepulsionSplitted} that, for $\beta\in (0,c_{16}(\eta))$ and $\mu > 0$ small enough, it holds that (upon taking $K \rightarrow \infty$)
\begin{equation}
\label{eq:AlmostFinalUpperBound}
\begin{split}
& \limsup_{N \rightarrow \infty}  \frac{1}{N^{d-2}} \log \bbP^\omega[ | \langle \bbX_N, \eta \rangle - \langle \mathscr{H}^\alpha_{\mathring{A}},\eta \rangle| > \Delta \,;\, \cD^\alpha_N] \\
 & \leq \Big( -\frac{1}{2} \big(a + \beta \tfrac{\Delta}{2}  \big)^2 \liminf_{N \rightarrow \infty}   \frac{ N^{2-d} \capa^\omega(A_N') }{1 + \beta^2c_{10}(\eta)N^{2-d} \capa^\omega(A_N')}  \Big) \\
& \vee \Big(-\frac{a^2}{2}\liminf_{N \rightarrow \infty}  \frac{\capa^\omega(A_N')}{N^{d-2}} -  \frac{a^2}{2} \mu \Big) \\
& = - \frac{1}{2} \Big( \big(a + \beta \tfrac{\Delta}{2}  \big)^2 \frac{  \capa^\homo(A')}{1+\beta^2 c_{10}(\eta)  \capa^\homo(A')} \Big) \wedge \Big(a^2( \capa^\homo(A') + \mu  )\Big),
\end{split}
\end{equation}
using again Corollary~\ref{cor:Capacity_convergence_NoKilling} in the final step. We let $a \uparrow \overline{\alpha}- \alpha$, $A' \uparrow \mathring{A}$ and see that the left-hand side in~\eqref{eq:AlmostFinalUpperBound} can be bounded from above by
\begin{equation}
\label{eq:AlmostFinalUpperBound2}
\begin{split}
 - \frac{1}{2} \Big( \big(\overline{\alpha}-\alpha + \beta \tfrac{\Delta}{2}  \big)^2 \frac{  \capa^\homo(\mathring{A})}{1+\beta^2 c_{10}(\eta)  \capa^\homo(\mathring{A})} \Big) \wedge \Big((\overline{\alpha}-\alpha)^2( \capa^\homo(\mathring{A}) + \mu  ) \Big).
\end{split}
\end{equation}
 Setting $\mu = \beta^2$ and choosing $\beta \in (0,c_{16}(\eta))$ small enough, the expression in~\eqref{eq:AlmostFinalUpperBound2} can be made smaller than $-\frac{1}{2} (\overline{\alpha}-\alpha)^2\capa^\homo(\mathring{A}) - c(\Delta,\alpha,\eta)$. 
\end{proof}
\textbf{Acknowledgements.} The authors wish to thank Alain-Sol Sznitman for helpful discussions and valuable comments at various stages of this project. The authors also wish to thank Mathias Sch\"affner for helpful discussions about the $\Gamma$-convergence result in~\cite{neukamm2017stochastic}. Moreover, the
authors thank the anonymous referees for their careful revision of the article and for valuable suggestions.

\appendix

\section{Proof of Theorem~\ref{thm:Resonance}}

In this appendix we prove Theorem~\ref{thm:Resonance}. The proof proceeds as the one in~\cite{nitzschner2017solidification} using ``$I$-families'' with the respective modifications. 
We introduce these $I$-families in the discrete case and sketch the proof, focusing on the part where Proposition~\ref{prop_multiscale_descent} is applied. Let us first prove the maximality property~\eqref{eq:maximality}. Note that for $U_0 \in \cU_{\ell_\ast, A}$, $x \in A$, one has $U_0 - x \in \cU_{\ell_\ast, \{ 0\}}$ and $\text{Res}(U_0 - x,I,J,L, \ell_\ast) =\text{Res}(U_0 ,I,J,L, \ell_\ast) - x$. For a given $\omega \in \Omega_\lambda$, one has
\begin{equation}
P^\omega_x[H_{\text{Res}} = \infty] = P^{\tau_x\omega}_0[H_{\text{Res} - x} = \infty] \leq \Phi_{J,I,L},
\end{equation}
so the maximality follows. We turn to (discrete) $I$-families. Let $I, J \geq 1$ and $L \geq L(J)$ be fixed, $\ell_\ast \geq 0$ $(I,J,L)$-compatible, $\omega \in \Omega_\lambda$ and $U_0 \in \cU_{\ell_\ast, \{ 0 \} }$ (we refer to~\eqref{eq:LJdef} and~\eqref{eq:CompatibilityCond} for the respective definitions). Recall also the definition of $\ell_{\text{min}}(\cdot)$ from~\eqref{eq:MinScaleDefinition}. An \emph{$I$-family} consists of stopping times $(S_i)_{i = 0}^I$, a random finite subset $\cL \subseteq (J+1) L \bbN \cap [\ell_{\min}((200J)^{-1}) +LJ, \infty)$, and integer valued random variables $\widehat{\ell}_i$, $1 \leq i \leq I$, such that

\begin{equation}\label{eq:Definition_I_Family}
\left\{ \begin{array}{rl}
{\rm (i)} & \mbox{$0 \leq S_0 \leq S_1 \leq ... \leq S_I$, $P^\omega_0$-a.s. finite stopping times,}
\\[1.5ex]
{\rm (ii)} & \mbox{$\cL$ is an $\cF_{S_0}$-measurable finite subset of} \\
 &\mbox{ $\cL \subseteq (J+1) L \bbN \cap [\ell_{\min}((200J)^{-1}) +LJ, \infty)$, and $|\cL| \geq I$},
\\[1.5ex]
{\rm (iii)} &\mbox{$\widehat{\ell}_i$, $1 \leq i \leq I$ are $\cF_{S_i}$-measurable, pairwise distinct and $\cL$-valued},
\\[1.5ex]
{\rm (iv)} & \mbox{$P^\omega_0$-a.s., $\sigma_{\widehat{\ell}_i }(X_{S_i}) \in [\tfrac{1}{2} - \tfrac{1}{2^{\ell_{\min}((200J)^{-1}) } }, \tfrac{1}{2} + \tfrac{1}{2^{\ell_{\min}((200J)^{-1}) } }]$, $1 \leq i \leq I.$}  
\end{array}\right.
\end{equation}

The ``canonical'' $I$-family as defined in (2.12) of~\cite{nitzschner2017solidification} also exists in the discrete case, if we replace the conditions $\sigma_\ell(X_{S_i}) = \tfrac{1}{2}$ by  $\sigma_{\ell_i }(X_{S_i}) \in [\tfrac{1}{2} - \tfrac{1}{2^{\ell_{\min}((200J)^{-1}) } }, \tfrac{1}{2} + \tfrac{1}{2^{\ell_{\min}((200J)^{-1}) } }]$. Given a general $I$-family as above, we also define for $1 \leq i \leq I$ the stopping times
\begin{equation}
T_i = \inf \{ s \geq S_i \, : \, |X_s - X_{S_i}|_\infty \geq 2 \cdot 2^{\widehat{\ell}_i} \},
\end{equation}
and ``intermediate labels'' and ``labels''
\begin{equation}
\cL_{\text{int}} = \{ \ell - jL \, : \, \ell \in \cL, 1 \leq j \leq J \}, \qquad \cL_\ast = \cL \cup \cL_{\text{int}}.
\end{equation}
Finally, we will need for $1 \leq k \leq J$ the $(\cL_\ast,k)$-resonance set 
\begin{equation}
\text{Res}_{(\cL_\ast,k)} = \bigg\{ x \in \bbZ^d \, : \, \sum_{\ell \in \cL_\ast }  \mathbbm{1}_{ \{ \widetilde{\sigma}_\ell(x) \in [\widetilde{\alpha}, 1 - \widetilde{\alpha} ] \} }  \geq k \bigg\},
\end{equation}
and the quantity
\begin{equation}
\Gamma^{\omega,(J)}_{k}(I) = \sup P^\omega_0[\inf \{ s \geq S_0 \, : \, X_s \in \text{Res}_{(\cL_\ast,k)}  \} > \max_{1 \leq i \leq I} T_i ],
\end{equation}
for $1 \leq k \leq J$, $I \geq 1$ (with the supremum over all $I$-families) and $\Gamma^{\omega, (J)}_k(I) =1$ whenever $I \leq 0$. The following discrete analogue of Lemma 2.2 in~\cite{nitzschner2017solidification} is the main ingredient of the proof of Theorem~\ref{thm:Resonance}. 
\begin{lem}
\label{thm:ResonanceLemma}
For $\omega \in \Omega_\lambda$, one has
\begin{equation}
\label{eq:ResonanceLemma_part1}
\Gamma^{\omega,(J)}_1(I) = 0, \qquad \text{for all } I \geq 1 \\
\end{equation}
and for $1 \leq k < J$, $I \geq 1$, $\Delta = \lfloor \sqrt{I} \rfloor$,
\begin{equation}
\label{eq:MainLemmaResonance}
\Gamma^{\omega, (J)}_{k+1}(I) \leq (1 - c_7(J))^{\sqrt{I} - 1} + I^{1 + \tfrac{k-1}{2}} \Gamma^{\omega, (J)}_k( \Delta - k + 1).
\end{equation}
\end{lem}
\begin{proof}
We only sketch the proof. The first part follows by noting that $P^\omega_0$-a.s., $\sigma_{\widehat{\ell}_1}(X_{S_1}) \in [\tfrac{1}{2} - \tfrac{1}{2^{\ell_{\min}((200J)^{-1}) } }, \tfrac{1}{2} + \tfrac{1}{2^{\ell_{\min}((200J)^{-1}) } }]$, and since $J \geq 1$, $2^{-\ell_{\min}(1/(200J)) } \leq \tfrac{1}{1600}$, hence $U_1$ and $U_0$ have relative volumes in $B(X_{S_1}, 2^{\widehat{\ell}_1})$ at least $\tfrac{799}{1600}$ and at most $\tfrac{801}{1600}$, or in other words, $\widetilde{\sigma}_{\widehat{\ell}_1}(X_{S_1}) \in [\widetilde{\alpha}, 1- \widetilde{\alpha}]$ and $\Gamma^{\omega,(J)}_1(I) = 0$ is immediate since $\inf \{s \geq S_0\,: X_s \in \text{Res}_{(\cL_\ast,1)} \} \leq S_1 \leq \max_{1 \leq i \leq I} T_i$, $P^\omega_0$-a.s., proving~\eqref{eq:ResonanceLemma_part1}.  

We set $m_\Delta = \lfloor \tfrac{I - 1}{\Delta} \rfloor$, such that $i_\Delta = 1 + m_\Delta \Delta \leq I < 1 + (m_\Delta + 1)\Delta$. For $I \geq 2$, we have 
\begin{equation}
 P^\omega_0[\inf \{ s \geq S_0 :  X_s \in \text{Res}_{(\cL_\ast,k+1)}  \} > \max_{1 \leq i \leq I} T_i ] \leq a_1^\omega + a_2^\omega,
\end{equation}
where 
\begin{align}
a_1^\omega & = P^\omega_0[T_i < S_{i + \Delta} \text{ for all } 1 \leq i \leq I - \Delta , \inf \{ s \geq S_0  :  X_s \in \text{Res}_{(\cL_\ast,k+1)}  \} > \max_{1 \leq i \leq I} T_i ], \\
a_2^\omega & = P^\omega_0[T_i \geq S_{i + \Delta} \text{ for some } 1 \leq i \leq I - \Delta , \inf \{ s \geq S_0 \, : \, X_s \in \text{Res}_{(\cL_\ast,k+1)}  \} > \max_{1 \leq i \leq I} T_i ].
\end{align}
For $a_2^\omega$, one has the bound
\begin{equation}
\label{eq:MainBounda_2}
a_2^\omega \leq I^{1 + \frac{k-1}{2}} \Gamma^{\omega, (J) }_k(\Delta - k + 1).
\end{equation}
Its proof proceeds exactly as in the Brownian case, see (2.29)--(2.33) of~\cite{nitzschner2017solidification} and is thus omitted. For the bound on $a_1^\omega$, note that one has
\begin{equation}
\label{eq:MainBounda_1_step1}
\begin{split}
a_1^\omega & \leq E^\omega_0\big[S_1 < T_1  < ... < S_{i_\Delta} < T_{i_\Delta} < \inf \{s \geq S_0\, : \, X_s \in \text{Res}_{(\cL_\ast,k+1)} \}] \\
& \leq E^\omega_0\big[S_1 < T_1  < ... < S_{i_\Delta} < \inf \{s \geq S_0\, : \, X_s \in \text{Res}_{(\cL_\ast,k+1)} \}, \\
& \widetilde{P}^\omega_{X_{S_{i_\Delta}}}[ \inf \{s \geq 0\, : \, |\widetilde{X}_s - \widetilde{X}_0|_\infty \geq 2 \cdot 2^{\widehat{\ell}_{i_\Delta} } \} < \inf \{ s \geq 0\, : \, \widetilde{X}_s \in \text{Res}_{(\cL_\ast,k+1)}  \} ] \big],
\end{split}
\end{equation}
having used the strong Markov property at time $S_{i_\Delta}$ for the second bound, and where $(\widetilde{X}_\cdot)$ denotes the canonical process which behaves as a random walk among conductances $\omega$, starting from $X_{S_{i_\Delta}}$ under $\widetilde{P}^\omega_{X_{S_{i_\Delta}}}$, and $\cL_\ast$ and $\widehat{\ell}_{i_\Delta}$ are not integrated under $\widetilde{P}^\omega_{X_{S_{i_\Delta}}}$. 
 
We use now Proposition~\ref{prop_multiscale_descent}: Choose $x = X_{S_{i_\Delta}}$ and recall that $\widehat{\ell}_{i_\Delta} - LJ \geq \ell_{\text{min}}((200J)^{-1})$ (by~\eqref{eq:Definition_I_Family}, (ii)) as well as $|\sigma_{\ell_{i_\Delta}}(X_{S_{i_\Delta}}) -\frac{1}{2}| \leq 2^{-\ell_{\text{min}}((200J)^{-1}) }$, see~\eqref{eq:Definition_I_Family}, (iv).  Since $k + 1 \leq J$, we have on an event that has $\widetilde{P}_{X_{S_{i_\Delta}}}^\omega$-probability bigger or equal to $c_7(J)$ that $\widetilde{X}_{\gamma_J} \in \text{Res}_{(\cL_\ast,k+1)}$, but $\sup \{ |\widetilde{X}_s - \widetilde{X}_0| \, : \, 0 \leq s \leq \gamma_J  \} \leq \tfrac{3}{2} \cdot 2^{\widehat{\ell}_{i_\Delta}}$, so on this event, the event within $\widetilde{P}^\omega_{X_{S_{i_\Delta}}}$ in the last line of~\eqref{eq:MainBounda_1_step1} \textit{does not} occur. We obtain that the expression in the last line of~\eqref{eq:MainBounda_1_step1} is bounded above by
\begin{equation}
\label{eq:MainBounda_1_final}
\begin{split}
a_1^\omega & \leq E^\omega_0\big[S_1 < T_1  < ... < T_{i_\Delta - \Delta} < \inf \{s \geq 0\, : \, X_s \in \text{Res}_{(\cL_\ast,k+1)} \}] (1-c_7(J)) \\
& \stackrel{\text{(induction)}}{\leq} (1 - c_7(J))^{m_\Delta + 1} \leq (1 - c_7(J))^{\sqrt{I} - 1},
\end{split}
\end{equation}
using in the last step that $m_\Delta > \frac{I - 1}{\sqrt{I}} \geq \sqrt{I} -1$. By combining the bounds~\eqref{eq:MainBounda_2} and~\eqref{eq:MainBounda_1_final}, we obtain:
\begin{equation}
 P^\omega_0[\inf \{ s \geq S_0  : X_s \in \text{Res}_{(\cL_\ast,k+1)}  \} > \max_{1 \leq i \leq I} T_i ] \leq (1 - c_7(J))^{\sqrt{I} - 1} + I^{1 + \tfrac{k-1}{2}} \Gamma^{\omega, (J)}_k( \Delta - k + 1)
\end{equation}
Finally, we take the supremum over all $I$-families, which yields~\eqref{eq:MainLemmaResonance} in the case where $I \geq 2$. For $I = 1$, the claim of~\eqref{eq:MainLemmaResonance} is true, since the right-hand side is bigger or equal to $1$. 
\end{proof}

We now turn to the proof of~\eqref{eq:maximality} of Theorem~\ref{thm:Resonance}. Similar to (2.34) of~\cite{nitzschner2017solidification}, we set
\begin{equation}
\widetilde{\Gamma}^{(J)}_k(I) = \begin{cases}
\sup \limits_{\ell_\ast}  \sup \limits_{U_0 \in \cU_{\ell_\ast, \{0 \}} }\sup \limits_{\omega \in \Omega_\lambda} \Gamma^{\omega, (J)}_k(I), & \text{ for }1 \leq k \leq J \text{ and } I \geq 1, \\
1, & \text{ for } 1\leq k \leq J \text{ and } I \leq 0.
\end{cases}
\end{equation}
where in the first case, the supremum in $\ell_\ast$ is over all $(I,J,L)$-compatible $\ell_\ast \geq 0$. Using the ``canonical'' $I$-family, one has that
\begin{equation}
\Phi_{J,I,L} \leq \widetilde{\Gamma}^{(J)}_J(I).
\end{equation}
Using~\eqref{eq:ResonanceLemma_part1} and~\eqref{eq:MainLemmaResonance} of Lemma~\ref{thm:ResonanceLemma}, we receive upon taking the suprema over $\omega \in \Omega_\lambda$, $U_0 \in \cU_{\ell_\ast,\{ 0\}}$ and $(I,J,L)$-compatible $\ell_\ast \geq 0$:
\begin{equation}
\begin{split}
&\widetilde{\Gamma}^{(J)}_1(I) = 0, \text{ for } I \geq 1, \\
& \widetilde{\Gamma}^{(J)}_{k+1}(I) \leq (1 - c_7(J))^{\sqrt{I} - 1} + I^{1 + \tfrac{k-1}{2}} \widetilde{\Gamma}^{(J)}_k( \Delta - k + 1), \text{ for }1 \leq k \leq J, I \geq 1.
\end{split}
\end{equation} 
The proof of~\eqref{eq:maximality} now follows by induction on $k$, in exactly the same way as (2.37)--(2.38) of~\cite{nitzschner2017solidification}.

%

%
%



\end{document}